\documentclass{article}
\usepackage{amsthm}
\usepackage{amsmath}
\usepackage{amssymb}
\usepackage{tikz}
\usepackage{booktabs}
\usepackage{setspace} 
\usepackage{enumerate}
\usepackage{geometry}
\usepackage{graphicx}
\usepackage{xcolor}
\usepackage{ifthen}
\usepackage[hidelinks]{hyperref}
\allowdisplaybreaks

\usepackage{float}
\floatstyle{ruled}
\newfloat{algorithm}{h}{loa}
\floatname{algorithm}{Algorithm}
\usepackage{algpseudocode}
\usepackage{tabularx}
\usepackage{booktabs}
\usepackage{array}
\usepackage{enumitem}
\usepackage{accents}

\newtheorem{prop}{Proposition}
\newtheorem{rem}[prop]{Remark}
\newtheorem{thm}[prop]{Theorem}
\newtheorem{defn}[prop]{Definition}
\newtheorem{lem}[prop]{Lemma}
\newtheorem{cor}[prop]{Corollary}
\newtheorem{ass}[prop]{Assumption}
\newtheorem{cond}[prop]{Condition}

\usepackage{tabularx}
\usepackage{booktabs}
\usepackage{array}

\newcommand{\extra}[1]{}%{#1}

\newcommand{\Vf}{{\hat{V}}}
\newcommand{\Hf}{{\hat{H}}}
\newcommand{\Ff}{{\hat{F}}}

\newcommand{\Ac}{\mathcal{A}}

\newcommand{\pfix}{{\hat{p}}}

\newcommand{\weight}{\omega}
\newcommand{\bias}{\beta}

\newcommand{\Gc}{\mathcal{G}}
\newcommand{\Rc}{\mathcal{R}}

\newcommand{\Yc}{\mathcal{Y}}

\newcommand{\Vc}{\mathcal{V}}
\newcommand{\Wc}{\mathcal{W}}

\newcommand{\Lc}{\mathcal{L}}

\newcommand{\C}{ \mathcal{C} }
\newcommand{\B}{ \mathcal{B} }
\newcommand{\G}{ \mathcal{G} }

\newcommand{\Nc}{ \mathcal{N} }

\makeatletter
\newcommand{\vc}[2][r]{%
  \gdef\@VORNE{1}
  \left(\hskip-\arraycolsep%
    \begin{array}{#1}\vekSp@lten{#2}\end{array}%
  \hskip-\arraycolsep\right)}
\def\vekSp@lten#1{\xvekSp@lten#1;vekL@stLine;}
\def\vekL@stLine{vekL@stLine}
\def\xvekSp@lten#1;{\def\temp{#1}%
  \ifx\temp\vekL@stLine
  \else
    \ifnum\@VORNE=1\gdef\@VORNE{0}
    \else\@arraycr\fi%
    #1%
    \expandafter\xvekSp@lten
  \fi}
\makeatother

\newcommand{\R}{\mathbb{R}}
\newcommand{\N}{\mathbb{N}}

\newcommand{\wrt}{\:\mathrm{d}}

\newcommand{\beq}{\begin{equation}}
\newcommand{\eeq}{\end{equation}}

\newcommand{\embed}{\hookrightarrow}

\newcommand{\Bc}{\mathcal{B}}
\newcommand{\Cc}{\mathcal{C}}
\newcommand{\Nt}{\mathcal{N}_\theta}
\newcommand{\Nta}{\mathcal{N}_{\theta,\alpha}}
\newcommand{\vtil}{{\widetilde{V}}}

\newcommand{\Ktil}{{\widetilde{K}}}
\newcommand{\qbar}{{\overline{q}}}

\newcommand{\ztil}{{\widetilde{z}}}
\newcommand{\zttil}{\accentset{\approx}{z}}
\newcommand{\uhddot}{{\ddot{u}^h}}
\newcommand{\uhdot}{{\dot{u}^h}}
\newcommand{\uzddot}{{\ddot{u}^z}}
\newcommand{\uzdot}{{\dot{u}^z}}

\newcommand{\WWs}{{\Wc,\Wc^*}}
\newcommand{\wws}{{W,W^*}}
\newcommand{\hhs}{{H,H^*}}
\newcommand{\vsv}{{V^*,V}}

\newcommand{\ltn}{{L^2(\Omega)}}
\newcommand{\ltlt}{{L^2(0,T;L^2(\Omega))}}
\newcommand{\ltv}{{L^2(0,T;V)}}
\newcommand{\hto}{{H^2(\Omega)\cap H^1_0(\Omega)}}
\newcommand{\htn}{{H^2(\Omega)}}

\newcommand{\tom}{((0,T)\times\Omega)}
\newcommand{\om}{(\Omega)}
\newcommand{\ti}{(0,T)}

\newcommand{\Dinv}{(-\Delta)^{-1}(-\Delta+\text{Id})^{-1}}

\newcommand{\compt}{\hookrightarrow\mathrel{\mspace{-15mu}}\rightarrow
}

\newcommand{\la}{\langle}
\newcommand{\ra}{\rangle}
\newcommand{\commA}[1]{\textcolor{black}{{#1}}}

\newcommand{\commB}[1]{\textcolor{black}{{#1}}}
\newcommand{\commE}[1]{\textcolor{black}{{#1}}}
\title{Learning-informed parameter identification in nonlinear time-dependent PDEs\footnote{CA and MH acknowledge funding by the Austrian Research Promotion Agency (FFG) (Project number 881561). }}

\author{
Christian Aarset\footnote{\href{mailto:c.aarset@math.uni-goettingen.de}{c.aarset@math.uni-goettingen.de} -- Revision work  was done in part at the University of Göttingen.}
\and 
Martin Holler\footnote{\href{mailto:martin.holler@uni-graz.at}{martin.holler@uni-graz.at}. MH is a member NAWI Graz (\url{https://www.nawigraz.at}) and BioTechMed Graz (\url{https://biotechmedgraz.at}).} 
\and 
Tram Thi Ngoc Nguyen \footnote{ \href{mailto:nguyen@mps.mpg.de}{nguyen@mps.mpg.de} -- Revision work was done in part at the Max-Planck Institute for Solar System Research with support of the DFG through grant 432680300 - SFB 1456.}
\and\\
Institute of Mathematics and Scientific Computing\\University of Graz, Austria
}

\date{}
\begin{document}
\maketitle
\begin{abstract}

\end{abstract}
We introduce and analyze a method of learning-informed parameter identification for partial differential equations (PDEs) in an all-at-once framework. The underlying PDE model is formulated in a rather general setting with three unknowns: physical parameter, state and nonlinearity. Inspired by advances in machine learning, we approximate the nonlinearity via a neural network, whose parameters are learned from measurement data. The later is assumed to be given as noisy observations of the unknown state, and both the state and the physical parameters are identified simultaneously with the parameters of the neural network. Moreover, diverging from the classical approach, the proposed all-at-once setting avoids constructing the parameter-to-state map by explicitly handling the state as additional variable. The practical feasibility of the proposed method is confirmed with experiments using two different algorithmic settings: A function-space algorithm based on analytic adjoints as well as a purely discretized setting using standard machine learning algorithms.\\


%\textcolor{magenta}{references suggest by referees: \cite{18ChristofMeyerWaltherClason,20CuiHePang,22DongHintermuellerPapafitsorosVoelkner}}

\section{Introduction}\label{sec:intro}
%\commA{Restructured this entire section and moved subsection \emph{The learning problem} to section 2. The original version of Section 1 is in the commented text, feel free to uncomment everything between end of section 1 and beginning of section 2 to see the original version.}

We study the problem of determining an unknown nonlinearity $f$ from data in a parameter-dependent dynamical system
\begin{equation}\label{ori-eq-1}
\begin{aligned}
& \dot{u}=F(\lambda,u)+f(\alpha,u) \qquad &&\text{in } (0,T) \times \Omega\\
& u(0)=u_0 &&\text{on }\Omega .
\end{aligned}
\end{equation}
Here, the state $u$ is a function on a finite time interval $(0,T)$ and a bounded Lipschitz domain $\Omega$, \commA{and $\dot{u}$ denotes the first order time derivative}. 
In \eqref{ori-eq-1}, both $F, f$ are  nonlinear \commB{Nemytskii} operators in $\lambda, \alpha, u$\commB{; these Nemytskii operators are induced by  nonlinear, time-dependent functions 
$
[F(\lambda,u)](t): = F(t,\lambda,u(t))$ and $ [f (\alpha,u)](t,x): = f (\alpha,u(t,x)),
$
where we consistently abuse notation in this manner throughout the paper; see also Lemmas \ref{lem-nemytskii-F}, \ref{lem-nemytskii-N}}. %\todo{I suggest to remove the following sentence as it interrupts too much the (much more important) explanation in the surrounding text.}
We assume that $F$ was specified beforehand from an underlying physical model,
that the terms $\lambda$, $u_0$ are physical parameters (with $\lambda=\lambda(x)$ depending only on space), and that $\alpha  $ is a finite dimensional parameter arising in the nonlinearity. Furthermore, the model \eqref{ori-eq-1} is equipped with Dirichlet or Neumann boundary conditions.

Some examples of partial differential equations (PDEs) of the from \eqref{ori-eq-1} are diffusion models $\dot{u}=\Delta u + f(\alpha,u)$ with a nonlinear reaction term $f(\alpha,u)$ as follows \cite{Pazy}: %\commA{added dependence on $\alpha$}
\begin{itemize}
\item $f(\alpha,u)= -\alpha u(1 - u)$: \emph{Fisher equation} in heat and mass transfer, combustion theory.
\item $f(\alpha,u)= –\alpha u(1 - u)(\alpha - u), 0<\alpha<1$: \emph{Fitzhugh–Nagumo equation} in population genetics. %\commA{changed $b$ in the constraint $0<b<1$ to $\alpha$, ok? Yes, thanks.}
\item $f(\alpha,u)=-u/(1+\alpha_1 u+\alpha_2 u^2)$, $\alpha = (\alpha_1,\alpha_2)$, $ \alpha_1>0, \alpha_1^2<4\alpha_2 $: Enzyme kinetics.
\item $f(\alpha,u) = f(u) =-u|u|^p$, $ p\geq 1 $: Irreversible isothermal reaction, temperature in radiating bodies.
\end{itemize}
The underlying assumption of this work is that in some cases, the nonlinearity $f$ is unknown due to simplifications or inaccuracies in the modeling process or due to undiscovered physical laws. In such situations, our goal is to learn $f$ from data. In order to realize this in practice, we need to use a parametric representation. For this, we choose neural networks, which have become widely used in computer science and applied mathematics due to their excellent representation properties\commA{, see for instance \cite{Hornik} for the classical universal approximation theorem, \cite{lu2021deep} for recent results indicating superior approximation properties of neural networks with particular activations (potentially at the cost of stability) and \cite{devore2021neural,Kutyniok21math_deep_learning} for general, recent overviews on the topic.} Learning the nonlinearity $f$ thus reduces to identifying parameters $\theta$ of a neural network $\Nt$ such that $\Nt \approx f$, rendering the problem of learning a nonlinearity to be a parameter identification problem of a particular form. 

\commE{For the majority of this paper, the nonlinearity $f$ will therefore not appear directly; instead, $f$ will consistently be replaced by its neural network representation $\Nt$, and our focus will be on showing the properties of $\Nt$, rather than those of $f$.}

A main point in our approach, which is motivated from feasibility for applications, is that learning the nonlinearity must be achieved only via indirect, noisy measurements of the state $y^\delta \approx Mu$ with $M$ a \commA{linear} measurement operator. More precisely, we assume to have $K$ different measurements 
\begin{align}\label{ori-eq-2}
\commA{y^k = Mu^k \qquad k=1\commB{,\ldots,\,} K}
\end{align}
of different states $u^k$ available, where the different states correspond to solutions of the system \eqref{ori-eq-1} with different, unknown parameters $(\lambda^k,\alpha^k,u_0^k)$, but the same, unknown nonlinearity $f$ which is assumed to be part of the ground truth model.
%In this context, it is important to note that the parameters $\lambda^k,u^k_0,\alpha^k$ and the state $u^k$ differ between different measurements $k$, but the unknown nonlinearity $f$ is assumed to be part of the ground truth model and hence stays the same.
\commA{The simplest form of $M$ is a full observation over time and space of the states, i.e. $M=\text{Id}$ as in e.g. (theoretical) population genetics. In other contexts, $M$ could be discrete observations at time instances of $u$, i.e.  $Mu=\commB{(u(t_i,\cdot))_{i=1}^{n_T}}, t_i\in(0,T)$, as in material science \cite{Pedretscher}, system biology \cite{Boiger} (see also Corollary \ref{prop-adjoints-dis}), or Fourier transform as in MRI acquisition \cite{BenningEhrhardt}, etc. In most cases, $M$ is linear, as is assumed here.}

Our approach to address this problem is to use \emph{an all-at-once formulation} that avoids constructing the parameter-to-state map (see for instance \cite{kaltenbacher17}). That is, we aim to identify all unknowns by solving a minimization problem of the form
\begin{equation}\label{min-prob} 
\min_{\substack{
(\lambda^k,\alpha^k,u_0^k,u^k)_k \commB{\subset X \times \R^m \times U_0\times\Vc} \\
\theta \commB{\in \Theta}
}} \sum_{k=1}^K \| \Gc(\lambda^k,\alpha^k,u_0^k,u^k,\theta) - (0,0,y^k) \|^2_{\commB{{\Wc\times H}\times\Yc}} + \Rc_1(\lambda^k,\alpha^k,u_0^k,u^k)  +  \Rc_2(\theta),
\end{equation}
\commB{where we refer to Section \ref{sec:setting} for details on the function spaces involved.}
Here, $\Gc$ is a forward operator that incorporates the PDE model, the initial conditions and the measurement operator via
\begin{align*}
\Gc(\lambda,\alpha,u_0,u,\theta)=(\dot{u}-F(\lambda,u)-\Nc_\theta(\alpha,u),u(0)-u_0,Mu),
\end{align*}
and $\Rc_1$, $\Rc_2$ are suitable regularization functionals.

Once a particular parameter $\hat{\theta}$ such that $\mathcal{N}_{\hat{\theta}}$ accurately approximates $f$ in \eqref{ori-eq-1} is learned, one can use the learning informed model in other parameter identification problems by solving
\begin{equation} \label{par-id}
\min_{
(\lambda,\alpha,u_0, u)\commB{\in X \times \R^m \times U_0\times\Vc}
} \| \Gc(\lambda,\alpha,u_0,u,\hat{\theta}) - (0,0,y) \|^2_{\commB{{\Wc\times H}\times\Yc}} + \Rc_1(\lambda,\alpha,u_0,u)
\end{equation}
for a new measured datum $y \approx Mu$.

\textbf{Existing research towards learning PDEs and all-at-one identification.}
Exploring governing PDEs from data is an active topic in many areas of science and engineering. With advances in computational power and mathematical tools, there have been numerous recent studies on data-driven discovery of hidden physical laws. One novel technique is to construct a rich dictionary of possible functions, such as polynomials, derivatives etc., and to then use sparse regression to determine candidates that most accurately represent the data \cite{Schaeffer, Brunton,Rudy}. This sparse identification approach yields a completely explicit form of the differential equation, but requires an abundant library of basic functions specified beforehand. %Recently, with the development of the deep learning techniques, different neural network architectures are employed to learn the PDE model. \commA{Rassi} assumes underlying model is known except for a few scalar coefficients. These coefficients are learned jointly with the solution map approximated by a deep neural network. %This approach uses less data than the sparse regression approach, but the explicit form of the model must be known.
In this work, we take the viewpoint that PDEs are constructed from principal physical laws.
As it preserves the underlying equation and learns only some unknown components of the models, e.g. $f$ in \eqref{ori-eq-1}, our suggested approach is capable of refining approximate models by staying more faithful to the underlying physics.

Besides the machine learning part, the model itself may contain unknown physical parameters belonging to some function space. This means that if the nonlinearity $f$ is successfully learned, one can insert it into the model. One thus has a learning-informed PDE, and can then proceed via a classical parameter identification. The latter problem was studied in \cite{DongHintermuellerPapafitsoros20} for stationary PDEs, where $f$ is learned from training pairs $(u,f(u))$. This paper emphasizes analysis of the error propagating from the neural network-based approximation of $f$ to the parameter-to-state map and the reconstructed parameter.

%\todo{TN: first say that a second setting is considered in \cite{DongHintermuellerPapafitsoros20}, which uses training pairs $(u,f(u))$. Then say the following:}
In reality, one does not have direct access to the true state $u$, but only partial or coarse observations of $u$ under some noise contamination. This factor affects the creation of training data pairs $(u,f(u))$ with $f(u)=\dot{u}-F(u)$ for the process of learning $f$, e.g in \cite{DongHintermuellerPapafitsoros20}. Indeed, with a coarse measurement of $u$, for instance $u\in L^2((0,T)\times\Omega)$, one cannot evaluate $\dot{u}$, nor terms such as $\Delta u$ that may appear in $F(u)$. %To overcome this, one can apply some filter to smooth the data $y=u$. %, which can be understood as an inverse problem in order to better control the smoothing level.
Moreover, with discrete \commB{observations}, e.g. a snap\commA{shot} $y=\commB{(u(t_i,\cdot))_{i=1}^{n_T}}, t_i\in(0,T)$, one is unable to compute $\dot{u}$ for the training data. 

For this reason, we propose an all-at-once approach to identify the nonlinearity $f$, state $u$ and physical parameter simultaneously. In comparison to \cite{DongHintermuellerPapafitsoros20}, our approach bypasses the training process for $f$, and accounts for discrete data measurements. The all-at-once formulation avoids constructing the parameter-to-state map, which is nonlinear and often involves restrictive conditions \cite{HaAs01,aao16,kaltenbacher17,KKV14b,Nguyen}. Additionally, we here consider time-dependent PDE models.

%\todo{TN: mention also Kunisch work: recent preprint, independent of this work}
For discovering nonlinearities in evolutionary PDEs, the work in \cite{CourtKunisch} suggests an optimal control problem for nonlinearities expressed in terms of neural networks. Note that the unknown state still needs to be determined through a control-to-state map, i.e. via the classical reduced approach, as opposed to the new all-at-once approach.

While \cite{DongHintermuellerPapafitsoros20, CourtKunisch} are the recent publications that are most related to our work\commB{, we also mention the very recent preprint \cite{22DongHintermuellerPapafitsorosVoelkner} on an extension of \cite{DongHintermuellerPapafitsoros20} that appeared independently and after the original submission of our work. Furthermore,} 
there is a wealth of literature on the topic of deep learning emerging in the last decade; for an authoritative review on machine learning in the context of inverse problems, we refer to \cite{ArridgeMaassOektemSchoenlieb:19}. For the regularization analysis, we follow the well known theory put forth in \cite{Engl96_book_regularization_ip_mh,KalNeuSch08,Kirsch,Troeltzsch_optimal_control_pde_mh}.
\commB{It is worthwhile to note that since this work, to the knowledge of the authors, is the first attempt at applying an all-at-once approach to learning-informed PDEs, our focus will be on this novel concept itself, rather than on obtaining minimal regularity assumptions on the involved functions, in particular on the activation functions. In subsequent work, we might further improve upon this by considering, e.g., existing techniques from a classical optimal control setting with non-smooth equations \cite{18ChristofMeyerWaltherClason} or techniques to deal with non-smoothness in the context of training neural networks \cite{20CuiHePang}.}
%22DongHintermuellerPapafitsorosVoelkner

\textbf{Contributions.} Besides introducing the general setting of identifying nonlinearities in PDEs via indirect, parameter-dependent measurements, the main contributions of our work are as follows: Exploiting an all-at-once setting of handling both the state and the parameters explicitly as unknowns, we provide well-posedness results for the resulting learning- and  learning-informed parameter identification problems. This is achieved for rather general, \commA{nonlinear} PDEs and under local Lipschitz assumptions on the activation function of the involved neural network. Further, for the learning-informed parameter identification setting, we ensure the tangential cone condition on the neural-network part of our model. Together with suitable PDEs, this yields local uniqueness results as well as local convergence results of iterative solution methods for the parameter identification problem. We also provide a concrete application of our framework for parabolic problems, where we motivate our function-space setting by a unique existence result on the learning-informed PDE. Finally, we consider a case study in a Hilbert space setting, where we compute function-space derivatives of our objective functional to implement the Landweber method as solution algorithm. Using this algorithm, and also a parallel setting based on the ADAM algorithm \cite{kingma2014adam}, we provide numerical results that confirm feasibility of our approach in practice.

\textbf{Organization of the paper.} Section \ref{sec:setting} introduces learning-informed parameter identification and the abstract setting. Section \ref{sec-abstractPI} examines existence, stability and solution methods for the minimization problem. Section \ref{sec:application} focuses on the learning-informed PDE, and analyzes some problem settings. Finally, in Section \ref{sec:case_study} we present a complete case study, from setup to numerical results.

\section{Problem setting}\label{sec:setting}
%At first, we introduce some basic notation.

\subsection{Notation and basic assertions}
%We will use the following standard notation.
Throughout this work, $\Omega \subset \R^d$ will always be a bounded Lipschitz domain, where additional smoothness will be required and specified as necessary. 
%Further, we always let $p \in \R $ be such that $1\leq p<\infty$ and 
We use standard notations for spaces of continuous, integrable and Sobolev functions with values in Banach spaces, see for instance \cite{DiestelUhl_mh,Roubicek}\commA{, in particular \cite[Section 7.1]{Roubicek} for Sobolev-Bochner spaces and associated concepts such as time-derivatives of Banach-space valued functions.}. For an exponent $p \in [1,\infty]$, we denote by $p^*$ the conjugate exponent given as $p^* = p/(p-1)$ if $p \in (1,\infty)$, $p^* = \infty$ if $p=1$ and $p^* = 1$ if $p=\infty$. For $l \in \N$, we denote by
\[ W^{l,p}\om\embed L^q\om \] the continuous embedding of $W^{l,p}\om$ to $L^q\om$, which exists for $q\preceq \frac{dp}{d-lp}$, \commA{where the notation $\preceq$ means if $lp<d$, then $q\leq \frac{dp}{d-lp}$, if $lp=d$, then $q<\infty$, and if $lp\geq d$, then $ q=\infty $ }. \commB{An example of such an embedding, which will be used frequently in Section \ref{sec:application}, is $H^1(\Omega)\embed L^6(\Omega)$ for $d=3$.} We further denote by  $C_{W^{l,p}\to L^q}$ the operator norm of the corresponding continuous embedding operator.

%
%The continuous embedding $j:W^{k,p}\om\to L^q\om$ \cite[Theorem 1.20]{Roubicek}, \cite{Adams,Evans} satisfies
%\begin{align}\label{embedding}
%W^{k,p}\om\embed L^q\om \quad\text{if}\quad q\preceq \frac{dp}{d-kp} \quad\text{meaning}\quad
%\begin{cases}
%\leq \frac{dp}{d-kp} \quad &kp<d\\
%< +\infty \quad &kp=d\\
%=+\infty \quad &kp>d
%\end{cases},
%\end{align}
%where the notation $q\preceq \frac{a}{b}$ means $q\leq \frac{a}{b}$ with strict inequality if $b = 0$.
We also use $\compt$ to denote the compact embedding (see \cite[Theorem 1.21]{Roubicek})
\begin{align}
W^{l,p}\om\compt L^{q-\epsilon}\om, \qquad \text{for }\epsilon\in(0,q-1].
\end{align}
\commA{The notation $C$ indicates generic positive constants. \commB{Given any Banach spaces $X$, $Y$}, we denote by $\|\cdot\|_{\commB{X\to Y}}$ the operator norm $\|\cdot\|_{\commB{\mathcal{L}(X,Y)}}$, and by $\la\cdot,\cdot\ra_{\commB{X,X^*}}$ the pairing between dual spaces \commB{$X$, $X^*$}. We write $\C_\text{locLip}(\commB{X,Y})$ for the space of locally Lipschitz continuous functions between \commB{$X$} and \commB{$Y$}.
Furthermore, $A\cdot B$ denotes the Frobenius inner product between \commB{generic} matrices $A$, $B$, while $AB$ stands for matrix multiplication, and $A^T$ stands for the transpose of $A$. The notation $\Bc^X_\rho(x^\dagger)$ means a ball of center $x^\dagger$, radius $\rho>0$ in $X$. For functions mapping between Banach spaces, by the term \emph{weak continuity} we will always refer to weak-weak continuity, i.e., continuity w.r.t. weak convergence in both the domain and the image space.}

\subsection{The dynamical system}\label{sec:dynamicsystem}

For the general setting considered in this work, we use the following set of definitions and assumptions. A concrete application where these abstracts assumptions are satisfied can be found in Section \ref{sec:application} below.
\begin{ass} \label{ass:general_basic}\
\begin{itemize}
\item The space $X$ (parameter space) is a reflexive Banach space. The spaces $V$ (state space) and $W$ (image space under the model operator), $Y$ (observation space) and $\tilde{V} $ are separable, reflexive Banach spaces. In view of initial conditions, we further require $U_0 $ (initial data space) to be a reflexive Banach space, and $H$ to be a separable, reflexive Banach space.
%\\\commA{added separability.}

\item We assume the following embeddings: %\commA{moved this item further towards the end of the list in order to have items that are more related to understanding the general setting earlier. TN: I have to move it back here since due to the reviewer's comment (and it is true), we need $H\commA{\embed \tilde{V}}$ thus $V\embed \tilde{V}$ to absorb the $\|u(t)\|_{\tilde{V}}$ in $\|\cdot\|_\Vc$.}
\begin{equation} \label{eq:ass:static_space_embeddings}
U_0 \hookrightarrow H\hookrightarrow  W, \quad  V\embed H\commA{\embed \tilde{V}}, \quad V\embed Y, \quad V\compt L^\pfix(\Omega) \embed  W \text{ for some } \pfix\in [1,\infty).
 \end{equation} 
Further, $\tilde{V}$ will always be such that either  $L^\pfix(\Omega)\hookrightarrow \tilde{V}$ or  $\tilde{V} \hookrightarrow L^\pfix(\Omega)$.

\item The function 
\[ F:(0,T)\times X \times V\to W \]
is such that for any fixed parameter $\lambda \in X$, $F(\cdot,\lambda,\cdot) : (0,T) \times V \rightarrow W$ meets the Carath\'eodory conditions, i.e., $F(\cdot,\lambda,v)$ is  measurable with respect to $t$ for all $v\in V$ and $F(t,\lambda ,\cdot) $ is continuous with respect to $v$ for almost every $t\in (0,T)$. Moreover, \commA{for almost all $t\in(0,T)$ and all $\lambda\in X$, $v\in V$,} the growth condition 
\begin{equation} \label{growth}
\|F(t,\lambda,v)\|_W\leq \Bc(\|\lambda\|_X,\|v\|_H)(\gamma(t)+\|v\|_V) 
\end{equation}
is satisfied for some $\Bc:\R^2\to\R$ such that $b \mapsto \Bc(a,b)$ is increasing for each $a \in \R$, and $\gamma\in L^2(0,T)$. %\commB{These assumptions are prerequisite for $F$ to induce the well-defined Nemytskii operator on overall spaces (c.f. Lemma \ref{lem-nemytskii-F}; recall the abuse of notation introduced following \eqref{ori-eq-1}).}
 %\todo{MH: need already here that $V \subset H$. Suggestions: define $U_0$ and $H$ here, also assume all non-time-space embeddings here and the time-space embeddings below}.
%\item For \todo{remove from general assumption?}
%\[f:\R^m \times \R \to  \R \]
%it holds that $f\in \Cc (\R^m \times \R,\R)$. 

\item  We define the overall state space and image space including time dependence  as
\begin{align}\label{space-UW}
\Vc=L^2(0,T;V)\cap H^1(0,T;\vtil), \quad \Wc=L^2(0,T;W),
\end{align}
respectively with the norms $\|u\|_\Vc:=\sqrt{\int_0^T \|u(t)\|_V^2+\|
\dot{u}(t)\|_\vtil^2\wrt t}$ and $\|u\|_\Wc:=\sqrt{\int_0^T \|u(t)\|_W^2 \wrt t}$.
\item We define the overall observation space including time as
\[ \Yc=L^2(0,T;Y),\]
with the norm $\|y\|_\Wc:=\sqrt{\int_0^T \|y(t)\|_Y^2 \wrt t}$
and the corresponding measurement operator 
\begin{align}\label{M}
 M \in \Lc(\Vc,\Yc).
\end{align}
%\commB{The overall functions spaces form the domain and image space of the Nemytskii operators.}
\item We further assume the following embeddings for the state space:
\begin{align*}
 \Vc\hookrightarrow L^\infty\tom,  \quad   \Vc \hookrightarrow C(0,T;H).
\end{align*}
\end{itemize}
\end{ass}

%\commA{Skip this Remark to shorten the manuscript.}
%\gray{
%\begin{rem}  
%\begin{itemize}
%\item
%By \cite[Proposition 1.38]{Roubicek}, separability and separability of $W$, we have  that $\Wc^*=(L^2(0,T;W))^*=L^2(0,T;W^*)$, which will be used in several estimates. A more detailed discussion about Bochner spaces could be found in, e.g. \cite[Section 1.5]{Roubicek}. 
%\item The embedding $\Vc \hookrightarrow C(0,T;H)$ is required for the initial condition $u(0) = u_0$, where $u_0 \in U_0$ with $U_0$ a set of admissible (potentially smooth) initial data.
%\item 
%Note that above, the assumption $V\hookrightarrow Y$ makes sense in particular in view of the case $M=\text{Id}$ since $\|Mu(t)\|_Y=\|u(t)\|_Y \leq C_{V\to Y}\|u(t)\|_V$, which is one typical case in practice of $M\in\Lc(V,Y)$
%\end{itemize}
%\end{rem}
%}
The embeddings in \eqref{eq:ass:static_space_embeddings} are very feasible in the context of PDEs.  The state space $V$ usually has some certain smoothness such that its image under some spatial differential operators belongs to $W$. For the motivation of $\Vc\hookrightarrow C(0,T;H)$, the abstract setting in \cite[Lemma 7.3.]{Roubicek} \commA{(see Appendix \ref{appendix-Roubicek})} is an example. \commA{Note that due to $\Vc \hookrightarrow C(0,T;H)$, clearly $U_0=H$ is a feasible choice for the initial space; for the sake of generality, only $U_0\embed H$ is assumed in \eqref{eq:ass:static_space_embeddings}.}

Under Assumption \ref{ass:general_basic}, the function $F$ induces a Nemytskii operator on the overall spaces.
\begin{lem}\label{lem-nemytskii-F}
Let Assumption \ref{ass:general_basic} hold. Then the function $F:(0,T)\times X \times V\to W $ induces a well-defined Nemytskii operator $F: X \times \Vc \rightarrow \Wc$ given as
\begin{equation}\label{Nemytskii}
[F(\lambda,u)](t) = F(t,\lambda,u(t)).
\end{equation}
\begin{proof}
%\todo{add proof by possibly referring to later lemmata and references:  \cite[Theorem 1.43]{Roubicek}, \cite[Section 1.3]{Roubicek}}
Under the Carath\'eodory assumption, $t \mapsto F(t,\lambda,u(t))$ is Bochner \commA{measurable} for every $\lambda \in X$ and $u \in \Vc$. For such $\lambda,u$, we further estimate
\begin{align*}
 \int_0^T \|F(t,\lambda,u(t) \|^2_W \commB{\wrt t}
 & \leq 2 \int_0^T \Bc(\|\lambda\|_X,\|\commA{u}(t)\|_H)^2(\gamma(t)^2+\|\commA{u}(t)\|^2_V) \commB{\wrt t}\\
 & \leq 2\Bc(\|\lambda\|_X,\|\commA{u}\|_{C(0,T;H)})^2( \|\gamma\|^2_{L^2(0,T)} + \|u\|_{\Vc}^2 ) < \infty
\end{align*}
by $b \mapsto \Bc(\|\lambda\|,b)$ being increasing and by the embedding $ \Vc \hookrightarrow C(0,T;H)$. This allows to conclude that $t \mapsto F(t,\lambda,u(t))$ is Bochner integrable (see \cite[Theorem II.2.2]{DiestelUhl_mh}) and that the Nemytskii operator $F: X \times \Vc \rightarrow \Wc$ is well-defined. 
%The pointwise function $f\in \Cc (\R^m \times \R,\R)$ obviously is a Carath\'eodory mapping, thus the mapping $(t,x) \mapsto f(\alpha,u(t,x))$ is measureable for every $\alpha \in \R^m$ and $u \in \Vc$.
% thus inducing the Nemytskii operator $f: \R^m\times \Vc \rightarrow \Wc$ \cite[Section 1.3]{Roubicek}. For $(\alpha, u)\in\R^m\times \Vc, \|(\alpha,u)\|_{L^\infty\tom}\leq M$ for some $M\geq0$ as $\Vc\embed L^\infty\tom$. Hence, with $f\in \Cc (\R^m \times \R,\R)$, $\|f(\alpha,u)\|_{L^\infty\tom}\leq N$ for some $N\geq 0$, which yields that $f(\alpha,u)\in\Wc$ if $L^p\om\embed W$. This shows well-definedness of the Nemytskii mapping $f$.
%Its well-definedness follows by the same logic as \eqref{welldefined-N}, with $f$ in place of $\Nc$, and $\theta$ is not present. There, the embeddings $\Vc\embed L^\infty\tom$ and $L^p\om\embed W$ are employed.
\end{proof}
\end{lem}
Note that we use the same notation for the function $F:(0,T) \times X \times V \rightarrow W$ and the corresponding Nemytskii operators.

%\todo{TN: maybe list concrete example already here?}
%\commA{flipped sign of NN to match others}
%An application that we will consider later is the system
%\begin{alignat}{3}\label{app}
%& \dot{u} - \nabla\cdot(a\nabla u) + cu - \Nc(u,\theta) = \varphi \quad&&\mbox{ in }\Omega\times\ti, \nonumber\\
%& u|_{\partial\Omega}=0 && \mbox{ in } \ti,\\
%& u(0) = u_0 &&\mbox{ in }\Omega, \nonumber
%\end{alignat}
%where $(a,c, \varphi,u_0), u$ and $\theta$ are to be determined.

\subsection{Basics of neural networks}\label{sec:CNN}
As outlined in the introduction, the unknown nonlinearity $f$ will be represented by a neural network. In this work, we use a rather standard, feed-forward form of neural networks defined as follows.

%\note{simplified architecture}
\begin{defn} \label{def:neural_network} A neural network $\Nc_\theta$ of depth $L \in \N$ with architecture $(n_i)_{i=0}^{L}$ is a function $\Nc_\theta:\R^{n_0 } \rightarrow \R^{n_L }$ of the form
\[ \Nc_\theta (x) = L_{\theta_L}\circ \ldots \circ L_{\theta_1}(x)\] where
$L_{\theta_l}:\R^{n_{l-1} } \rightarrow \R^{n_{l}} $, for $z \in \R^{n_{l-1} }$ is given as 
\[ L_{\theta_l}(z) :=   \sigma ( \weight^l z + \bias^l ) \text{ for $l=1,\ldots,L-1$}, \qquad L_{\theta_L}(z) :=   \weight^L z + \bias^L.\]
Here, $\weight^l \in \Lc( \R^{n_{l-1}},\R^{n_{l}})$, $\bias^l \in \R^{n_{l}}$, $\theta_l = (\weight^l,\bias^l)$ summarizes all the parameters of the $l$-th layer and $\sigma$ is a pointwise nonlinearity that is fixed. Given a depth $L \in \N$ and architecture $(n_i)_{i=0}^{L}$, we also use $\Theta $ to denote the finite dimensional vector space containing all possible parameters $\theta _1,\ldots,\theta_L$ of neural networks with this architecture. 
\end{defn}
%\begin{defn} \label{def:neural_network} A neural network $\Nc_\theta$ of depth $L \in \N$ with architecture $((n_i,k_i))_{i=0}^{L}$ is a function $\Nc_\theta:\R^{n_0 \times k_0} \rightarrow \R^{n_L \times k_L}$ of the form
%\[ \Nc_\theta (x) = L_{\theta_L}\circ \ldots \circ L_{\theta_1}(x)\] where
%$L_{\theta_l}:\R^{n_{l-1} \times k_{l-1}} \mapsto \R^{n_{l} \times k_{l}} $, for $z \in \R^{n_{l-1} \times k_{l-1}}$, is given, for $l=1,\ldots,L-1$, as 
%\[ L_{\theta_l}(z)_{\cdot ,j} =  \sum_{s=1}^{k_{l-1}} \sigma ( \weight_{s,j}^l z_{\cdot,s} + \bias_{s,j}^l ) \quad \text{ for }j = 1,\ldots,k_{l}\]
%and 
%\[  L_{\theta_L}(z)_{\cdot ,j} =  \sum_{s=1}^{k_{L-1}}  \weight_{s,j}^L z_{\cdot,s} + \bias_{s,j}^L  \quad \text{ for }j = 1,\ldots,k_{L}.\]
%Here, $\weight^l_{s,j} \in \Lc( \R^{n_{l-1}},\R^{n_{l}})$ and $\bias_{s,j}^l \in \R^{n_{l}}$ and $\theta_l = (\weight^l_{s,j},\bias^l_{s,j})_{s,j}$ summarizes all the parameters of the $l$-th layer and $\sigma$ is a pointwise nonlinearity that is fixed. Given a depth $L \in \N$ and architecture $((n_i,k_i))_{i=0}^{L}$, we also use $\Theta $ to denote the finite dimensional vector space containing all possible parameters $\theta _1,\ldots,\theta_N$ of neural networks with this architecture. 
%\end{defn}
In this work, neural networks will be used to approximate the nonlinearity $f:\R^{m+1} \rightarrow \R$. Consequently, we always deal with neural networks $\Nt:\R^{m+1} \rightarrow \R$, i.e., $n_0=m+1$ and $n_L = 1$.

\commB{As such, rather than showing that $f$ induces a well-defined Nemytskii operator, we instead show that $\Nt$ does so. A sufficient condition for this to be true is the continuity of the activation function $\sigma$, as the following Lemma shows.}
 
%\commB{As with $F$ in Lemma \ref{lem-nemytskii-F}, one needs to ensure well-posedness of the Nemytskii operator induced by the unknown nonlinearity $f$. Since the neural network $\Nt$ is incorporated to represent $f$ in our approach, this means imposing neccessary conditions on $\Nt$, such as the Carath\'eodory condition, together with suitable function space embeddings (see Lemma \ref{lem-nemytskii-N}}). \commA{The Carath\'eodory condition on $\Nt$ means that $\Nt$ is continuous with respect to $(\alpha,z)\in \R^m\times\R$; this requires an assumption on the continuity of activation function.}

% \commA{.
%\begin{ass}\label{ass:general_NN}
%The network $\Nt$
%\[\Nt:\R^m\times\R\to\R\]
%has continuous activation $\sigma \in \Cc(\R,\R)$. 
%\end{ass}}

\begin{lem}\label{lem-nemytskii-N}
%Let Assumption \ref{ass:general_NN} hold.
Assume that $\sigma \in \Cc(\R,\R)$. Then, with the setting of Assumption \ref{ass:general_basic}, $\Nc_\theta:\R^{m} \times \R \rightarrow \R$ as in Definition \ref{def:neural_network} induces a well-defined Nemytskii operator $\Nc_\theta: \R^m \times \Vc \rightarrow  L^2(0,T;L^{\commA{\pfix}}(\Omega)) $ via
\[ [\Nc_\theta (\alpha,u)](t)(x) = \Nc_\theta (\alpha,u(t,x)) ,\]
regarding $u \in \Vc$ as $u \in  L^\infty\tom$ by the embedding $ \Vc\hookrightarrow L^\infty\tom$. Further, using the embedding $L^2(0,T;L^{\commA{\pfix}}(\Omega)) \embed  \Wc$, $\Nc_\theta$ induces a well-defined Nemytskii operator $\Nc_\theta:\R^m \times \Vc \rightarrow \Wc$.
\begin{proof}
We first fix $\alpha \in \R^m$. By continuity of $\sigma$, $\Nc_\theta$ is also continuous and, for  $u \in L^\infty((0,T) \times \Omega)$, $\sup_{t,x} |\Nc_\theta(\alpha,u(t,x))| < \infty$; thus, $\Nc_\theta(\alpha,u(t,\cdot)) \in L^\pfix(\Omega)$ for \commB{almost} every $t \in (0,T)$. It then follows by standard measurability arguments that the mapping $t \mapsto \int_\Omega \Nc_\theta(\alpha,u(t,x))w^*(x)\wrt x$ is measurable for every $w^* \in L^{\pfix^*}(\Omega)$. Using separability and the Pettis theorem \cite[Theorem II.1.2]{DiestelUhl_mh}, it follows that $t \mapsto \Nc_\theta(\alpha,u(t,\cdot)) \in L^\pfix(\Omega)$ is Bochner measureable. This, together with  $\sup_{t,x} |\Nc_\theta(\alpha,u(t,x))| < \infty$ as before, implies that the Nemytskii operator  $\Nc_\theta:\R^m \times \Vc \rightarrow L^2(0,T;L^\pfix(\Omega)) $ is well defined. The remaining assertions follow immediately from $L^\pfix(\Omega) \embed W$.
\end{proof}
\end{lem}
We again use the same notation for $\Nc_\theta:\R^{m} \times \R \rightarrow \R$ and the corresponding Nemytskii operator.

\subsection{The learning problem}\label{sec:learning-problem}
%In order to define the minimization problem associated to learning an approximation of the nonlinearity $f$ via $\Nc_\theta$, we represent the PDE via the function
As the nonlinearity $f$ is represented by a neural network $\Nt:\R^{m+1} \rightarrow \R$, we rewrite the partial-differential-equation (PDE) model  \eqref{ori-eq-1} into the form
\begin{align}\label{nn-eq-2}
e:X\times \R^m  \times U_0  \times \Vc \times \Theta \to\Wc\times H, \quad  e(\lambda,\alpha,u_0,u,\theta)=(\dot{u}-F(\lambda,u)-\Nc_\theta(\alpha,u),u(0)-u_0),
\end{align}
and introduce the forward operator $\Gc$, which incorporates the observation operator $M$, as
\begin{align}
\begin{split}\label{eq:forward_operator_general}
\Gc:X \times \R^m \times U_0 \times\Vc \times \Theta\to\Wc\times H \times \Yc, \\
\Gc(\lambda,\alpha,u_0,u,\theta)=(e(\lambda,\alpha,u_0,u,\theta),Mu).%=(0,0,y).
\end{split}
\end{align}
Here, $U_0$ and $H$ are \commA{the} spaces related to the initial condition \commA{and the trace operator, that is, one has unknown initial data $u_0\in U_0$ and trace operator $(\cdot)_{t=0}: \Vc\ni u\mapsto u(0)\in H$. With $U_0\embed H$ as assumed in \eqref{eq:ass:static_space_embeddings}, one has $u(0)-u_0\in H$}. %\commA{Remove: We refer to Section \ref{sec:setting} for details on the involved spaces and functionals.}

The minimization problem for the learning process is then given by
\begin{equation} \label{eq:main_identification_aao_setting}
\min_{\substack{
(\lambda^k,\alpha^k,u_0^k,u^k)_k \subset X \times \R^m \times U_0\times\Vc \\
\theta \in \Theta
}} \sum_{k=1}^K \| \Gc(\lambda^k,\alpha^k,u_0^k,u^k,\theta) - (0,0,y^k) \|^2_{\Wc\times H\times\Yc} + \Rc_1(\lambda^k,\alpha^k,u_0^k,u^k)  +  \Rc_2(\theta),
\end{equation}
where $\Rc_1:X \times \R^m \times U_0 \times \Vc \rightarrow [0,\infty]$ and $\Rc_2:\Theta \rightarrow [0,\infty]$ are suitable regularization functionals.

\commB{Assume now that the particular parameter $\hat{\theta}$ has been learned. As in \eqref{par-id}, one can now solve other parameter identification problems, given new measured datum $y\approx Mu$}, by solving

\begin{equation} \label{eq:standard_pid_aao_setting}
\min_{
(\lambda,\alpha,u_0, u)\in X \times \R^m \times U_0\times \Vc
} \| \Gc(\lambda,\alpha,u_0,u,\hat{\theta}) - (0,0,y) \|^2_{\Wc\times H\times\Yc} + \Rc_1(\lambda,\alpha,u_0,u).
\end{equation}

\section{Learning-informed parameter identification}\label{sec-abstractPI}

\subsection{Well-posedness of minimization problems}
We start our analysis by studying existence theory for the optimization problems \eqref{eq:main_identification_aao_setting} and \eqref{eq:standard_pid_aao_setting}, where the unknown nonlinearity is replaced by a neural network approximation. To this aim, we first establish weak closedness of the forward operator. \commA{In what follows, the architecture of the network $\Nc$ is considered fixed.}

%\commA{weakened $\Cc^1$ to $\C_\text{locLip}(\R,\R)$ to include, e.g. RELU activation.}
\begin{lem} \label{lem:closedness_forward_operator}
Let Assumption \ref{ass:general_basic} hold. Then, if $\sigma \in \C_\text{locLip}(\R,\R)$,  $\Nc:\R^m \times \Vc \times \Theta \rightarrow \Wc$ is weakly continuous. Further, if
 \textbf{either} 
 \begin{equation} \label{eq:existence_f_weak_cont_ass}
 F(t,\cdot):X \times H \rightarrow W \text{ is weakly continuous for a.e. }t \in (0,T)
\footnote{\commA{This somewhat abusive notation in particular implies that for this specific case distinction to hold, $F$ in \eqref{ori-eq-1} must have been well defined for $u\in H$, with Assumption \ref{ass:general_basic} holding for the restriction of $F(t,\cdot):X \times H \rightarrow W$ to $X\times V$.}}
\end{equation}
\textbf{or} 
\begin{equation}  \label{pseudo_ass}
V\compt H, \quad  H\hookrightarrow W^*,
\end{equation}
\commA{and $(-F)$ is pseudomonotone in the sense that} for almost all $t\in(0,T),$
\begin{align}\label{pseudo_def}	
\left.\begin{aligned}
 (u_k(t),\lambda_k) \overset{H\times X}{\rightharpoonup} (u,\lambda)&\\
\underset{k\rightarrow\infty}{\liminf} \la F(t,\lambda_k,u_k(t)),u_k(t)-u(t)\ra_\wws \geq 0&
\end{aligned}\right\}
\Rightarrow
\begin{cases}
\forall v \in W^*: \la F(t,\lambda,u(t)),u(t)-v \ra_\wws \\ \geq \underset{k\rightarrow\infty}{\limsup}\la F(t,\lambda_k,u_k(t)),u_k(t)-v \ra_\wws,
\end{cases}
\end{align} 
then $F$ is weakly closed. Moreover, if $\Nc$ is weakly continuous and $F$ is weakly closed, then $\Gc$ as in \eqref{eq:forward_operator_general} is weakly closed.

%if $\sigma \in \C_\text{locLip}(\R,\R)$ and either \eqref{eq:existence_f_weak_cont_ass} or \eqref{pseudo_ass} holds, then the operator $\Gc$ as in \eqref{eq:forward_operator_general} is weakly closed as mapping from $X \times \R^m \times U_0 \times\Vc \times \Theta$ to $\Wc\times H \times \Yc$ in the sense that $x_n\rightharpoonup x$ and $ \Gc(x_n)\rightharpoonup g$ implies $ \Gc(x) = g$.
\begin{proof}
We first consider weak closedness of $\Gc$. To this aim, recall that $\Gc$ is given as
 \[ \Gc(\lambda,\alpha,u_0,u,\theta)=(\dot{u}-F(\lambda,u)-\Nc_\theta(\alpha,u),u(0)-u_0,Mu). \]
First note that $ M\in \Lc(\Vc,\Yc)$ by \commB{\eqref{M}}. Weak closedness of $((\cdot)_{t=0},\text{Id}):\Vc \times U_0 \rightarrow H$ follows from weak continuity of $\text{Id}:U_0 \rightarrow H$ as $U_0 \hookrightarrow H$, and from weak-weak continuity of $(\cdot)_{t=0}:\Vc \rightarrow H$ which follows from $\|u(0)\|_H \leq \sup_{t\in [0,T]}\|u(t)\|_H \leq C\|u\|_\Vc$ for $C>0$ and $\Vc\hookrightarrow C(0,T;H)$. Weak continuity of $\frac{d}{dt}:\Vc \rightarrow \Wc$ results from the choice of norms in the respective spaces. Thus, weak closedness of $\Gc$ follows when $F$ is weakly closed and $\Nc$ is weakly continuous.

\textbf{Weak continuity of $\Nc$.}
First, we observe that $\Nc:  \R^m \times \R \times \Theta  \to  \R, (\alpha,y,\theta) \mapsto \Nc_\theta(\alpha,y)$ is in $\C_\text{locLip}(\Theta\times\R^m \times \R,\R)$, since the activation function $\sigma $ is locally Lipschitz continuous. 
For a sequence $(\alpha_n,u_n,\theta_n)_n$ converging weakly to $(\alpha,u,\theta)$ in $\R^m \times \Vc \times \Theta$, we observe that by the embedding $\Vc\hookrightarrow L^\infty\tom$, $\sup_{t,x} \| (\alpha_n,u_n(t,x),\theta_n)\|< M$ for some $M>0$. 
%This particularly indicates that $\Nc_\theta(\alpha)$ satisfies the Carath\'eodory conditions and thus induces the corresponding Nemytskii operator $\Nta:\Uc\to\Wc.$\todo{details on nemytskii: why a mapping between correct function spaces. Also, maybe jointly estimate continuity in $\alpha,y$ and $\theta$}

Now the embeddings  $V\compt L^\pfix(\Omega)\hookrightarrow W$ %\todo{TN: need $$?}
imply in particular that $\Vc\compt L^2(0,T;L^\pfix(\Omega))$ 
(in case $\widetilde{V}\hookrightarrow L^\pfix\om$, this follows from $\Vc\subset L^2(0,T;V)\cap H^1(0,T;L^\pfix(\Omega))\compt L^2(0,T;L^\pfix(\Omega))$ together with \cite[Lemma 7.7]{Roubicek} \commA{(see Appendix \ref{appendix-Roubicek})}, in the other case that $L^\pfix \om \hookrightarrow \widetilde{V}$, this follows directly from \cite[Lemma 7.7]{Roubicek}). Based on this, we deduce $u_n\to u$ in $L^2(0,T;L^\pfix(\Omega))$.
Then
\begin{align}\label{welldefined-N}
&\|\Nc(\alpha_n,u_n,\theta_n)-\Nc(\alpha,u,\theta)\|_\Wc
 =\underset{ \substack{ w^* \in \Wc^*, \\ \|w^*\|_{\Wc^*}\leq1}}{\sup} \la \Nc(\alpha_n,u_n,\theta_n)-\Nc(\alpha,u,\theta),w^* \ra_\WWs \nonumber\\
&\quad =\underset{ \substack{ w^* \in \Wc^*, \\ \|w^*\|_{\Wc^*}\leq1}}{\sup} \int_0^T \int _ \Omega \left(\Nc(\alpha_n,u_n(t,x),\theta_n)-\Nc(\alpha,u(t,x),\theta)\right)w^*(t,x)  \wrt x\wrt t \nonumber\\
&\quad\leq L(M) \underset{\substack{ w^* \in \Wc^*, \\ \|w^*\|_{\Wc^*}\leq1}}{\sup}\int_0^T  \int_\Omega  \left( |\alpha_n - \alpha| + |u_n(t,x)-u(t,x)|  + |\theta_n - \theta| \right)|w^*(t,x)|  \wrt x\wrt t \nonumber\\
&\quad\leq C L(M)  \underset{\substack{ w^* \in \Wc^*, \\ \|w^*\|_{\Wc^*}\leq1}}{\sup} \left( \|u_n-u\|_{L^2(0,T;L^\pfix(\Omega))} + |\alpha_n - \alpha| + |\theta_n-\theta| \right) \|w^*\|_{L^2(0,T;L^{\frac{\pfix}{\pfix-1}}(\Omega))} \nonumber\\
 &\quad\leq C L(M)   \left( \|u_n-u\|_{L^2(0,T;L^\pfix(\Omega))} + |\alpha_n - \alpha| + |\theta_n-\theta| \right) \overset{n\to\infty}{\to}0
\end{align}
as $W^*\hookrightarrow L^{\frac{\pfix}{\pfix-1}}(\Omega)$, $u_n\overset{n\to\infty}{\to}u$ in $L^2(0,T;L^\pfix(\Omega))$, \commB {and $\Nc(\alpha,u_n(t,\cdot),\theta_n)) \in L^\pfix(\Omega)$, as argued in the proof of Lemma  \ref{lem-nemytskii-N}}. Above $L(M)$ denotes the Lipschitz constant of $(\alpha,y,\theta) \mapsto \Nc_\theta(\alpha,y)$ in the ball with radius $M$ and $\pfix/(\pfix-1) = \infty$ in case $\pfix=1$. \commA{This shows that here, we even obtain weak-strong continuity of $\Nt$, which is stronger than weak-weak continuity, as required.}

\textbf{Weak closedness of $F$.} To show weak closedness of the Nemytskii operator $F:X \times \Vc\to\Wc$, \commB{we consider two cases}. We first consider the case that $F(t,\cdot)$ is weakly continuous. To this aim, take $(\lambda_n,u_n)_n$ to be a sequence weakly converging to $(\lambda,u)$ in $X  \times \Vc$.
As $\Vc\hookrightarrow C(0,T;H)$, we have $u_n\overset{C(0,T;H)}{\rightharpoonup}u$ as $n \rightarrow \infty$. \commA{Now, we show $u_n(t)\overset{H}{\rightharpoonup}u(t)$ for all $t\in(0.T)$ via the fact that the point-wise evaluation function $(\cdot)(t):\Vc\to H$ for any $ t\in[0,T]$ is linear and bounded, thus weak-weak continuous}. Indeed, its linearity is clear and boundedness follows from 
\[ \la(\tilde u)(t),h^*\ra_\hhs\leq \underset{ \tilde t\in[0,T]}{\max}\la(\tilde u)(\tilde t),h^*\ra_\hhs \leq \|\tilde u\|_{C(0,T;H)}\|h^*\|_H\leq C\|\tilde u\|_\Vc\|h^*\|_H.\]
From this, we obtain \commA{$u_n(t)\overset{H}{\rightharpoonup}u(t)$}, thus having $ (u_n(t),\lambda_n)\overset{H\times X}{\rightharpoonup}(u(t),\lambda))$ for all $t \in (0,T).$ Using the growth condition \eqref{growth}, we now estimate
\begin{alignat}{3} \label{welldefined-F}
&\la F(\lambda_n,u_n)-F(\lambda,u),w^* \ra_\WWs = \int_0^T  \la F(\lambda_n,u_n)(t)-F(\lambda,u)(t),w^*(t) \ra_\wws  \wrt t =: \int_0^T \epsilon_n(t)\wrt t \nonumber\\
& \leq  \int_0^T(\|F(\lambda_n,u_n)(t)\|_W+\|F(\lambda,u)(t)\|_W)\|w^*(t)\|_{W^*}\wrt t \nonumber\\
& \leq \left(\Bc(\|\lambda_n\|_X,\sup_t \|u_n(t)\|_H)(\|\gamma\|_{L^2(0,T)}  +\|u_n\|_\Vc) + \Bc(\|\lambda\|_X,\sup_t \|u(t)\|_H)(\|\gamma\|_{L^2(0,T)}+\|u\|_\Vc)\right)\|w^*\|_{\Wc^*}\nonumber\\& \leq C(\|\lambda\|_X,\|u\|_\Vc)\|w^*\|_{\Wc^*},
\end{alignat}
where $C(\|\lambda\|_X,\|u\|_\Vc)>0$ can be obtained independently from $n$ due to $\Vc\hookrightarrow C(0,T;H)$, $\Bc$ being increasing, and boundedness of $((u_n,\lambda_n))_n $ in $\Vc \times X$. 
Since F is assumed to be weakly continuous on $H\times X$, when $n\to\infty$ we have $\epsilon_n(t)\to0$ pointwise in $t$. Hence, applying Lebesgue’s Dominated Convergence Theorem yields convergence of the time integral to $0$, thus weak convergence of $F(\lambda_n,u_n)$ to $F(\lambda,u)$ in $\Wc$ as claimed. \commB{Accordingly, if the condition \eqref{eq:existence_f_weak_cont_ass} holds, we obtain weak-weak continuity of $F$.}

Now we consider the second case, \commB{i.e. \eqref{pseudo_ass}-\eqref{pseudo_def},  for weak closedness of $F$}. Assume that $ V\compt H$ as in \eqref{pseudo_ass}, $ H\hookrightarrow W^*$
and that $-F$ is pseudomonotone as in \eqref{pseudo_def}.
Given $(u_n,\lambda_n) \overset{\Vc\times X}{\rightharpoonup} (u,\lambda)$, $F(\commA{\lambda_n,u_n})\overset{\Wc}{\rightharpoonup} g$ and $ V\compt H\hookrightarrow \commA{W^*}, H \embed \tilde{V}$, %\todo{again, we need $H \embed \tilde{V}$, right?}
it follows that $\Vc\compt L^2(0,T;H)$ \cite[Lemma 7.7]{Roubicek} \commA{(see Appendix \ref{appendix-Roubicek})} and that $u_{n}\to u $ strongly in $L^2(0,T;H)$. 
%Hence there exists a further (non-relabeled) subsequence $(u_{n_k})_n$ such that $u_{n_k}(t)\overset{H}{\to}u(t)$ for almost every  $t\in(0,T)$, which, using that $H \hookrightarrow W^*$, implies that $ u_{n_k}(t)\overset{W^*}{\to}u(t)$ for almost every  $t\in(0,T)$. 
By the embedding $H \hookrightarrow W^*$, it holds also $u_{n} \rightarrow u$ in $\Wc^*$. With $\xi_{n}\commB{(t)}:= |\la F(t,\lambda_{n},u_{n}(t)),u_{n}(t)-u(t)\ra_\wws|$, we obtain 
\begin{align} \label{pseudo-premise}
 \int_0^T |\xi_{n}(t)|\wrt t  \leq \|F(\lambda_{n},u_{n})\|_\Wc \|u_{n}-u\|_{\Wc^*} \leq C\|u_{n}-u\|_{\Wc^*}\overset{n\to\infty}{\rightarrow}0.
\end{align}
By moving to a subsequence indexed by $(n_k)_k$, we thus have $\xi_{n_k}(t)\to 0 $ as $k\to\infty$ for almost every $t\in(0,T)$. As $\underset{k\rightarrow\infty}{\liminf} \, \xi_{n_k}(t)\to 0$, pseudomonotonicity (as in \eqref{pseudo_def}) implies that for any $v\in\Wc^*$,
\begin{align*}
\la F(t,u(t),\lambda),u(t)-v(t) \ra_\wws \geq \underset{k\rightarrow\infty}{\limsup}\la F(t,u_{n_k}(t),\lambda_{n_k}),u_{n_k}(t)-v(t) \ra_\wws.
\end{align*}
%Further, from the reverse Fatou Lemma \todo{reference} and weak convergence of $F(\lambda_n,u_n)$ to $g$ in $\Wc$, we obtain that for any $v \in \Wc^*$,
%\begin{align*}
%\int_0^T\underset{k\rightarrow\infty}{\limsup} \la F(\lambda_{n_k},u_{n_k}(t)),u(t)-v(t) \ra_\wws\wrt t &\geq \underset{k\rightarrow\infty}{\limsup} \int_0^T\la F(\lambda_{n_k},u_{n_k}(t)),u(t)-v(t) \ra_\wws\wrt t\\
%&=  \la g,u-v \ra_\WWs.
%\end{align*}
Further, from the Fatou–Lebesgue theorem, we get
\begin{align*}
&\la F(\lambda,u),u-v \ra_\WWs = \int_0^T \la F(t,\lambda,u(t)),u(t)-v(t) \ra_\wws\wrt t\\ 
&\geq \int_0^T\underset{k\rightarrow\infty}{\limsup} \la F(t,\lambda_{n_k},u_{n_k}(t)),u_{n_k}(t)-v(t) \ra_\wws\wrt t \\
&\geq \underset{k\rightarrow\infty}{\liminf} \int_0^T\la F(t,\lambda_{n_k},u_{n_k}(t)),u_{n_k}(t)-v(t) \ra_\wws\wrt t\\
%\geq \int_0^T \underset{k\rightarrow\infty}{\lim} \la F(t,u_{n_k}(t),\lambda_k),u_{n_k}(t)-v(t) \ra_\wws\wrt t \geq \\
&\geq  \underset{k\rightarrow\infty}{\liminf}\int_0^T  \la F(\lambda_{n_k},u_{n_k}(t)),u_{n_k}(t)-u(t) \ra_\wws\wrt t+\underset{k\rightarrow\infty}{\liminf}\int_0^T \la F(\lambda_{n_k},u_{n_k}(t)),u(t)-v(t) \ra_\wws\wrt t\\
&= \underset{k\rightarrow\infty}{\lim}\int_0^T   \la F(\lambda_{n_k},u_{n_k}(t)),u_{n_k}(t)-u(t) \ra_\wws\wrt t+\underset{k\rightarrow\infty}{\lim}\int_0^T \la F(\lambda_{n_k},u_{n_k}(t)),u(t)-v(t) \ra_\wws\wrt t\\
&=0+\la g,u-v \ra_\WWs,
\end{align*}
where the last estimate follows from \eqref{pseudo-premise} and from weak convergence of of $F(\lambda_n,u_n)$ to $g$ in $\Wc$. As this estimate is valid for any $v\in\Wc^*$, we conclude that $F$ is weakly closed on $X\times\Vc$, that is,
\[F(\commA{\lambda,u})=g.\]
\end{proof}
\end{lem}

%\commA{added remark}
%\begin{rem}
%If $\sigma$ is Lipschitz continuous, say $\C_\text{Lip}(\R,\R)$, for weak closeness of $\Nc$ the inclusion $\Vc\hookrightarrow L^\infty\tom$ is not necessary.
%\end{rem}

Existence of a solution to \eqref{eq:main_identification_aao_setting} and \eqref{eq:standard_pid_aao_setting} now follows from a standard application of the direct method \cite{Engl96_book_regularization_ip_mh, Troeltzsch_optimal_control_pde_mh}, using weak-closedness of $\Gc$ and weak lower semi-continuity of the involved quantities.

% from standard results  that essentially assume weakly closedness of the forward operator. 

\begin{prop}[Existence]\label{prop-minexist}
Let the assumptions of Lemma \ref{lem:closedness_forward_operator} hold, and assume that $\Rc_1, \Rc_2 $ are nonnegative, weakly lower semi-continuous and such that the sublevel sets of $(\lambda,\alpha,u_0,u,\theta) \mapsto \Rc_1(\lambda,\alpha,u_0,u) + \Rc_2(\theta)$ are weakly precompact. Then the minimization problems \eqref{eq:main_identification_aao_setting} and \eqref{eq:standard_pid_aao_setting} 
%in the space setting \eqref{space-UW} 
admit a solution.
\end{prop}
%
%\begin{proof}
%Using weak-closedness of $\Gc$ as shown in Lemma \ref{lem:closedness_forward_operator} and weak lower semi-continuity of $\|\cdot\|_{\Wc\times H\times\Yc},$ $\Rc_1$ and $ \Rc_2$, existence follows from a standard application of the direct method. 
%\end{proof}

\begin{rem}[Stability] We note that under the assumptions of Proposition \ref{prop-minexist}, also stability for the minimization problems \eqref{eq:main_identification_aao_setting} and \eqref{eq:standard_pid_aao_setting} follows with standard arguments, see for instance \cite[Theorem 3.2]{Hofmann07_nonlinear_tikhonov_banach_mh}. Here, stability means that for convergent sequence of data $(y_n)_n$ converging to some $y$, any corresponding sequence of solutions admits a weakly convergent subsequence, and any limit of such weakly convergent subsequence is a solution of the original problem with data $y$.
\end{rem}

Next we deal with minimization problem \eqref{eq:main_identification_aao_setting} in the limit case where the given data converges to a noise-free ground truth, and the PDE should be fulfilled exactly. Our result in this context is a direct extension of classical results as provided for instance in \cite{Hofmann07_nonlinear_tikhonov_banach_mh}, but since also variants of this result will be of interest, we provide a short proof.
\begin{prop}[\commE{Limit case}]\label{prop-convergence} With the assumption of Proposition \ref{prop-minexist} and parameters $\beta^e,\beta^M>0$, consider the parametrized learning problem
\begin{multline} \label{eq:parametrized_learning_problem}
\min_{\substack{
(\lambda^k,\alpha^k,u_0^k,u^k)_k \subset X \times \R^m \times U_0\times\Vc \\
\theta \in \Theta
}} \sum_{k=1}^K \beta^e\| e(\lambda^k,\alpha^k,u_0^k,u^k,\theta) \|^2_{\Wc\times H} + \beta^M\| Mu^k-y^k \|^2_{\Yc}  \\ + \Rc_1(\lambda^k,\alpha^k,u_0^k,u^k)  +  \Rc_2(\theta),
\end{multline}
and assume that, for $((y^\dagger)^k)_k \in \Yc^K$, there exists $(\hat \lambda^k,\hat	 \alpha^k,\hat u_0^k,\hat u^k)_k \in X \times \R^m \times U_0\times\Vc $ and $\hat \theta \in \Theta $ such that $ e(\hat \lambda^k,\hat \alpha^k,\hat u_0^k,\hat u^k,\hat \theta) = 0 $ \commB{, $ 
M\hat u^k = (y^\dagger)^k$, $\Rc_1(\hat \lambda^k,\hat \alpha^k,\hat u_0^k,\hat u^k)< \infty $ for all $k$ and $  \Rc_2(\hat  \theta)< \infty$.}
%\todo{$y^\dagger$ vs. $(y^\dagger)^k$ in multiple places.}

Then, for any sequence $(y_n)_n = (y_n^1,\ldots,y_n^K)_n $ in $\Yc^K$ with $\sum_{k=1}^K\|y^k_n - \commA{(y^\dagger)^k} \|^2_{\Yc} := \delta_n^2 \rightarrow 0$ and parameters $\beta_n^e,\beta_n^M$ such that
\[ \beta^e_n \rightarrow \infty, \, \beta^M_n \rightarrow \infty \text{ and }\beta^M_n \delta^2_n \rightarrow 0 \]
as $n \rightarrow \infty$, any sequence of solutions $((\lambda_n^k,\alpha_n^k,(u_0^k)_n,u_n^k)_k,\theta_n)_n$ of \eqref{eq:parametrized_learning_problem} with parameters $\beta_n^e,\beta_n^M$ \commB{and data $y_n$} admits a weakly convergent subsequence, and any limit of such a subsequence is a solution to
\begin{equation} \label{eq:convergence_r_minimizing}
\min_{\substack{
(\lambda^k,\alpha^k,u_0^k,u^k)_k \subset X \times \R^m \times U_0\times\Vc \\
\theta \in \Theta
}} \sum_{k=1}^K\Rc_1(\lambda^k,\alpha^k,u_0^k,u^k)  +  \Rc_2(\theta) \quad \text{s.t. for all }k: \begin{cases}
e(\lambda^k,\alpha^k,u_0^k,u^k,\theta) = 0 \\ 
Mu^k = (y^\dagger )^k
\end{cases}
\end{equation}
If, further, the solution to \eqref{eq:convergence_r_minimizing} is unique, then the entire sequence  $((\lambda_n^k,\alpha_n^k,(u_0^k)_n,u_n^k)_k,\theta_n)_n$ weakly converges to the solution of \eqref{eq:convergence_r_minimizing}.
\begin{proof}
With $(\hat \lambda^k,\hat \alpha^k,\hat u_0^k,\hat u^k)_k $ and $\hat \theta$ arbitrary such that $e(\hat\lambda^k,\hat\alpha^k,\hat u_0^k,\hat u^k,\hat\theta) = 0 $ and $ M\hat u^k = (y^\dagger)^k$, and $((\lambda_n^k,\alpha_n^k,(u_0^k)_n,u_n^k)_k,\theta_n)_n$ any sequence of solutions to \eqref{eq:parametrized_learning_problem} with parameters $\beta^e_n,\beta^M_n$, \commA{by optimality} it holds that 
\begin{multline} \label{eq:convergence_main_estimate}
\sum_{k=1}^K \beta^e_n\| e(\lambda^k_n,\alpha^k_n,(u_0^k)_n,u^k_n,\theta_n) \|^2_{\Wc\times H} + \beta^M_n \| Mu_n^k-y_n^k \|^2_{\Yc}   + \Rc_1(\lambda_n^k,\alpha_n^k,(u_0^k)_n,u^k_n)  +  \Rc_2(\theta_n)
 \\ \leq \commB{\sum_{k=1}^K }\Rc_1(\hat \lambda^k,\hat \alpha^k,\hat u_0^k,\hat u^k)  + \beta_n^M\delta_n^2 +  \Rc_2(\hat \theta) 
\end{multline}
By weak precompactness of the sublevel sets of  $\Rc_1$ and $\Rc_2$ and convergence of $\beta_n^M\delta_n^2$ to zero it thus follows that $((\lambda_n^k,\alpha_n^k,(u_0^k)_n,u_n^k)_k,\theta_n)_n$ admits a weakly convergent subsequence \commB{in $(X \times \R^m \times U_0\times\Vc)^K\times\Theta$}.

Now let $((\lambda^k,\alpha^k,u_0^k,u^k)_k,\theta)$ be the limit of such a weakly convergent subsequence, which we again denote by $((\lambda_n^k,\alpha_n^k,(u_0^k)_n,u_n^k)_k,\theta_n)_n$. Closedness of $\Gc$ together with  \commB{lower semi-continuity of the norm $\|\cdot \|_{\Wc\times H}$ } and the estimate \eqref{eq:convergence_main_estimate} (possibly moving to another non-relabeled subsequence) then yields that both
\begin{align*}
\sum_{k=1}^K \| e(\lambda^k,\alpha^k,u_0^k,u^k,\theta) \|^2_{\Wc\times H} 
& \leq \liminf_n \sum_{k=1}^K \| e(\lambda^k_n,\alpha^k_n,(u_0^k)_n,u^k_n,\theta_n) \|^2_{\Wc\times H} \\
& \leq  \liminf _n \commB{\sum_{k=1}^K }\Rc_1(\hat \lambda^k,\hat \alpha^k,\hat u_0^k,\hat u^k)/\beta^e_n +  \beta_n^M(\delta_n^2 /\beta^e_n) +  \Rc_2(\hat \theta)/\beta^e_n   = 0
\end{align*} 
and
\begin{multline} 
\| Mu^k-(y^\dagger)^k \|^2_{\Yc} \leq \liminf_n 
\| Mu_n^k-y_n^k \|^2_{\Yc}
 \leq \liminf _n \Rc_1(\hat \lambda^k,\hat \alpha^k,\hat u_0^k,\hat u^k)/\beta^M_n  +  \Rc_2(\hat \theta)/\beta^M_n + \delta_n^2  = 0.
\end{multline}
This shows that $e(\lambda^k,\alpha^k,u_0^k,u^k,\theta) = 0 $ and $ Mu^k = (y^\dagger)^k$ for all $k$. Again using the estimate \eqref{eq:convergence_main_estimate}, now together with weak lower semi-continuity of $\Rc_1,\Rc_2$, we further obtain that 
\begin{align*}
\Rc_1(\lambda^k,\alpha^k,u_0^k,u^k)  +  \Rc_2(\theta) &\leq
 \liminf_n \Rc_1(\lambda_n^k,\alpha_n^k,(u_0^k)_n,u_n^k)  +  \Rc_2(\theta_n)
\\ & \leq \liminf_n   \Rc_1(\hat \lambda^k,\hat \alpha^k,\hat u_0^k,\hat u^k)  +  \Rc_2(\hat \theta) + \beta_n^M\delta_n^2  \\
& =  \Rc_1(\hat \lambda^k,\hat \alpha^k,\hat u_0^k,\hat u^k)  +  \Rc_2(\hat \theta).
\end{align*}
Since $(\hat \lambda^k,\hat \alpha^k,\hat u_0^k,\hat u^k)_k $ and $\hat \theta$ were arbitrary solutions of $e(\hat\lambda^k,\hat\alpha^k,\hat u_0^k,\hat u^k,\hat\theta) = 0 $ and $ M\hat u^k = (y^\dagger)^k$, it follows that $((\lambda^k,\alpha^k,u_0^k,u^k)_k,\theta)$ solves \eqref{eq:convergence_r_minimizing} as claimed.

At last, in case the solution to  \eqref{eq:convergence_r_minimizing} is unique, weak convergence of the entire sequence follows by a standard argument, using that any subsequence contains another subsequence that weakly converges to the same limit.
\end{proof}
\end{prop}
\begin{rem}[Different limit cases] The above result considers the limit case of both fulfilling the PDE exactly and matching noise-free ground truth measurements. Variants can be easily obtained as follows: In case only the PDE should be fulfilled exactly, one can consider $\beta^M$ fixed and only $\beta^e$ converging to infinity (at an arbitrary rate), such that the resulting limit solution will be a solution of the reduced setting. Likewise, one can consider the case that $\beta^e$ is fixed and $\beta^M$ converges to infinity appropriately in dependence of the noise level $\delta$, in which case the limit solutions solves the all-at-once setting with the hard constraint $Mu^k=(y^\dagger)^k$, see \cite{holler18coupled_mh} for some general results in that direction. The corresponding assumption of existence of $((\hat \lambda^k,\hat \alpha^k,\hat u_0^k,\hat u^k)_k,\hat \theta)$ such that $e(\hat \lambda^k,\hat \alpha^k,\hat u_0^k,\hat u^k,\hat \theta) = 0 $ and $ 
M\hat u^k = (y^\dagger)^k$ can be weakened in both cases accordingly. 

Further, note that the convergence result as well as its variants can be deduced also for the learning-informed parameter identification problem \eqref{eq:standard_pid_aao_setting} exactly the same way.
\end{rem}

\begin{rem}[Uniqueness of minimum-norm solution.]\label{rem:unique} A sufficient condition for uniqueness of a minimum-norm solution, and thus for convergence of the entire sequence of minimizers as stated in Proposition \ref{prop-convergence}, is the tangential cone condition and existence of a solution $(\hat{\lambda}^k,\hat{u}_0^k,\hat{u}^k)$ to the  PDE such that $Mu^k = (y^\dagger)^k$, see \cite[Proposition 2.1]{KalNeuSch08}. In Section \ref{sec:tcc} below, we discuss this condition in more detail and provide a result which, together with Remark \ref{lem-TCC}, ensures this condition to hold for some particular choices of $F$  and $\Nc_\theta$. Regarding solvability of the PDE, we refer to Proposition \ref{prop-ex0-PDEexistence} below, where a particular application is considered.
\end{rem}

%\todo{TN: duplicated indices $M$ for number of data and observation operator. Changed number of data to $K$.}

%\todo{As Proposition \ref{prop-convergence} mentions solvability of the learning-informed PDE, should we briefly refer to proposition ?}

\subsection{Differentiability of the forward operator}

%Solution methods for nonlinear optimization problems related to inverse problems, like gradient descent or Newton-type methods, require structural assumptions on the nonlinear forward operator $\Gc$ in order to ensure convergence. These are uniform boundedness of the derivative of the forward operator and the tangential cone condition.
%These crucial conditions enforce certain local convexity of the model residual guaranteeing uniqueness of a minimum-norm solution (see \cite[Proposition 2.1]{KalNeuSch08} and, furthermore, ensure convergence of iterative methods \cite{Sche95,KalNeuSch08}. 

Solution methods for nonlinear optimization problems, like gradient descent or Newton-type methods, require uniform boundedness of the derivative of $\Gc$. Differentiability of $\Gc$ is a question of differentiability of $F$ and $\Nc$, which is discussed in the following. Note that there, and henceforth, we denote by $H'(a): A \rightarrow B$ the \commB{G\^ateaux} derivative of a function $H:A \rightarrow B$ and define \commB{G\^ateaux} differentiability in the sense of \cite[Section 2.6]{Troeltzsch_optimal_control_pde_mh}, i.e., require $H'(a)$ to be a bounded linear operator. The basis for differentiability of the forward operator is the following lemma, which is a direct extension of \cite[Lemma 4.12]{Troeltzsch_optimal_control_pde_mh}.

\begin{lem} \label{lem:differentiability_general}
Let $A,B,S$ be Banach spaces such that $A \embed S$. For $\Sigma \subset \R^N$ open and bounded, and $r \in [1,\infty)$, let $\Ac$, $\Bc$ be Banach spaces such that $\Ac \embed L^r(\Sigma,A)$ and $\Ac \embed L^\infty(\Sigma,S)$, and \commB{$L^r(\Sigma,B) \embed \Bc$}. Further, let $H:\Sigma \times A \rightarrow B$ be a function such that $H(z,\cdot)$ is \commB{G\^ateaux} differentiable for every $z \in \Sigma$ with derivative $H'(z,\cdot)$, and such that $H$ is locally Lipschitz continuous in the sense that, for any $M>0$ there exists $L(M)>0$ such that for every $a,\xi \in A$ with $\max\{\|a\|_S,\|\xi\|_S\} \leq M$
\begin{equation} \label{eq:differentiability_general_lipschitz_estimate}
 \|H(z,a) - H(z,\xi)\|_B \leq L(M)(\|a-\xi\|_A + (\max \{\|a\|_A,\|\xi\|_A \}+1) \|a-\xi\|_S)   .
 \end{equation}
Then, if the Nemytskii operators $H:\Ac \rightarrow \Bc$ given as $H(a)(z) = H(z,a(z))$ and $H':\Ac \rightarrow \Lc(\Ac,\Bc)$ given as $H'(a)(\xi)(z) = H'(z,a(z))(\xi(z))$ are well defined, then $H:\Ac \rightarrow \Bc$ is also \commB{G\^ateaux} differentiable with $H'(a) \in \mathcal{L}(\Ac,\Bc)$  given as $H'(a)(\xi)(z) = H'(z,a(z))(\xi(z))$. 
Further, $H'$ is locally bounded in the sense that, for any bounded set $\tilde{\Ac} \subset \Ac$, 
$ \sup_{a \in \tilde{A}} \|H'(a)\|< \infty.
$
\begin{proof}
Fix $M>0$ and $z \in \Sigma$. Local Lipschitz continuity implies for any $\tilde{a},\xi \in A$ with $\|\tilde{a}\|_S+1\leq M$,
\begin{equation} \label{eq:abstract_gateaux_lipschitz_estimate}
 \|H'(z,\tilde{a})\xi\|_{\commA{B}} = \lim_{\delta \rightarrow 0 } \left\| \frac{H(z,\tilde{a}+\delta \xi) - H(z,\tilde{a})}{\delta} \right\|_B \leq L(M)(\|\xi\|_A + (\|\tilde{a}\|_A + 2) \|\xi\|_S) .
\end{equation}
Next, define $h:[0,1] \rightarrow B$ as $h(s) = H(z,a + \epsilon s \xi)$, for $a \in A$ and $\epsilon\in (0,1)$ such that $\|a\|_S+2  \leq M$, $\epsilon\|\xi\|_S \leq 1$. We note that $h$ is differentiable and Lipschitz continuous (hence absolutely continuous), such that by the fundamental theorem of calculus for Bochner spaces, see \cite[Theorem 2.2.17]{Gasinski_nonlinear_analysis_mh},
%\begin{equation} \label{eq:diff_fundamental_thm}
$h(1) - h(0) = \int_0^1 h'(s) \wrt s.$
%\end{equation} 
This yields
\begin{multline*}
\left( \frac{1}{\epsilon}\|H(z,a+\epsilon\xi)-H(z,a)-\epsilon H'(z,a)\xi \|_B\right) ^r
=\frac{1}{\epsilon^r}\left\|\int_0^1  \epsilon H'(z,a + s \epsilon \xi) \xi -\epsilon H'(z,a)\xi \commA{\wrt s} \right\|_B^r \\
\commA{\leq \left(\int_0^1 \left\| H'(z,a + s \epsilon \xi) \xi \right \| _B+ \left \| H'(z,a)\xi\right\|_B \wrt s \right)^r  \leq \left(\int_0^1  \sup_{\tilde{s} \in [0,1]} 2\left\| H'(z,a + \tilde{s} \epsilon \xi) \xi \right \| _B \wrt s \right)^r }\\
\leq \sup_{\tilde{s}\in[0,1]} 2^r \|H'(z,a+\tilde{s}\epsilon\xi) \xi\|_B^r
\leq 2^{2r-1}L(M)^r (\|\xi\|_A^r + (\|a\|+\|\xi\|_A+2)^r\|\xi\|_S^r).
\end{multline*}

Now by \commB{$\Ac \embed L^\infty(\Sigma,S)$}, for $a,\xi \in \Ac$, we can apply the above with $M:= \sup_{z \in \Sigma}\|a(z)\|_S+2$ and $\epsilon$ sufficiently small such that $\epsilon \sup_{z \in \Sigma}\|\xi(z)\|_S \leq 1$ and obtain
\begin{multline*}
r_H(\epsilon):=
 \int_\Sigma  \left(\frac{1}{\epsilon} \|H(z,a(z)+\epsilon\xi(z))-H(z,a(z))-\epsilon H'(z,a(z))\xi(z) \|_B \right) ^r\wrt z\\
\leq \int_\Sigma 2^{2r-1}L(M)^r (\|\xi(z)\|_A^r + (\|a(z)\|_A + \|\xi(z)\|_A+2)^r\|\xi(z)\|_S^r) dz \\
 \leq 2^{2r-1}L(M)^r\left(  \|\xi\|_\Ac^r  + \sup_{z\in \Sigma} \|\xi(z)\|_S^r \int_\Sigma (\|a(z)\|_A + \|\xi(z)\|_A + 2)^r \right) < \infty
\end{multline*}
Using the Lebesgue’s Dominated Convergence Theorem, we deduce $\lim_{\epsilon \rightarrow 0} r_H(\epsilon)=0$, which, by \commB{$L^r(\Sigma,B) \embed \Bc $}, shows \commB{G\^ateaux} differentiability. 

Local boundedness as claimed follows direct from choosing $M:= \sup_{a \in \tilde{\Ac}}\sup_{t \in (0,t)}\|a(t)\|_S+1$, and  integrating the $r$th power of \eqref{eq:abstract_gateaux_lipschitz_estimate} over time.
\end{proof}
\end{lem}

\begin{prop}[Differentiability]\label{prop-Differentiability}
Let Assumption \ref{ass:general_basic} hold and let \commA{\commB{$\sigma\in\Cc^1(\R,\R)$}}. 
Assume that for every $t \in (0,T)$, the mapping $F(t,\cdot,\cdot):X \times V \rightarrow W$ is jointly \commB{G\^ateaux} \commA{differentiable} with respect to  the second  and  third  arguments, with $(t,\lambda,u,\xi,v) \mapsto F'(t,\lambda,u)(\xi,v)$ satisfying the Carathéodory conditions. 

%  with continuous derivative $F'(t,\cdot)$ for  almost all $t\in(0,T)$.\\
In addition, assume that $F$ satisfies the following local Lipschitz continuity condition:
For all $ M\geq 0$ there exists $L(M)>0$, such that for all  $v_i \in V$ and $\lambda_i \in X$, $i=1,2$, with  $\max\{\|v_i\|_H, \|\lambda_i\|_X\} \leq M$ and for almost every $t \in (0,T)$,
\begin{equation}\label{cond-Lipschitz}
\|F(t,\lambda_1,v_1)-F(t,\lambda_2,v_2)\|_W
\leq L(M) (\|v_1-v_2\|_V+(\max\{\|v_1\|_V,\|v_2\|_V\} + 1)( \|v_1-v_2\|_H +\|\lambda_1-\lambda_2\|_X)).
\end{equation}

Then $\G:X \times \R^m \times U_0 \times\Vc \times \Theta\rightarrow \Wc\times H \times \Yc$ is \commB{G\^ateaux} differentiable with
\begin{align*}
\G'(\lambda,\alpha,u_0,u,\theta)=
\begin{pmatrix}
-F'_\lambda(\cdot,\lambda,u) & -\Nc'_\alpha(\alpha,u,\theta) & 0 & \frac{d}{dt}-F'_u(\cdot,\lambda,u) - \Nc'_u(\alpha,u,\theta) & -\Nc'_\theta(\alpha,u,\theta)\\
0 & 0 & -\text{Id} & (\cdot)_{t=0} & 0\\
0 & 0 & 0 & M & 0
\end{pmatrix}.
\end{align*}
Furthermore, $\Gc'(\cdot)$ is locally bounded in the sense specified in Lemma \ref{lem:differentiability_general}.
\end{prop}
\begin{proof}
First note that it suffices to show corresponding differentiability and local boundedness assertions for the different components of $\Gc$ given as $u \mapsto \dot{u}$, $F$, $\Nc$, $(u,u_0) \mapsto u(0) - u_0$ and $M$. For all except $F$ and $\Nc$, the corresponding assertions are immediate, hence we focus on the latter two.

Regarding $F$, this is an immediate consequence of Lemma \ref{lem:differentiability_general} with $A = X \times V$, $B = W$, $S = X \times H$, $\Sigma = (0,T)$, $r = 2$, $\Ac = X \times \Vc $ with $\|(\lambda,v)\|_\Ac = \|\lambda\|_X + \|v\|_\Vc$, $\Bc = \Wc$ and $H(t,(\lambda,v)) = F(t,\lambda,v)$. 

For $\Nc$, this is again an immediate consequence of Lemma \ref{lem:differentiability_general} with $A =S= \R^m \times \R \times \Theta$, $B = \R$, $\Sigma = (0,T) \times \Omega$, $r=\max\{2,\pfix\}$, $\Ac = \R^m \times \Vc \times \Theta$ with $\|(\alpha,v,\theta)\|_\Ac = |\alpha| + \|v\|_\Vc + |\theta|$, $\Bc = \Wc$ and $H((t,x),(\alpha,v,\theta)) = \Nc_{\theta}(\alpha,v(t,x))$.
\end{proof}

\begin{rem}\label{rem-diff-C2activation}
For stronger image spaces $W\nsupseteq L^q\om, \forall q\in[1,\infty)$, differentiability of $F$ remains valid if \eqref{cond-Lipschitz} holds, while differentiability of $\Nc$ requires a smoother activation function, e.g.,  the one suggested in Remark \ref{rem-C2activation} below.
\end{rem}

\subsection{Lipschitz continuity and the tangential cone condition} \label{sec:tcc}

In this section, we focus on showing a rather strong Lipschitz-type result for the neural network. This property allows us to apply (finite-dimensional) gradient-based algorithms to learn the neural networks, where the Lipschitz constant and its derivatives are used to determine the step size. Moreover, by this Lipschitz continuity, the tangential cone condition on \eqref{eq:standard_pid_aao_setting} can be verified. This condition, together with solvability of the learning-informed PDE, answers the important question of uniqueness of a minimizer to \commE{the limit case of} \eqref{eq:standard_pid_aao_setting}, \commE{as mention in Remark \ref{rem:unique}}.

For ease of notation, we assume in this lemma that the outer layer of the neural network has activation $\sigma$, as in the lower layers. Adapting the proof for $\sigma=\text{Id}$ in the last layer is straightforward. 

\begin{lem}[Lipschitz properties of neural networks]\label{NN-Lipschitz}

Consider an $L$-layer neural network \commA{$\Nc: \R^{m+1}\times\Theta \ni (z,\theta) \mapsto \Nc_\theta((z_1,\ldots,z_m),z_{m+1})\in\R$}, $L\in\N$ ($z$ taking the role of $(\alpha,u(t,x))$ in Lemma \ref{lem-nemytskii-N}). Denote by $\Nc^{i}_{\theta^{i}}$ the $i$ lowest layers of the neural network, depending only on $z$ and on the $i$ lowest-index pairs of parameters $\theta^{i}$, while $\Nc^0_{\theta^0}(z):=z\in\R^{m+1}$.

Fix any subset $\Bc\subseteq\R^{m+1}\times\Theta$. \commB{For each $1\leq i\leq L$, define $\Bc_i:=\{\omega^i\Nc^{i-1}_{\theta^{i-1}}(z)+\beta^i \mid (z,\theta)\in\Bc)\}$, that is, the image of the $i$-th layer before applying the activation function.} Assume that the activation function $\sigma\in\Cc^1(\R,\R)$ associated to $\mathcal{N}$ \commB{for all $1\leq i\leq L$ satisfies the Lipschitz inequalities}
\[
\left|\sigma(x) - \sigma(\tilde{x})\right| \leq C_\sigma \left|x - \tilde{x}\right|, \qquad
\left|\sigma'(x) - \sigma'(\tilde{x})\right| \leq C'_\sigma \left|x - \tilde{x}\right|
\]
for all $x$, $\tilde{x}\in\commB{\Bc_i}$ and some positive constants $C_\sigma$, $C'_\sigma$, and that $s_i:=\sup_{x\in \commB{\Bc_i}}\left|\sigma'(x)\right| < \infty$ .

Fix now a layer $l$, $1\leq l\leq L$, as well as $(\tilde{z},\theta)$, $(z,\bar{\theta})$, $(z,\hat{\theta})\in \Bc$, where $\bar{\theta}$ differs from $\theta$ only in that its $l$-th weight is replaced by some $\tilde{\omega^l}$ and $\hat{\theta}$ differs from $\theta$ only in that its $l$-th bias is replaced by some $\tilde{\beta^l}$; explicitly,
$$
(\bar{\theta}_j)_k = \begin{cases}\tilde{\omega}^l, & (j,k) = (1,l), \\ (\theta_j)_k & \text{otherwise,}\end{cases} \qquad (\hat{\theta}_j)_k = \begin{cases}\tilde{\beta}^l, & (j,k) = (2,l), \\ (\theta_j)_k & \text{otherwise.}\end{cases}
$$
Then $\Nc$ satisfies the Lipschitz estimates
\begin{align}
\begin{split}\label{lipschitz_estimates_nn}
\left|\Nc(z,\theta) - \Nc(\tilde{z},\theta)\right| & \leq (C_\sigma)^L \left(\prod_{k=1}^L\left|\omega^k\right|\right)\left|z - \tilde{z}\right|, \\
\left|\Nc(z,\theta) - \Nc(z,\bar{\theta})\right| & \leq (C_\sigma)^{L-l+1} \left(\prod_{k=l+1}^L\left|\omega^k\right|\right)\left|\Nc^{l-1}_{\theta^{l-1}}(z)\right|\left|\omega^l - \tilde{\omega}^l\right|, \\
\left|\Nc(z,\theta) - \Nc(z,\hat{\theta})\right| & \leq (C_\sigma)^{L-l+1}\left|\beta^l - \tilde{\beta}^l\right|,
\end{split}
\end{align}
while its derivatives with regards to $z$, $\omega^l$ and $\beta^l$, respectively, satisfy the Lipschitz estimates
\begin{align}
\left|\Nc'_z(z,\theta) - \Nc'_z(\commA{\tilde{z}},\theta)\right| & \leq C^z_1\left|\omega^1\right|\left|z - \tilde{z}\right|, \label{lipschitz_estimates_nn_derivative:z} \\
\left|\Nc'_{\omega^l}(z,\theta) - \Nc'_{\omega^l}(z,\bar{\theta})\right| & \leq C^{\omega^l}_l\left|\Nc^{l-1}_{\theta^{l-1}}(z)\right|\left|\omega^l - \tilde{\omega}^l\right|, \label{lipschitz_estimates_nn_derivative:w} \\
\left|\Nc'_{\beta^l}(z,\theta) - \Nc'_{\beta^l}(z,\hat{\theta})\right| & \leq C^{\beta^l}_l\left|\beta^l - \tilde{\beta}^l\right|, \label{lipschitz_estimates_nn_derivative:b}
\end{align}
where one defines $C^z_{L+1}:=C^{\omega^l}_{L+1}:=C^{\beta^l}_{L+1}:=0$ and, by \commE{backward} recursion for $1\leq i\leq L$,
\begin{align}
\begin{split}\label{lipschitz_constant_derivative}
C^z_i & := C'_\sigma (C_\sigma)^{i-1}\left(\prod_{k=i+1}^Ls_k\right)\left(\prod_{k=1}^{L}\left|\omega^k\right|\right) +
C^z_{i+1}s_i\left|\omega^{i+1}\right|, \\
C^{\omega^l}_i & := C'_\sigma (C_\sigma)^{i-l}\left(\prod_{k=i+1}^Ls_k\right)\left(\prod_{k=l+1}^{L}\left|\omega^k\right|\right)\left|\Nc^{l-1}(z,\theta^{l-1})\right| +
C^{\omega^l}_{i+1}s_i\left|\omega^{i+1}\right|, \\
C^{\beta^l}_i & := C'_\sigma (C_\sigma)^{i-l}\left(\prod_{k=i+1}^Ls_k\right)\left(\prod_{k=l+1}^{L}\left|\omega^k\right|\right) +
C^{\beta^l}_{i+1}s_i\left|\omega^{i+1}\right|.
\end{split}
\end{align}
\end{lem}

\begin{proof}
\commA{See Appendix \ref{appendix-proof-NN}.}
\end{proof}

\begin{rem}\label{rem-Lip}
%\begin{itemize}
%\item 
%Note that if the outer layer of the neural network is taken to be the identity mapping, one requires $A_L:=1$, $C^u_L:=C^{\omega_l}_L:=C^{\beta_l}_L:=0$; however, the induction does not change.
%\item 
\commB{If $\sigma'$ is locally Lipschitz continuous on $\R$, the existence of $C_\sigma$, $C'_\sigma$ and the $s_i$ is clear whenever $\Bc$ is a bounded set. Thus,} it is a direct consequence of Lemma \ref{NN-Lipschitz} (or follows simply by the properties of the functions $\Nc$ is composed of) that the mapping $(z,\theta) \mapsto \Nc(z,\theta)$ restricted to any bounded set is bounded, Lipschitz continuous and has Lipschitz continuous derivative. This is relevant for gradient-based optimization algorithms to solve the learning problem \eqref{eq:main_identification_aao_setting}, where Lipschitz continuity of the derivative of the objective function is a key ingredient for (local) convergence, see for instance \cite{smyrlis2004local_convergence_gradient} for a result in Hilbert spaces. 
In particular, Lipschitz continuity of $\theta \mapsto \Nc(z,\theta)$ for $z$ fixed is useful for the learning problem \eqref{eq:main_identification_aao_setting}, where the exact $(\lambda,u)$ is known. In this case, one simply learns the finite-dimensional hyperparameter $\theta$, thus standard convergence results on gradient-based methods in finite dimensional vector spaces apply, see, e.g., \cite[Section 5.3]{ruszczynski2011nonlinear}.
\end{rem}

Based on these Lipschitz estimates, we can study the tangential cone condition for the problem \eqref{eq:standard_pid_aao_setting}, given a learned $\Nt$. For this, we assume that $\Nt(\alpha,u)=\Nt(u)$.
\begin{cond}[Tangential cone condition {\cite[Expression (2.4)]{KalNeuSch08}}]\label{ass-tcc}
We say that the tangential cone condition for a mapping \commA{$G:\mathcal{D}(G)(\subseteq X)\to Y$} holds in a ball $\Bc^X_\rho(x^\dagger)$, if there exists $c_{tc}<1$ such that
\[\|G(x)-G(\tilde x)-G'(x)(x-\tilde x)\|_{\commA{X}}\leq c_{tc}\|G(x)-G(\tilde x)\|_{\commA{Y}}\qquad \forall x,\tilde  x\in \Bc_\rho(x^\dagger).\]
\commA{Here, $G'(x)h$ denotes the directional derivative \cite{TCC21}.}
\end{cond}
Analyzed in the all-at-once setting \eqref{eq:standard_pid_aao_setting}, the tangential cone condition reads as
\begin{multline}\label{tcc}
\|F(\lambda,u)-F(\tilde \lambda ,\tilde u )-F_\lambda'(\lambda,u)(\lambda-\tilde \lambda)-F_u'(\lambda,u)(u-\tilde u)  +\Nt(u)-\Nt(\tilde u)-\Nc_\theta'(u)(u-\tilde u)\|_\Wc \\
\leq c_{tc} \Bigl(
\|\dot{u}-\dot{\tilde u}-F(\lambda,u)+F(\tilde \lambda,\tilde u)-\Nt(u)+\Nt(\tilde u)\|_\Wc^2 +\|u(0)-u^\dagger(0)-u_0+\tilde u_0\|_{U_0}^2
+\|M(u-\tilde u)\|_\Yc^2\Bigr)^{1/2}
\end{multline}
for all $(\lambda,u_0,u), (\tilde \lambda,\tilde u_0,\tilde u)\in B_\rho^{X \times U_0\times\Vc}(\lambda^\dagger,u_0^\dagger,u^\dagger)$, \commA{ where $F'$ and $\Nc' $ are the \commB{G\^ateaux} derivatives}.

%\begin{ass}[Tangential cone condition]\label{ass-tcc}
%The tangential cone condition for $G(x)=y$, which is solvable at $\tilde{x}\in\Bc_\rho(x^\dagger)$ %\todo{in  $Bc_\rho(x^\dagger)$?} \commA{yes. added!}, 
%of the original form \cite[Expression (2.4)]{KalNeuSch08}
%\[\|G(x)-G(x^\dagger)-G'(x)(x-x^\dagger)\|\leq c_{tc}\|G(x)-G(x^\dagger)\|\qquad c_{tc}<1,\forall x\in \Bc_\rho(x^\dagger)\]
%analyzed in the all-at-once setting \eqref{eq:standard_pid_aao_setting} reads as
%\begin{align}\label{tcc}
%&\|F(\lambda,u)-F(\lambda^\dagger ,u^\dagger )-F_\lambda'(\lambda,u)(\lambda-\lambda^\dagger)-F_u'(\lambda,u)(u-u^\dagger) \nonumber\\
%&\qquad+\Nt(u)-\Nt(u^\dagger)-\Nc_u'(u)(u-u^\dagger)\|_\Wc \nonumber\\
%&\leq c_{tc} \Bigl(
%\|\dot{u}-\dot{u}^\dagger-F(\lambda,u)+F(\lambda^\dagger,u^\dagger)-\Nt(u)+\Nt(u^\dagger)\|_\Wc^2 \\\
%&\qquad+\|u(0)-u^\dagger(0)-u_0+u^\dagger_0\|_{U_0}^2
%+\|M(u-u^\dagger)\|_\Yc^2\Bigr)^{1/2} \nonumber\\
%&\hspace{6cm} c_{tc}<1,\forall (\lambda,u_0,u)\in B_\rho^{X \times U_0\times\Vc}(\lambda^\dagger,u_0^\dagger,u^\dagger),\nonumber
%\end{align}
%where $\Bc_\rho(x^\dagger)$ is a ball of radius $\rho$ and center $x^\dagger$. Here, $F$ only needs to be the Gateaux derivative; the same is assumed for $\Nc$.
%\end{ass}
The tangential cone condition strongly depends on the PDE model $F$ and the architectures of $\Nc$. By triangle inequality, a sufficient condition for \eqref{tcc} to hold is that the tangential cone condition holds for $F$ and for $\Nc$ separately. %, see \cite{TCC21} for a detail study of this condition on different nonlinear PDE models $F$. 
The \commB{tangential} cone condition in combination with solvability of equation $G(x)=0$ ensures uniqueness of a minimum-norm solution \cite[Proposition 2.1]{KalNeuSch08} \commA{(see Appendix \ref{appendix-Roubicek})}. Solvability of the operator equation $G(x)=0$, according to the all-at-once formulation, is the question of solvability of the learning-informed PDE and exact measurements, i.e. $\delta=0$.  For solvability of the learning-informed PDE, we refer to Proposition \ref{prop-ex0-PDEexistence} in Section \ref{sec:application}. %\todo{don't we mix here different notations of solvability of the pde: once the parameters are fixed, once not. also the measurement operator is involved only in one case.}. 
In the following, we focus on the tangential cone condition for the neural networks \commA{by studying Condition \ref{ass-tcc} for $G:=\Nt.$}

\begin{lem}[Tangential cone condition for neural networks]\label{tcc-nn}
The tangential cone condition in Condition \ref{ass-tcc} for \commA{$G = \Nc_\theta\commB{: \Vc\to\Wc}$ with fixed parameter $\theta$} holds in any ball $\B^\Vc_\rho (u^\dagger) $ if $M=\text{Id},$ $ Y\embed L^\pfix(\Omega)$ with $\pfix>0$ as in \eqref{eq:ass:static_space_embeddings}, $\sigma\in\Cc^1(\R,\R)$ and $\rho$, depending on the Lipschitz constant in Lemma \ref{NN-Lipschitz} is sufficiently small.
\end{lem}
\begin{proof}
Since $\Vc\embed L^\infty((0,T)\times\Omega)$ for $u, \tilde u \in\Bc_\rho^\Vc(u^\dagger)$, we have for almost all $(t,x)\in(0,T)\times\Omega$ that $u(t,x), \tilde  u(t,x)\in\Bc$ for some $\Bc$ bounded. Thus, we can use Lemma \ref{NN-Lipschitz}
with such a $\Bc$, and in particular the estimate \eqref{lipschitz_nn:z} for $z=u(t,x)$, to obtain
\begin{align*}
&\|\Nt(u)-\Nt( \tilde u)-\Nt'(u)(u- \tilde u)\|_\Wc =\left\|\int_0^1 \Nt'( \tilde u+\mu(u- \tilde u))\,d\mu(u- \tilde u)-\Nt'(u)(u- \tilde u)\right\|_\Wc\\
&\quad\leq C_{L^\pfix\to W}\left\|\int_0^1 \left(\Nt'( \tilde u+\mu(u- \tilde u))-\Nt'(u)\right)d\mu(u- \tilde u)\right\|_{L^2(0,T;L^\pfix(\Omega))}\\
&\quad\leq C_{L^\pfix\to W}C^z_1|\omega_1|\left\|\int_0^1 (1-\mu)\,d\mu|u- \tilde u|^2\right\|_{L^2(0,T;L^\pfix(\Omega))}\\
&\quad \leq (1/2)C_{Y\to L^\pfix(\Omega)}C_{L^\pfix\to W} C^z_1|\omega_1|\|u- \tilde u\|_{L^\infty((0,T)\times\Omega)}\|u- \tilde u\|_{\Yc}\\
&\quad \leq \rho\, C_{\Vc\to L^\infty((0,T)\times\Omega)} C_{Y\to L^\pfix(\Omega)}C_{L^\pfix(\Omega)\to W}  C^z_1|\omega_1|\|u- \tilde u\|_{\Yc}=: c_{tc}\|u- \tilde u\|_{\Yc}\\
&\quad= \commB{c_{tc}\|M(u- \tilde u)\|_{\Yc}}
\end{align*}
where $C^z_1|\omega_1|$ is the Lipschitz constant of $\Nc'_u$ derived in Lemma \ref{NN-Lipschitz}, and $c_{tc}<1$ if \[\rho<1/\left(C_{\Vc\to L^\infty((0,T)\times\Omega)} C_{Y\to L^\pfix(\Omega)}C_{L^\pfix(\Omega)\to W}  C^z_1|\omega_1|\right).\]
\commB{We note that having full observation, i.e. $M=\text{Id}$, is crucial for establishing the tangential cone condition, as it allows us to link the estimate from $\|u- \tilde u\|_\Yc$ to $\|M(u- \tilde u)\|_\Yc$, yielding the last quantity on the right hand side of \eqref{tcc}. The necessity of full observation has also been mentioned in \cite{TCC21}.} 
\qedhere 
\end{proof}
%\note{adapted to use general setting, original version is in comments...}
Now using \cite[Proposition 2.1]{KalNeuSch08} together with Lemma \ref{tcc-nn}, a uniqueness result follows.
\begin{prop}[Uniqueness of minimizer for the \commE{limit case of} \eqref{eq:standard_pid_aao_setting}]\label{unique-min}
\commB{With $\nu\geq1$}, consider the regularizer \commB{$\Rc_1=\|\cdot\|^\nu_{X\times\R^m \times U_0\times \Vc}$}, and assume that the conditions in Lemma \ref{tcc-nn} are satisfied. Moreover, suppose that the tangential cone condition for $F$ holds in $\Bc_\rho^{X\times U_0\times \Vc}(\lambda^\dagger,u_0^\dagger,u^\dagger)$ and the equation $\Gc(\lambda,u_0,u,\hat{\theta}) = 0$ with $\Gc$ in \eqref{eq:forward_operator_general} and $\theta $ fixed is solvable in $\Bc_\rho^{X\times U_0\times\Vc}(\lambda^\dagger,u_0^\dagger,u^\dagger)$. 
Then the \commE{limit case of the} parameter identification problem \eqref{eq:standard_pid_aao_setting} admits a unique minimizer in the ball $\Bc_\rho^{X\times U_0\times\Vc}(\lambda^\dagger,u_0^\dagger,u^\dagger)$.
%\todo{agree?}
\end{prop}

\begin{rem}\label{lem-TCC}
We refer to Section \eqref{sec:application} below for solvability of the learning-informed PDE in an application. We refer to \cite{TCC21} for concrete choices of $F$ and of function space settings such that the tangential cone condition can be verified.

Note that, while the tangential cone condition for \commE{limit case of the} of the parameter identification problem \eqref{eq:standard_pid_aao_setting} can be confirmed as above, the same question for the learning problem \eqref{eq:main_identification_aao_setting} remains open.
\end{rem}

%\begin{prop}[Uniqueness of minimizer for \eqref{eq:standard_pid_aao_setting}]\label{unique-min}
%Given the regularizer $\Rc=\|\cdot\|$ and assume that conditions in Lemma \ref{tcc-nn} are satisfied. Moreover, suppose that the tangential cone condition for $F$ holds in $\Bc_\rho^{X\times U_0\times \Vc}(\lambda^\dagger,u_0^\dagger,u^\dagger)$ and the PDE \eqref{app} is solvable in $\Bc_\rho^{X\times U_0\times\Vc}(\lambda^\dagger,u_0^\dagger,u^\dagger)$. 
%Then, the parameter identification problem \eqref{eq:standard_pid_aao_setting} for the application \eqref{app} admits a unique minimizer in the ball $\Bc_\rho^{X\times U_0\times\Vc}(\lambda^\dagger,u_0^\dagger,u^\dagger)$.
%\end{prop}
%\begin{proof}
%\cite[Proposition 2.1]{KalNeuSch08}, where solvability of \eqref{app} is proven in Proposition \ref{prop-ex0-PDEexistence} and tangential cone condition for $\Nt$ is shown in Lemma \ref{tcc-nn}.
%\end{proof}

%\begin{rem}\label{lem-TCC}
%\begin{itemize}
%\item 
%In the application \eqref{app}, the tangential condition for $F$  naturally holds if the unknowns are $u_0$ and $\lambda=\varphi$ due to linearity. If $\lambda=c,a$, $F$ becomes bilinear in $(\lambda,u)$; more care on the choice of $X$ and $\Yc$ is required (see, e.g.\cite{TCC21}).
%\item
%While the tangential cone condition for the parameter identification problem \eqref{eq:standard_pid_aao_setting} can be confirmed as above, the same question for the learning problem \eqref{eq:main_identification_aao_setting} remains open.
%\end{itemize}
%
%\end{rem}

\section{Application}\label{sec:application}
In this section, as special case of the dynamical system \eqref{ori-eq-1}, we examine a class of general parabolic problems given as
\begin{alignat}{3}
& \dot{u} - \nabla\cdot(a\nabla u) + cu - f(\commA{\alpha,u}) = \varphi \quad&&\mbox{ in }\Omega\times\ti,\nonumber\\
& u|_{\partial\Omega}=0 && \mbox{ in } \ti, \label{ex0}\\
& u(0) = u_0 &&\mbox{ in }\Omega,\nonumber
\end{alignat}
where \commB{$\Omega\subset\R^d$} is a bounded $C^2$-class domain, with $ d\in \{1,2,3 \}$ \commB{being relevant in practice}. The nonlinearity $f$, \commA{which can be replaced by a neural network later, 
is assumed to be given as the Nemytskii operator $f:\R^m\times\Vc\to\Wc$ \cite[Section 1.3]{Roubicek} of a pointwise function $f:\R^m\times\R\to\R$, making use of the notation $ [f(\alpha,u)](t,x) = f(\alpha,u(t,x))$.} %\note{ and $f(\alpha,\cdot):\R \rightarrow \R$ is sufficiently smooth such that the induced Nemytskii mappings as used in this section are continuous. OK like this?} \commA{added here and below} \commA{flipped sign of f}
 We initially work with the following parameter spaces
\begin{equation}\label{eq:application_paramter_spaces}
\varphi\in X_\varphi := H^{-1}\om, \quad c\in X_c :=  L^2\om,\quad a\in X_a:=  W^{1,Pa}\om \quad u_0\in U_0:=  H^2(\Omega),
\end{equation}
where $Pa >d$, and, for existence of a solution, we will require the constraints
\begin{equation} \label{application_general_diffusion_constraint}
0<\underline{a}\leq a(x)\leq \overline{a}\ \quad \text{for a.e. } x\in \Omega.
\end{equation}  
Thus, the overall parameter space $X$ is given as $X=(X_\varphi,X_c,X_a)$. %, and the constraint \eqref{application_general_diffusion_constraint} needs to be included in the regularization functional $\Rc_1$.%\todo{adopt this notation consistently in the paper} \commA{done}

\subsection{\commB{Unique existence results for \eqref{ex0}}}\label{sec:unique-existence}
Our next goal is to study unique existence of \eqref{ex0}. The main purpose of this is to inspire a relevant choice of function space setting for the all-at-once setting of \eqref{eq:main_identification_aao_setting} and \eqref{eq:standard_pid_aao_setting}, even \commA{though} unique existence is not required there. Also, a unique existence result is of interest for studying the reduced setting, where well-definedness of the parameter-to-state map is needed.

We will proceed in two steps: In the first step, we prove that \eqref{ex0} admits a unique solution \[u\in W^{1,\infty,\infty}(0,T;L^2\om,L^2\om)\cap W^{1,\infty,2}(0,T;H^1_0\om,H^1_0\om)\]
with $W^{1,p,q}(0,T;V_1,V_2):=\{u\in L^p(0,T;V_1):\dot{u}\in L^q(0,T;V_2)\}$. Then, in the second step, we lift the regularity of $u$ to the somewhat stronger space \[u\in L^\infty\tom\] to achieve boundedness in time and space of the solution, which will later serve our purpose of working with a neural network acting pointwise. \commA{It is worth noting that the study for unique existence is carried out first of all for classes of general nonlinearity $f$ satisfying some specific assumptions, such as pseudomonotonicity and growth condition, see Lemmas \ref{prop:step1} and \ref{prop:step2} below. The nonlinearity $f$ as a neural network will then be considered in Proposition \ref{prop-ex0-PDEexistence}, Remark \ref{rem-Lipactivation}.}

Before investigating \eqref{ex0}, we summarize the unique existence theory as provided in \cite[Theorems 8.18, 8.31]{Roubicek} 
%, which shall benefit our first step mentioned above. This theorem allows for time-dependent source terms; in this section, however, we present the theorem while skipping the time-dependence of $\varphi$. Apart from that, in comparison to the abstract analysis presented in Section \ref{sec-abstractPI}, we here consider the autonomous $\Ftil$.
for the autonomous case.

\begin{thm} \label{theo-Roubicek}  Let $\Vf$ be a Banach space, $\Hf$ be a Hilbert spaces and assume that for  $\Ff:\Vf\to \Vf^*$, $u_0 \in \Hf$ and $\varphi \in \Vf^*$, with the Gelfand triple $\Vf\subseteq \Hf\cong \Hf^* \subseteq \Vf^*$, the following holds:
\begin{enumerate}[label=S\arabic*., ref=S\arabic*]
\item \label{ex0-S-pseu}
$\Ff$ is pseudomonotone. %, i.e. $\Ftil$ is bounded and
%\begin{align*}	
%\left.\begin{aligned}
%u_k \overset{V}{\rightharpoonup} u&\\
%\underset{k\rightarrow\infty}{\limsup} \la \Ftil(u_k),u_k-u\ra_{V^*,V} \leq 0& \\
%\end{aligned}\right\}
%\Rightarrow
%\begin{cases}
%\forall v \in V: \la \Ftil(u),u-v\ra_{V^*,V} \\ \leq \underset{k\rightarrow\infty}{\liminf}\la \Ftil(u_k),u_k-v\ra_{V^*,V}.
%\end{cases}
%\end{align*}

\item \label{ex0-S-coe}
$\Ff$ is semi-coercive, i.e,
\[ \forall v \in \Vf: \la \Ff(v),v\ra_{\Vf^*,\Vf} \geq c_0|v|^2_\Vf -c_1|v|_\Vf-c_2\|v\|_{\Hf}^2 \]
for some $c_0 > 0$ and some seminorm $|.|_{\Vf}>0$ satisfying $\forall v \in \Vf: \|v\|_{\Vf} \leq c_{|.|}(|v|_{\Vf} + \|v\|_{\Hf})$.
\item \label{ex0-S-reg}
$\Ff$, $u_0$ and $\varphi$ satisfy the regularity condition $\Ff(u_0)-\varphi\in \Hf$, $u_0 \in \Vf$ and
%\begin{alignat*}{3}	
%& u_0\in V \quad\text{such that}\quad \Ftil(u_0)-\varphi\in H\\
%& \la \Ftil(u)-\Ftil(v), u-v\ra_{V^*,V}\geq  C_0|u-v|^2_V -C_2\|u-v\|_H^2
%\end{alignat*}
\[ \la \Ff(u)-\Ff(v), u-v\ra_{\Vf^*,\Vf}\geq  C_0|u-v|^2_{\Vf} -C_2\|u-v\|_{\Hf}^2
\]
for all $u,v \in V$ with some $C_0>0.$
\end{enumerate}
Then the abstract Cauchy problem
\[ 
\dot{u}(t) + \Ff(u(t)) = \varphi \, \qquad u(0) = u_0 
\]
has a unique solution $u \in W^{1,\infty,\infty}(0,T;\Hf,\Hf)\cap W^{1,\infty,2}(0,T;\Vf,\Vf)$.
\end{thm}
By verifying the conditions in Theorem \ref{theo-Roubicek}, we now obtain unique existence as follows.

\begin{lem}[\commB{Unique existence}] \label{prop:step1} Let  the nonlinearity $f(\alpha,\cdot):H_0^1(\Omega) \rightarrow H_0^1(\Omega)^*$ be given as the Nemytskii mapping of a measurable function $f(\alpha,\cdot) :\R \rightarrow \R$ that satisfies
\begin{equation}\label{ex0-monotone}
\begin{split}
%\text{\commA{assumption:}} \quad 
&(-f(\alpha,\cdot)):H_0^1(\Omega) \rightarrow H_0^1(\Omega)^* \text{ monotone and continuous}, \\& f(\alpha,0)=0, \quad|f(\alpha,v)|\leq C_\alpha(1+|v|^5), \quad \text{for some } C_\alpha\geq0.
\end{split}
\end{equation} Then, equation \eqref{ex0} with parameter $\varphi$, $c$, $a$ and $u_0$ such that  \eqref{eq:application_paramter_spaces}, \eqref{application_general_diffusion_constraint} hold, admits a unique solution 
\[u\in W^{1,\infty,\infty}(0,T;L^2\om,L^2\om)\cap W^{1,\infty,2}(0,T;H^1_0\om,H^1_0\om) \]

\proof 
We verify the conditions in Theorem \ref{theo-Roubicek} for $\Hf = L^2(\Omega)$, $\Vf = H_0^1(\Omega)$ with $\|u\|_{\Vf}=\|\nabla u\|_{\Hf}$ and $\Ff(u) := -F(u) - f(\alpha,u)$, where $F:\Vf \rightarrow \Vf^*$ is given as
\[ F(u) =  \nabla\cdot(a\nabla u) - cu. \]

First, note that due to measurability and the growth constraint, the Nemytskii mapping $f(\alpha,\cdot):\Vf \rightarrow \Vf^*$, where we set $f(\alpha,u)(w):= \int_\Omega f(\alpha,u(x))w(x) \wrt x$ for $w \in \Vf$, is indeed well-defined since,
\begin{align*}
\|f(\alpha,v)\|_{\Vf^*} &= \underset{\|w\|_\Vf\leq1}{\sup}  \int_\Omega f(\alpha,\commA{v}(x))w(x) \wrt x \leq\commA{ \underset{\|w\|_\Vf\leq1}{\sup}   C_{\alpha}(|\Omega|^{5/6}+\|v^5\|_{L^{6/5}\om})\|w\|_{L^{6}\om}}\\
&\leq \commA{C}C_{H^1\to L^6}(1+\|v^5\|_{L^{6/5}\om})\leq \commA{C}(C_{H^1\to L^6})^6(1+\|v\|_{H^1\om}^5).
\end{align*}%\todo{this can be removed if you agree to may not at the beginning of this section: Further, we assume that $f(\alpha,\cdot):\Vf \rightarrow \Vf^*$ is continuous, as this will be the case in our application below.}

Since $0<\underline{a}\leq a$ almost everywhere on $\Omega$ and $c\in L^2\om$, the estimate
\begin{align}\label{ex0-estimate-cu}
\la cu,u \ra_{\Vf^*,\Vf} %=\int_\Omega cu^2\wrt x 
& \leq \|c\|_{L^2\om}\|u^{3/2}u^{1/2}\|_{L^2\om} 
\leq (C_{H^1\to L^6})^{3/2}\|c\|_{L^2\om}\|u\|^{3/2}_ {\Vf} \|u\|^{1/2}_{\Hf} \nonumber\\
&\leq  \frac{3}{4}\left(\underline{a}^{3/4}\|u\|^{3/2}_\Vf \right)^{4/3}+\frac{1}{4}\left( \frac{(C_{H^1\to L^6})^{3/2}\|c\|_{L^2\om}}{\underline{a}^{3/4}}\|u\|^{1/2}_\Hf\right)^{4} \\
& =\frac{3\underline{a}}{4}\|u\|_\Vf^2 + \frac{(C_{H^1\to L^6})^6}{4\underline{a}^3}\|c\|_{L^2\om}^4\|u\|_\Hf^2 \nonumber
\end{align}
yields
\begin{align*}
\la -\nabla\cdot(a\nabla u) + cu,u \ra _{\Vf^*,\Vf}&\geq \underline{a}\|u\|_\Vf^2-\left(\frac{3\underline{a}}{4}\|u\|_\Vf^2 + \frac{(C_{H^1\to L^6})^6}{4\underline{a}^3}\|c\|_{L^2\om}^4\|u\|_\Hf^2\right)%\\&=\frac{\underline{a}}{4}\|u\|_\Vf^2 - \frac{(C_{H^1\to L^6})^6}{4\underline{a}^3}\|c\|_{L^2\om}^4\|u\|_\Hf^2
= c_0\|u\|_\Vf^2 - c_2\|u\|_\Hf^2,
\end{align*}
with $c_0:=\underline{a}/4$, $c_2:=(C_{H^1\to L^6})^6\|c\|_{L^2\om}^4/4\underline{a}^3$.
Together with monotonicity of $\commA{-f(\alpha,\cdot)}$ and $f(\alpha,0)=0$, one has $\la f(\alpha,u),u\ra_\vsv=\la f(\alpha,u)-f(\alpha,0), u-0\ra_{\Vf^*,\Vf}\geq 0$. This implies that semicoercivity as in \ref{ex0-S-coe} with $c_0$, $c_2$ as above and $c_1 = 0$. Also, the second estimate in the regularity condition \ref{ex0-S-reg} now follows directly with
\[c_0=C_0, \quad C_2=c_2,\]
where again, we employ monotonicity of $f(\alpha,\cdot)$.

In order to verify pseudomonotonicity \ref{ex0-S-pseu}, we first notice that $\Ff:\Vf\to\Vf^*$ is bounded, i.e., it maps bounded sets to bounded sets, and continuous where the latter follows from continuity of $F$, which is immediate, and continuity of $f$, which holds by assumption. Using this, one can apply \cite[Lemma 6.7]{FRANCU} to conclude pseudomonotonicity if  the following statement is true
\begin{align*}
[\,\, u_n \overset{\Vf}{\rightharpoonup} u \quad\commA{\text{and}}\quad\limsup_{n\to\infty}\,\la \Ff(u_n)-\Ff(u),u_n-u \ra_{\Vf^*,\Vf} \leq 0\,\,]
\quad\Rightarrow\quad u_n \overset{\Vf}{\rightarrow}u. 
\end{align*}
The latter follows since, by $\Vf\compt \Hf$, one gets for $u_n \overset{\Vf}{\rightharpoonup} u$ that $u_n\overset{\Hf}{\to}u$ and
\begin{align} \label{ex0-pseudo}
0 \geq\quad \limsup_{n\to\infty}\,\la\Ff(u_n)-\Ff(u),u_n-u \ra_{\Vf^*,\Vf} &\geq c_0 \limsup_{n\to\infty}\|u_n-u\|_\Vf^2-c_2\lim_{n\to\infty}\|u_n-u\|_\Hf^2 \nonumber\\
&=c_0\limsup_{n\to\infty}\|u_n-u\|_\Vf^2,
\end{align}
which implies $u_n \overset{\Vf}{\rightarrow}u$ as $n\to\infty.$ With this, Theorem \ref{theo-Roubicek} implies unique existence of a solution \[u\in W^{1,\infty,\infty}(0,T;\Hf,\Hf)\cap W^{1,\infty,2}(0,T;\Vf,\Vf).\] \qedhere
\end{lem}
Note that, by embedding, $u\in W^{1,\infty,\infty}(0,T;\Hf,\Hf)\cap W^{1,\infty,2}(0,T;\Vf,\Vf)$ implies that $u\in L^\infty(0,T;\Vf)\cap H^1(0,T;\Vf))$. In a second step, we now aim to find suitable assumptions on the parameter spaces $X_\varphi$, $X_c$, $X_a$ and $U_0$ such that regularity of the solution $u$ of \eqref{ex0} as obtained in the previous proposition is lifted to $u\in L^\infty\tom$.

\begin{rem}\label{regularty-2approaches}
There are at least two ways to achieve this: One is to enhance space regularity of $u$ from $H^1(\Omega) $ to $W^{k,p}(\Omega)$ with $ kp>d$ such that $W^{k,p} (\Omega) \hookrightarrow C(\overline{\Omega})$ and we can ensure $u \in L^\infty( (0,T), C(\overline{\Omega})) \hookrightarrow L^\infty\tom$. The other possible approach is to ensure a $ W^{2,q}\om$-space regularity with $q$ sufficiently large such that $ u \in L^2((0,T),W^{2,q}\om)\cap H^1(0,T;W^{2,q}\om)\hookrightarrow C(0,T;L^\infty\om)$.

While the first approach might yield weaker condition on $kp$, it imposes a non-reflexive state space. The latter choice on the other hand fits better into our setting of reflexive spaces, thus we proceed with the latter choice.
\end{rem}

Now our goal is to determine an exponent $q$ such that, if $u\in L^2(0,T;W^{2,q}\om)\cap H^1(0,T;H^1\om)$, it follows that $u \in  C(0,T;W^{1,2p}\om)$ with $p>d/2$ such that $W^{1,2p}\om \embed L^\infty(\Omega)$ and ultimately $u \in L^\infty\tom$.
To this \commA{aim}, first note that for $u\in L^2(0,T;W^{2,q}\om\cap H^1_0(\Omega))\,\cap\, H^1(0,T;H^1\om)$, by Friedrichs's inequality, %\todo{adapt to zero boundary conditions whenever necessary: $W^{k,p}_0$ instead of $W^{k,p}$}  \commA{done}
it follows that $u \in C(0,T;L^{2p}(\Omega))$ if
\[  |\commA{\nabla} u|^p \in C(0,T;L^2\om) .\]
To ensure the latter, we use that $(\commB{\nabla u})^p \in L^2(0,T;W^{1,q/p}\om)\cap  H^1(0,T;L^{2/p}\om)$ and that 
\[ L^2(0,T;W^{1,q/p}\om)\cap  H^1(0,T;L^{2/p}\om) \embed C(0,T;L^2\om)
\]
provided that $dp>q\geq\frac{dp}{d+1}$ and $\frac{p}{2}\leq 1-\frac{p}{q}+\frac{1}{d}$. Indeed, in this case it follows that
\[ L^{\frac{2}{p}}\om \embed L^{\frac{dq}{dq-dp+q}}\om \embed \commA{(}W^{1,\frac{q}{p}}\om)^* \]
such that the embedding into $C(0,T;L^2\om)$ follows from \cite[Lemma 7.3]{Roubicek} \commA{(see Appendix \ref{appendix-Roubicek})}.
Since $2pd/(2d+2-dp)\geq dp/(d+1)$, it follows that we can ensure for $ p>d/2$ that $ u \in  L^\infty\tom $ if 
\[ dp > q \geq \frac{2dp}{2d+2-dp}.
\]
This is fulfilled for $p=d/2+\epsilon$ with $\epsilon>0$ if  $dp > q\geq (2d^2+4\epsilon d)/(4+4d-d^2-2\epsilon d)$ and, more concretely, in case $d=2$ for $p=1+\epsilon$ $q=(2+2\epsilon)/(2-\epsilon)$ and $\epsilon \in (0,1)$ and in case $d=3$ for $p=3/2+\epsilon$, $q = (18+12\epsilon)/(7-6\epsilon)$ and $\epsilon \in (0,1/2)$.

Let us focus on the latter case of $d=3$ and derive suitable assumptions on $X_\varphi$, $X_c$, $X_a, U_0$ and $f$  such that the solution $u$ to \eqref{ex0} fulfills
\begin{equation*}%\label{ex0-space}
\begin{aligned}
u \in L^2(0,T;W^{2,q}\om)\cap H^1(0,T;H^1\om)\embed C(0,T;W^{1,2p}\om), 
\end{aligned}
\end{equation*}
where the embedding holds by our choice of $q$ and $p$.

\begin{lem}[\commB{Lifted regularity}] \label{prop:step2} In addition to the assumptions of Lemma \ref{prop:step1}, assume that $d=3$ and that, for positive numbers $p,\epsilon,q,\bar{q}$ and $Pa$ with 
\[
 p=3/2+\epsilon ,\quad  \min\{6,3p\}> q\geq \frac{18+12\epsilon}{7-6\epsilon},\quad \frac{q\bar{q}}{2\overline{q}-q} \leq 2 ,\quad \overline{q}\preceq\frac{3q}{3-q}, \quad Pa >\max\{3,\frac{q\bar{q}}{\overline{q}-q}\}
\]
it holds that
\begin{align*}
& c\in L^q\om,\quad a\in W^{1,Pa}\om, \text{ and } 0<\underline{a}\leq a(x)\leq \overline{a}\,\,\, \commB{\text{for almost all }} x\in \Omega,\\
& \varphi\in L^q\om,\quad  u_0\in H^2(\Omega),\\
& |f(\alpha,v)|< C_\alpha(1+|v|^{B}) \text{ with }B < 6/q + 1,
%\\ & \text{and the Nemytskii mapping } f(\alpha,\cdot):H_0^1(\Omega) \rightarrow H_0^1(\Omega)^* \text{ is continuous}
\end{align*}
Then, the unique solution of \eqref{ex0} fulfills 
\begin{align}\label{eq:prop-ex0-PDEexistence_embeddings}
u\in&L^2(0,T;W^{2,q}\om)\cap H^1(0,T;H^1\om)\embed C(0,T;W^{1,2p}\om)\embed L^\infty((0,T) \times \Omega)
\end{align}
\proof
 From \eqref{ex0} we get
\begin{align}\label{ex0-energy-0}
a\Delta u= \dot{u}-\nabla a\cdot\nabla u + cu + f(\alpha,u)-\varphi,
\end{align}
and by
$\overline{q}\preceq\frac{3q}{3-q}$ such that $W^{1,q}\om\embed L^{\overline{q}}\om$), we estimate the components of the right hand side of \eqref{ex0-energy-0}, using parameters $\delta,\delta_1>0$ (which will be small later on). %\todo{Notation clash with $p=3/2+\epsilon$. Rename second $\epsilon $ and $\epsilon_1$ to $\delta$ and $\delta_1$, or similar?} \commA{done}. 
Since $q\leq6$
\begin{align}
\|\dot{u}\|_{L^q\om}\leq C_{H^1\to L^q}\|\dot{u}\|_{H^1\om}.\label{ex0-energy-udot}
\end{align}
By $q \leq 6$ and $c\in L^q\om$, using density, we can choose $c_\infty \in L^\infty(\Omega)$ such that $\|c-c_\infty\|_{L^q\om} \leq \delta$ and obtain
\begin{align}
\|cu\|_{L^q\om} & \leq \|c_\infty u\|_{L^q\om}+\|(c-c_\infty) u\|_{L^q\om} \leq \|c_\infty\|_{L^\infty\om}\|u\|_{L^q\om}+\|c - c_\infty\|_{L^q\om}\|u\|_{L^\infty\om} \nonumber\\
& \leq C_{H^1\to L^q}\|c_\infty\|_{L^\infty\om}\|u\|_{H^1\om}+ C_{W^{2,q}\to L^\infty}\delta \|u\|_{W^{2,q}\om}.  \label{ex0-energy-c}
\end{align}
Now by the assumption $|f(\alpha,v)| \leq C_\alpha (1 + |v|^{B}) $ with $B < 6/q+1$ (note that this means also $B \leq 5$) then, by possibly increasing $B$, we can assume that $6/q < \commA{ B <  6/q+1}$ and select $\beta:= B-6/q \in (0,1)$, such that $q(B-\beta) = 6$. Applying Young's inequality with arbitrary positive factor $\delta_1>0$, we have
\begin{align}
\|f(\alpha,u)\|_{L^q\om}
& \leq C_\alpha(1+\||u|^B\|_{L^q\om}) \leq C_\alpha\left(1 + \|u\|^\beta_{L^\infty\om}\|u^{B-\beta}\|_{L^q\om}\right) \nonumber\\
& \leq C_\alpha\left(1+\beta\delta_1^{1/\beta}\|u\|_{L^\infty\om} +\dfrac{1-\beta}{\delta_1^{1/(1-\beta)}}\|u\|^\frac{B-\beta}{1-\beta}_{L^{q(B-\beta)}\om} \right) \nonumber\\
&\leq C_\alpha\left(1+ C_{W^{2,q}\to L^\infty}\beta\delta_1^{1/\beta}\|u\|_{W^{2,q}\om} +\dfrac{1-\beta}{\delta_1^{1/(1-\beta)}} C_{H^1\to L^{q(B-\beta)}} \|u\|^{\frac{B-\beta}{1-\beta}}_{H^1\om} \right) 
%\nonumber\\ &\qquad\qquad\qquad\text{if}\quad 0<\beta <1, q(B-\beta)\leq 6 \quad\Leftrightarrow\quad B< \frac{6}{q}+1,\\
\end{align}
Using $ a\in  W^{1,Pa}\om$  with $Pa\geq\frac{q\bar{q}}{\overline{q}-q}$ and  $\frac{q\bar{q}}{2\overline{q}-q} \leq 2$, again using density, we can choose $a_\infty\in W^{1,\infty}\om $ such that $\|\nabla a-\nabla a_\infty\|_{L^{\frac{q\bar{q}}{\overline{q}-q}}\om}<\delta$ and obtain % with $\epsilon_1>0$ arbitrary,
\begin{align}
%4----------
&\|\nabla a\cdot\nabla u\|_{L^q\om}\leq \|(\nabla a-\nabla a_\infty)\cdot\nabla u\|_{L^q\om}+\|\nabla a_\infty\cdot\nabla u\|_{L^q\om} \nonumber\\
&\quad\leq \|\nabla a-\nabla a_\infty\|_{L^{\frac{q\bar{q}}{\overline{q}-q}}\om} \|\nabla u\|_{L^{\overline{q}}\om}+\|\nabla a_\infty\|_{L^\infty\om}\||\nabla u|^{1/2}|\nabla u|^{1/2}\|_{L^q\om}  \nonumber\\
&\quad\leq \delta \|\nabla u\|_{L^{\overline{q}}\om} + \|\nabla a_\infty\|_{L^\infty\om}\left(\frac{\delta_1}{2}\|\nabla u\|_{L^{\overline{q}}\om}+\frac{1}{2\delta_1}\|\nabla u\|_{L^{\frac{q\bar{q}}{2\overline{q}-q}}\om} \right) \nonumber\\
&\quad\leq C_{W^{2,q}\to W^{1,\overline{q}}}\left(\delta+\frac{\delta_1}{2}\|\nabla a_\infty\|_{L^\infty\om}\right)\| u\|_{W^{2,q}\om} +  C_{L^2\to L^{\frac{q\bar{q}}{2\overline{q}-q}}}\frac{\|\nabla a_\infty\|_{L^\infty\om}}{2\delta_1}\| u\|_{H^1\om} 
%& \qquad\qquad\qquad\text{if }\quad a\in X_a\subset W^{1,Pa}\om\text{ with } Pa\geq\frac{q\bar{q}}{\overline{q}-q} \quad\text{and}\quad  \frac{q\bar{q}}{2\overline{q}-q} \leq 2, \label{ex0-energy-aterm}\\
\end{align}
Using that also $\varphi\in  L^q\om$, taking the spatial $L^q$-norm in \eqref{ex0-energy-0}, estimating by the triangle inequality, raising everything to the second power,
we arrive at
\begin{align*}
&\underline{a}^2\|\Delta u\|^2_{L^2( 0,T;L^q\om)}\leq\|a\Delta u\|^2_{L^2( 0,T;L^q\om)}\\
&\leq 5\left( \|\dot{u}\|^2_{L^2( 0,T;L^q\om)}+\|\nabla a \cdot \nabla u \|^2_{L^2( 0,T;L^q\om)} + \|cu\|^2_{L^2( 0,T;L^q\om)}  + \|f(\alpha,u)\|^2_{L^2( 0,T;L^q\om)} \right. \\
& \left. \qquad \quad + \|\varphi\|^2_{L^2( 0,T;L^q\om)}  \right)\\
&\leq 15\left(\|\dot{u}\|^2_{L^2( 0,T;L^q\om)}+C_{c,a}\|u\|^2_{L^2( 0,T;H^1\om)}+ T\|\varphi\|^2_{L^q\om}+ TC_{B,\alpha,\beta}\|u\|^{2\frac{B-\beta}{1-\beta}}_{L^\infty( 0,T;H^1\om)} +TC^2_\alpha \right. \\
&\left. \qquad \quad  +\tilde{\epsilon}\|\Delta u\|^2_{L^2( 0,T;L^q\om)} \right)
\end{align*}
with
$$
\tilde{\epsilon}:= \left[C_{W^{2,q}\to L^\infty }\delta+ C_\alpha C_{W^{2,q}\to L^\infty}\beta\delta_1^{1/\beta}+C_{W^{2,q}\to W^{1,\overline{q}}}\left(\delta+\frac{\delta_1}{2}\|\nabla a_\infty\|_{L^\infty\om}\right) \right]^2.
$$ 
For sufficiently small $\delta,\delta_1$, this leads to %\todo{replace $V$ by $H^1\om$ here?} \commA{done}
\begin{align} \label{ex0-energy-final}
&0<\quad \left(\frac{\underline{a}^2}{2^6}-\tilde{\epsilon}\right)\|\Delta u\|^2_{L^2(0,T;L^q\om)}\leq C_{c,a,\varphi,B,\beta,T}\left(\|\dot{u}\|^2_{L^2(0,T;H^1_0(\Omega))}+\|u\|^2_{L^\infty(0,T;H^1_0(\Omega))} \right) \quad< \infty.
\end{align}
%\todo{please check in particular if this is correct: TN: Yes, it is. I just rephrased it a little bit to make it clearer. We don't need to mention this in the response letter, I think.}
The fact that $\nabla u\in L^2(0,T;L^2(\Omega))$ and $\Delta u\in L^2(0,T;L^q(\Omega)), q\geq 2$ as above imply $\nabla u\in L^2(0,T;H^1\om)$ thus $\nabla u\in L^2(0,T;L^q\om)$ for $q\leq 6$. This and \eqref{ex0-energy-final} ensures that $u\in L^2(0,T;W^{2,q}\om)$. By Lemma \ref{prop:step1}, $u\in W^{1,\infty,\infty}(0,T;\Hf,\Hf)\cap W^{1,\infty,2}(0,T;\Vf,\Vf)$; thus, by embedding, $u \in H^1(0,T;H^1(\Omega))$. Consequently, \[u\in L^2(0,T;W^{2,q}\om\cap H^1(0,T;H^1(\Omega)).\] This, together with the argumentation after Remark \ref{regularty-2approaches} completes the proof. \qedhere
\end{lem}

The obtained unique existence result in now summarized in the following proposition. %Here, we focus on the practical case $d=3$. %\todo{fix $d=3$ also below, update according to commend on $q$ as above} \commA{done}

\begin{prop}\label{prop-ex0-PDEexistence}
i) The nonlinear parabolic PDE \eqref{ex0} with $d=3$ admits the unique solution
\begin{align}\label{eq:prop-ex0-PDEexistence_embeddings}
u\in&L^2(0,T;W^{2,q}\om)\cap H^1(0,T;H^1\om)\embed C(0,T;W^{1,2p}\om)\embed L^\infty((0,T) \times \Omega)
\end{align}
if the following conditions are fulfilled:%\todo{$q \geq$ or $=$ below?}
{\allowdisplaybreaks
\begin{align*}
& p=3/2+\epsilon \text{ with }\epsilon>0 \\
&  \min\{6,3p\}> q\geq \frac{18+12\epsilon}{7-6\epsilon}, \quad\text{and}\quad \frac{q\bar{q}}{2\overline{q}-q} \leq 2 \quad\text{with } \overline{q} \text{ such that } \overline{q}\preceq\frac{3q}{3-q}\\% W^{1,q}\om\embed L^{\overline{q}}\om\\
& c\in L^q\om,\quad a\in W^{1,Pa}\om, Pa >\max\{3,\frac{q\bar{q}}{\overline{q}-q}\}\text{ and } 0<\underline{a}\leq a(x)\leq \overline{a}\,\,\, \commB{\text{for almost all }} x\in \Omega,\\
& \varphi\in L^q\om,\quad  u_0\in H^2(\Omega),\\
&(-f(\alpha,\cdot)) \text{ is monotone and } f(\alpha,0)=0, |f(\alpha,v)|< C_\alpha(1+|v|^{B}) \text{ with }B < 6/q + 1,
%\\ & \text{and the Nemytskii mapping } f(\alpha,\cdot):H_0^1(\Omega) \rightarrow H_0^1(\Omega)^* \text{ is continuous}
\end{align*}}
%\note{the last line above can be removed if you agree to my note at the beginning of this section.}
ii) Moreover, the claim in i) still holds in case $f(\alpha,\cdot)$ is replaced by neural network $\Nc_\theta(\alpha,\cdot)$ with $\sigma\in \C_\text{Lip}(\R,\R)$. 
%\begin{align}\label{ex0-NN}
%\sigma\in \C_\text{Lip}(\R,\R) \quad \text{ s.t. }\quad |\Nta(u)-c_{\theta,\alpha}|\leq C_{\theta,\alpha}|u| \quad\text{for some constants } c_{\theta,\alpha}, C_{\theta,\alpha}.
%\end{align} 
\end{prop}
\begin{proof}%\note{Added details in proof}
i) Lemma \ref{prop:step1} ensures that \eqref{ex0} admits a unique solution 
\[u\in W^{1,\infty,\infty}(0,T;L^2(\Omega),L^2(\Omega))\cap W^{1,\infty,2}(0,T;H^1(\Omega),H^1(\Omega)),\] such that in particular $u\in L^\infty(0,T;H^1(\Omega))\cap H^1(0,T;H^1(\Omega))$. Proposition \ref{prop:step2} ensures the embeddings as in \eqref{eq:prop-ex0-PDEexistence_embeddings} 
%L^2(0,T;W^{2,q}\om)\cap H^1(0,T;H^1\om)\embed C(0,T;W^{1,2p}\om)\embed L^\infty((0,T) \times \Omega)$ 
hold true again by our choice of $p,q$.

ii) Now consider the case that $f(\alpha,\cdot)$ is replaced by $\Nc_\theta(\alpha,\cdot)$ for some known $\alpha, \theta$. With $L_{\theta,\alpha}$ the Lipschitz constant of $\Nc_\theta(\alpha,\cdot):\R \rightarrow \R$, we first observe that, for $v \in \R$, 
\[ |\Nc_\theta(\alpha,v)| \leq |\Nc_\theta(\alpha,0)| + |\Nc_\theta(\alpha,v)- \Nc_\theta(\alpha,0)| \leq |\Nc_\theta(\alpha,0)| + L_{\theta,\alpha}|v|
\]
such that the growth condition $|\Nc_\theta(\alpha,v)|< C_\alpha(1+|v|^{B})$  with $B < 6/q + 1$ and in particular the growth condition of Proposition \ref{prop:step2} holds. 
This shows in particular that the induced Nemytskii mapping $\Nc_\theta(\alpha,\cdot):H^1(\Omega) \rightarrow H^1(\Omega)^*$ is well-defined. Further, we can observe that, again for $u,v \in H^1(\Omega)$
\begin{align*}
&|\la \Nt(\alpha,u),u \ra_{H^1(\Omega)^*,H^1(\Omega)}|\leq L_{\theta,\alpha}\|u\|^2_{L^2(\Omega)}+|\Nc_{\theta,\alpha}(0)|C_{H^1\to L^1}\|u\|_{H^1(\Omega)},\\
&|\la \Nt(\alpha,u)-\Nt(\alpha,v),u-v \ra_{H^1(\Omega)^*,H^1(\Omega)}|\leq L_{\theta,\alpha}\|u-v\|^2_{L^2(\Omega)}.
\end{align*}
Using these estimates, it is clear that the conditions \ref{ex0-S-coe} and \ref{ex0-S-reg} in Theorem \ref{theo-Roubicek} can be shown similarly as in Step 1 without requiring $\Nt(\alpha,0)=0$  or monotonicity of $\Nc_\theta (\alpha,\cdot)$. This completes the proof.
%This means the estimates in \ref{ex0-S-coe}-\ref{ex0-S-reg} and \eqref{ex0-pseudo} do not change when $\Nc$ plays the role of $f$. Furthermore, we wish to emphasize that $\Nta$ does not need to satisfy monotonicity \todo{monotonicity only used in..?} as well as $\Nta(0)=0$ like $f$.\\
\end{proof}
\begin{rem}\label{rem-Lipactivation}
\commA{For neural networks, some examples fulfilling the conditions in Proposition \ref{prop-ex0-PDEexistence}, i.e. Lipschitz continuous activation functions, are the RELU function $\sigma(x)=\max\{0,x\}$,  the tansig function $\sigma(x)=\tanh(x)$, the sigmoid (or soft step) function $\sigma(x)=\frac{1}{1+e^{-x}}$, the softsign function $\sigma=\frac{x}{1+|x|}$ or the softplus function $\sigma(x)=\ln(1+e^x)$.}
\end{rem}

\subsection{Well-posedness for the all-at-once setting}
With the result attained in Proposition \ref{prop-ex0-PDEexistence}, we are ready to determine the function spaces for the minimization problems \eqref{eq:main_identification_aao_setting}, \eqref{eq:standard_pid_aao_setting} in the all-at-once setting and explore further properties discussed in Section \ref{sec-abstractPI}.

\begin{rem}\label{phi}
For minimization in the reduced setting, we usually invoke monotonicity in order to handle high nonlinearity (c.f. Proposition \ref{prop-ex0-PDEexistence}). The minimization problems in the all-at-once setting, however, do not require this condition, thus allowing for more general classes of functions, e.g. \commA{by including in $F$ another known nonlinearity} $\phi$  as in the following Proposition.
\end{rem}

%\begin{rem}[Reduced setting and neural network]
%The minimization problems \eqref{eq:main_identification_aao_setting}, \eqref{eq:standard_pid_aao_setting} associated with \eqref{ex0} in the reduced setting require the construction of the parameter-to-state map 
%\[S:X \times \R^m \times U_0\times \Theta \to\Uc,\quad (\lambda,\alpha,u_0,\theta) \mapsto u \quad\text{s.t. }\quad u \text{ solves } \eqref{ex0}.\]
%$S$ exists if
%\begin{align*}
%\sigma\in \C_\text{locLip}(\R,\R) \text{ is chosen}\quad \text{ s.t. }\quad |\Nta(u)-c_{\theta,\alpha}|\leq C_{\theta,\alpha}|u| \quad\text{for some constants } c_{\theta,\alpha}, C_{\theta,\alpha}.
%\end{align*} 
%The reason for this is that under the above assumption, the estimates in \ref{ex0-S-coe}-\ref{ex0-S-reg} do not change when $\Nc$ plays the role of the nonlinearity.\\
%Some examples for the activation functions are: RELU $\sigma(x)=\max\{0,x\}$,  tansig $\sigma(x)=\tanh(x)$, sigmoid or soft step $\sigma(x)=\frac{1}{1+e^{-x}}$, softsign $\sigma=\frac{x}{1+|x|}$, softplus $\sigma(x)=\ln(1+e^x)$ etc.
%\end{rem}

\begin{prop}\label{lem-ex0-existence}
For $d=3$ and $\epsilon>0$ sufficiently small, define the spaces
\begin{align*}
V= W^{2,q}\om, \quad \vtil=H^1\om,\quad H=W^{1,2p}\om, \quad W=L^q\om,\quad p=\frac{3}{2}+\epsilon, q=\frac{18+12\epsilon}{7-6\epsilon},
\end{align*} %\quad \todo{Y\supseteq V},
and $Y$ such that $V \embed Y$,
resulting in the following state-, image- and observation spaces
\begin{align*}
&\Vc=L^2(0,T;W^{2,q}\om)\cap H^1(0,T;H^1\om),\quad \Wc=L^2(0,T;L^q\om),\quad \Yc=L^2(0,T;Y).
\end{align*}
Further, define the corresponding parameter spaces $ U_0=H^2\om$, $X = X_\varphi \times X_c \times X_a$, where
\begin{align*}
&X_\varphi=X_c= L^q\om,\quad X_a= \{a\in W^{1,Pa}\om, Pa >3\}, % 0<\underline{a}\leq a\leq \overline{a}\quad\text{a.e. on } \Omega\},
\end{align*}
and let $M \in \Lc(\Vc,\Yc)$ be the observation operator.

Consider the minimization problems \eqref{eq:main_identification_aao_setting} and \eqref{eq:standard_pid_aao_setting}, 
with $\commB{F: (0,T) \times X \times V} \rightarrow W$ given as
\[
F(t,(\varphi,c,a),u) \commA{:=} \nabla\cdot(a\nabla u) - cu+\varphi + \phi(u),
\]
where $\phi:V \rightarrow W$ \commA{is an additional known nonlinearity in $F$ (c.f Remark \ref{phi})}; $\phi$ is the induced Nemytskii mapping of a function $\phi\in \C_\text{locLip}(\R,\R)$. The associated PDE given as,
\begin{alignat}{3}
& \dot{u} - \nabla\cdot(a\nabla u) + cu + \phi(u) + \Nt(\alpha,u)= \varphi \quad&&\mbox{ in }\Omega\times\ti,\nonumber\\
& u|_{\partial\Omega}=0 && \mbox{ in } \ti, \label{ex0_extended}\\
& u(0) = u_0 &&\mbox{ in }\Omega,\nonumber
\end{alignat}
with the activation functions $\sigma $ of $\Nt(\alpha,u)$ satisfying $ \sigma\in \C_\text{locLip}(\R,\R)$, and with $\Rc_1, \Rc_2 $ nonnegative, weakly lower semi-continuous and such that the sublevel sets of $(\lambda,\alpha,u_0,u,\theta) \mapsto \Rc_1(\lambda,\alpha,u_0,u) + \Rc_2(\theta)$ are weakly precompact. 
%
%
%associated with the PDE \eqref{ex0} with  \todo{$f(\alpha,u)$ and $f(u,\alpha)$? role of $\phi$?} $f(\lambda,u)=\Nt(\lambda,u)+\phi(u)$, $ \sigma\in \C_\text{locLip}(\R,\R)$, and $\phi\in \C_\text{locLip}(\R,\R)$ is known \todo{$\phi$ does not fit to general setting}.
%We pose these minimization problems on 
%\begin{align*}
%V= W^{2,q}\om, \quad \vtil=H^1\om,\quad H=W^{1,2p}\om, \quad W=L^q\om,\quad Y\supseteq V,\quad p=3/2+\epsilon, q=18/7+\epsilon,
%\end{align*}
%meaning that the preimage space, the image space and the observation space are
%\begin{align*}
%&\Vc=L^2(0,T;W^{2,q}\om)\cap H^1(0,T;H^1\om),\quad \Wc=L^2(0,T;L^q\om),\quad \Yc=L^2(0,T;Y),
%\end{align*}
%and with the parameter spaces $ U_0=H^2\om$, $X = X_\varphi \times X_c \times X_a$, where \todo{$X_a$ is not a space, need to include it in regularization and verify conditions}.
%\begin{align*}
%&X_\varphi=X_c= L^q\om,\quad X_a= \{a\in W^{1,Pa}\om, Pa >3: 0<\underline{a}\leq a\leq \overline{a}\quad\text{a.e. on } \Omega\},\quad .
%\end{align*}
Then, each of \eqref{eq:main_identification_aao_setting} and \eqref{eq:standard_pid_aao_setting} admits a  minimizer.
\end{prop}
\begin{proof}
Our aim is examining the assumptions proposed \commB{Lemma \ref{lem:closedness_forward_operator}, which leads to the result in Proposition \ref{prop-minexist}}.
At first, we verify Assumption \ref{ass:general_basic}. The embeddings
\[  U_0 \hookrightarrow H\hookrightarrow  W, \quad  V\embed H, \quad V\embed Y, \quad V\compt L^\pfix(\Omega)  = W, \quad \tilde{V} \embed W. 
\]
are an immediate consequence of our choice of $p$ and $q$ and standard Sobolev embeddings.
The embeddings
\[\Vc\hookrightarrow L^\infty\tom,\quad\Vc\hookrightarrow C(0,T;H)\]
follow from the discussion in Step 2 above, see also Proposition \ref{prop-ex0-PDEexistence}.

%At first, the embeddings assumed in Assumption \ref{ass:general_basic}
%\[  U_0 \hookrightarrow H\hookrightarrow  W, \quad  V\embed H, \quad V\embed Y, \quad V\compt L^k(\Omega)  = W, \quad \tilde{V} \embed W \quad \text{for some }k\in[1,\infty)
%\]
%are an immediate consequence of our choice of $p$ and $q$ and standard Sobolev embeddings. The embeddings
%\[\Vc\hookrightarrow L^\infty\tom,\quad\Vc\hookrightarrow C(0,T;H)\]
%follow from the discussion in Step 2 above summarized in Proposition \ref{prop-ex0-PDEexistence}.
Noting that well-definedness of the Nemytskii mappings as well as the growth condition \eqref{growth} \commA{are consequences of the following arguments on weak continuity}. We focus on weak continuity of $F:\Vc\times(X_c,X_a,X_\varphi)\to\Wc, F(\lambda,u):=\nabla\cdot(a\nabla u) - cu+\varphi + \phi(u)$ via weak continuity of the operator inducing it as presented in Lemma \ref{lem:closedness_forward_operator}.
First, for the $cu$ part we see $(c,u)\mapsto cu$ is weakly continuous on $(X_c,H)$.  Indeed, for $c_n\rightharpoonup c$ in $X_c$, $u_n\rightharpoonup u$ in $H=W^{1,2p}\om\compt L^\infty(\Omega)$ thus $u_n\to u$ in $L^\infty(\Omega)$, one has for any $w^*\in W^*= L^{q^*}(\Omega)$, %\note{corrected signs} \commA{thanks.} \commA{"convergence" does not appear in the original version. We only have weak convergence of $c_n,a_n,u_n...$.}
\begin{align*}
&\int_\Omega (cu-c_nu_n)w^*\wrt x = \int_{\commB{\Omega}} (c-c_n)uw^*\wrt x +\int_{\commB{\Omega}} c_n(u-u_n)w^*\wrt x
\quad\overset{n\to\infty}{\to}0
\end{align*}
due to $uw^*\in L^{q^*}(\Omega)$, $ \|c_n w^*\|_{L^1(\Omega)}\leq C<\infty$ for all $n$ and $u_n\to u$ in $L^\infty(\Omega)$.\\
For the $\nabla\cdot(a\nabla u)$ part, $H=W^{1,2p}\om$ is not strong enough to enable weak continuity of $(a,u)\mapsto \nabla\cdot(a\nabla u)$ on $(X_a,H)$, we therefore evaluate directly weak continuity of the Nemytskii operator. So, \commB{let $(a_n,u_n) \rightharpoonup (a,u)$ in $X_a\times\Vc$}, taking $w^*\in L^2(0,T;L^{q^*}(\Omega))$ we have 
\begin{align*}
\int_{\commB{\Omega\times(0,T)}}& (\nabla\cdot(a\nabla u)-\nabla\cdot(a_n\nabla u_n))w^*\wrt x\wrt t\\&=\int_{\commB{\Omega\times(0,T)}} \nabla(a-a_n)\cdot\nabla u w^* \wrt x\wrt t  + \int_{\commB{\Omega\times(0,T)}} \nabla a_n\cdot\nabla(u-u_n) w^* \wrt x\wrt t\\
&\quad+\int_{\commB{\Omega\times(0,T)}} (a-a_n)\cdot\Delta u_n w^*\wrt x\wrt t+\int_{\commB{\Omega\times(0,T)}} a\Delta(u-u_n) w^*\wrt x\wrt t \qquad\overset{n\to\infty}{\to}0
\end{align*}
\commA{due to the following:} we have $\nabla u w^*\in L^2(0,T;L^{Pa^*}(\Omega)), \nabla a_n\rightharpoonup\nabla a$ in $L^{Pa}\om$ in the first estimate, and $u_n\to u$ in $L^2(0,T;W^{1,18}(\Omega)), \|\nabla a_n w^*\|_{L^2(0,T;L^{18/17}(\Omega))}\leq C<\infty$ for all $n$ in the second estimate. In the third estimate, one has $a_n\to a$ in $L^\infty\om$ and \[\|\Delta u_n w^*\|_{\commE{L^1}(0,T;L^1(\Omega))}\commB{\leq\|\Delta u_n\|_{L^2(0,T;L^{q}(\Omega))}\|w^*\|_{L^2(0,T;L^{q^*}(\Omega))}}\leq C<\infty \quad \text{for all $n$}.\] Finally, in in the last estimate it is clear that $aw^*\in L^2(0,T;L^{q^*}(\Omega)), u_n \rightharpoonup u$ in $L^2(0,T;W^{2,q}\om)$ \commB{implying $\Delta u_n \rightharpoonup \Delta u$ in $L^2(0,T;L^{q}\om$)}.\\
For the term $\phi$, by $H=W^{1,2p}\om\compt L^{\infty}\om$ we attain weak-strong continuity of $\phi$ on $H$ 
\begin{align} %\label{ex0-highnonlinear}
\|\phi(u_n)-\phi(u)\|_W\leq \|u_n-u\|_{L^\infty(\Omega)} L\left(\|u_n\|_{H},\|u\|_{H}\right) \quad \to 0\quad\text{for}\quad u_n\overset{H}{\rightharpoonup}u. 
\end{align}
%The growth condition \eqref{growth} in Assumption \ref{ass:general_basic} can be confirmed by a similar estimate.\\ 
%Finally, the fact that the  unknown nonlinear part $f(\alpha,u)$ is approximated by $\Nt$ with the activation function $\sigma\in \C_\text{locLip}(\R,\R)$ completes the verification of the assumptions of Proposition \ref{prop-minexist}.
Finally, the fact that activation function $\sigma $ satisfies $\sigma\in \C_\text{locLip}(\R,\R)$ completes the verification that the result of Proposition \ref{prop-minexist} holds.
\end{proof}
\commB{For the following results, we set $\phi=0$.}
\begin{lem}[Differentiability]\label{lem-differentiability}
In accordance with Proposition \ref{prop-Differentiability} and the frameworks in Proposition \ref{lem-ex0-existence}, \commB{setting $\phi=0$,} the model operator $F:X\times\Vc\to\Wc$ 
%, Remarks \ref{rem-C2activation}-\ref{rem-Hilbertspace-app} 
is \commB{G\^ateaux} differentiable, \commB{as is the neural network $\Nt:\R^m\times\Vc\to\Wc$ with $\sigma\in\Cc^1(\R,\R)$.} 
\begin{proof}
%\todo{comment on gateaux+caratheodory assumptions} \commA{commented in the second last line}
With the setting in Proposition \ref{lem-ex0-existence}, %\todo{what a about the other two settings with stronger/different space $W$} \commA{commented in the last line},
we verify local Lipschitz continuity of $F(\lambda,u)=\nabla\cdot(a\nabla u)-cu+\varphi$ with $\lambda = (\varphi,c,a)$. %\todo{check consistency of signs w.r.t., e.g., \eqref{ex0} \commA{sign is done}, the general existence result, and the section below} \todo{direct consequence of \eqref{eq:application_cu_estimate} and \eqref{eq:application_au_estimate}?}
%Similar to \eqref{eq:application_cu_estimate} and \eqref{eq:application_au_estimate}, one has
To this aim, we estimate
\begin{align*}
&\|F(\lambda_1,u_1)-F(\lambda_2,u_2)\|_W\\
&=\|\nabla\cdot(a_1\nabla(u_1-u_2)) -\nabla\cdot((a_2-a_1)\nabla u_2 )-c_1(u_1-u_2)+(c_2-c_1)u_2 + \varphi_1-\varphi_2 \|_{L^q\om} \\
&\leq \|\nabla a_1\|_{L^{Pa}\om}\|\nabla (u_1-u_2)\|_{L^\qbar\om} +\|a_1-a_2\|_{L^\infty\om}\|\Delta u_1-\Delta u_2\|_{L^q\om} + \|\nabla (a_2-a_1)\|_{L^{Pa}\om}\|\nabla u_2\|_{L^\qbar\om}\\
&\quad +\|a_2-a_1\|_{L^\infty\om}\|\Delta u_2\|_{L^q\om} + \|c_1\|_{L^q\om}\|u_1-u_2\|_{L^\infty\om} +\|c_2-c_1\|_{L^q\om}\|u_2\|_{L^\infty\om} + \|\varphi_1-\varphi_2\|_{L^q\om}\\
&\leq L(\|u_1\|_H,\|u_2\|_H,\|\lambda_1\|_X,\|\lambda_2\|_X)\big(\|u_1-u_2\|_V+\|u_1-u_2\|_H  + (1+\|u_2\|_V)\|\lambda_1-\lambda_2\|_X \big)
\end{align*} 
with $\qbar\overline{q}\preceq\frac{3q}{3-q}$. Also, \commB{G\^ateaux} differentiability of $F:X\times V\to W$ as well as Carath\'eodory assumptions are clear from this estimate and bilinearity of $F$ with respect to $\lambda, u$.
\commB{Differentiability of $\Nt$ with $\sigma\in\Cc^1(\R,\R)$ has been shown in Proposition \ref{prop-Differentiability}, the last paragraph of its proof.}
\end{proof}
\end{lem}

\commB{When the image space $\Wc$ is stronger, that is, $W\nsupseteq L^q\om, \forall q\in[1,\infty)$ as discussed in Remark \ref{rem-diff-C2activation}, we require smoother activation functions than what was employed in Lemma \ref{lem-differentiability} in order to ensure differentiability of $\Nt$.}
\begin{rem}[Strong image space $\Wc$ and smoother neural network] \label{rem-C2activation}
Consider the case where the unknown parameter is $\varphi$,  parameters $a, c$ are known, and the neural network $\Nt$ \commB{has smoother activation}
\[\sigma \in \C^1_\text{locLip}(\R,\R), \text{ i.e. } \sigma'\in\C_\text{locLip}(\R,\R).\]
The minimization problems introduced in Proposition \ref{lem-ex0-existence} have minimizers \commB{that belong to the Hilbert spaces}
\begin{align*}
&\Vc=L^2(0,T;H^3\om)\cap H^1(0,T;H^1\om),%\embed C(0,T;H^2\om),
\quad\Wc=L^2(0,T;H^1\om), \quad\Yc=L^2(0,T;Y),\\
& V=H^3\om\commB{\embed Y}, \quad \widetilde{V}=H^1\om, \quad H=H^2\om, \quad W=H^1\om,
\end{align*}
%where $Y$ is a Hilbert space %such that $V \embed Y$, 
and
\begin{align*}
&X_c= H^1\om,\quad X_a= H^2\om%: 0<\underline{a}\leq a\leq \overline{a}\text{ a.e. on } \Omega\}
,\quad X_\varphi= H^1\om,\quad U_0=H^2\om.
\end{align*}
\begin{proof}
For fixed $\theta, \alpha$, let us denote $\Nt(\alpha,\cdot)=:\Nc_\theta$. It is clear that this setting fulfills all the embeddings in Assumption \ref{ass:general_basic}. % \todo{weak continuity for $\Phi$? \note{$phi=0$ as introduced before the remark} I guess you mean growth condition and continuity of $F$ is clear?} \note{added a sentence for this}.  
Weak-strong continuity of $\Nc_\theta$ is derived from
\begin{align*}
\|\Nta(u_n)-\Nta(u)\|_\Wc^2 &= \|\Nta(u_n)-\Nta(u)\|_{L^2(0,T;L^2\om)}^2+
\|\nabla\Nta(u_n)-\nabla\Nta(u)\|_{L^2(0,T;L^2\om)}^2\\
&=:A+B\quad \overset{n\to\infty}{\to}0,
\end{align*}
since with $\Vc\embed C(0,T;H^2\om)$ and $\sigma\in \C^1_\text{locLip}(\R,\R)$, one has
\begin{align*}
A&\leq C(L'_{\theta,\alpha}(\|u\|_\Vc))^2  \|u_n-u\|^2_{L^2(0,T;L^2\om)},\\
B&\leq 2\|\Nta'(u_n)(\nabla u_n-\nabla u)\|_{L^2(0,T;L^2\om)}^2+2\|(\Nta'(u_n)-\Nta'(u))\nabla u\|_{L^2(0,T;L^2\om)}^2\\
&\leq 2\|\Nta'(u_n)\|^2_{L^\infty\tom}\|\nabla u_n-\nabla u\|_{L^2(0,T;L^2\om)}^2 + 2(L''_{\theta,\alpha}(\|u\|_\Vc))^2\|(u_n- u)\nabla u\|_{L^2(0,T;L^2\om)}^2\\
&\leq 2(L'_{\theta,\alpha}(\|u\|_\Vc))^2  \|\nabla u_n-\nabla u\|^2_{L^2(0,T;L^2\om)} + 2(L''_{\theta,\alpha}(\|u\|_\Vc))^2\|\nabla u\|_{C(0,T;L^6\om)}^2 \|u_n-u\|^2_{L^2(0,T;L^3\om)},
%\\&A+B \leq M_{\theta,\alpha}(\|u\|_\Vc)  \|u_n-u\|_{L^2(0,T;H^1\om)}^2\quad\overset{n\to\infty}{\to}0
\end{align*}
implying $A+B \to 0$ for $u_n\overset{\Vc}{\rightharpoonup}u, \Vc\compt L^2(0,T;H^1\om)$ and Lipschitz constants $L', L''$. This shows continuity of $\Nc$ in $u$; continuity of $\Nc$ in $(\alpha,\theta)$ can be done similarly.
For $F$, when $c,a$ are known and fixed, it is just a linear operator on $u$. Weak continuity of $F$ hence can be explained through its boundedness, which can be confirmed in the same fashion as $A, B$ above.
\end{proof}
\end{rem}

%\begin{rem}[Hilbert space framework]\label{rem-Hilbertspace}
%The framework suggested in Remark \ref{rem-C2activation} is also a feasible choice for the reduced formulation, in which the parameter-to-state map must be well-defined.
%%\todo{Is this the approach?: First, note that the assumptions of Step 1 of Proposition \ref{prop-ex0-PDEexistence} are fulfilled, such that a solution $u \in \ldots$ of \eqref{ex0} exists. Then ...}
%First, note that the assumptions in Step 1 of Proposition \ref{prop-ex0-PDEexistence} are fulfilled, thus ensuring existence of a unique solution $u\in W^{1,\infty,\infty}(0,T;L^2\om,L^2\om)\cap W^{1,\infty,2}(0,T;H^1_0\om,H^1_0\om)$. In Step 2, taking gradient on both sides of the first equation in \eqref{ex0} 
%\begin{align*}
%a\nabla\Delta u&=\nabla\dot{u} +(c-2\nabla a)\Delta u -\Delta a\nabla u +\nabla c\, u + f'_u(\alpha,u)\nabla u - \nabla\varphi\\
%&=: \nabla\dot{ u} + A\Delta u+ B\nabla u + Cu + \Fc(\alpha,u) -\Phi,
%\end{align*}
%we form a new parabolic PDE with new parameters $A, B, C, \Phi$ and the new nonlinear term $\Fc$. Estimating in the same fashion as Step 2 of Proposition \ref{prop-ex0-PDEexistence} with noting $\dot{u}\in L^2(0,T;H^1\om)$ obtained from Step 1, one can prove $u\in L^2(0,T;H^3\om).$ This shows well-definedness of the parameter-to-state map, which is needed in the reduced setting.
%\end{rem}

\commB{To conclude this section, we consider a Hilbert space setting that will be relevant for our subsequent applications.}

\begin{rem}[Hilbert space framework for application]\label{rem-Hilbertspace-app}
Another possible Hilbert space framework where the all-at-once setting is applicable is %\todo{again define spaces $V,H,\tilde{V}$ etc. first} \commA{done}
\begin{align*}
&\Vc=H^1(0,T;H^2\om)\embed C(0,T;H^2\om),\quad\Wc=L^2(0,T;L^2\om),\quad \Yc=L^2(0,T;Y),\\
& V=\widetilde{V}=H=H^2\om\commB{\embed Y}, \quad W=L^2\om
\end{align*}
where $Y$ is a Hilbert space, %with $H^2(\Omega) \embed Y$, 
and
\begin{align*}
&X_c= L^2\om,\quad X_a= H^2\om,% 0<\underline{a}\leq a\leq \overline{a}\quad\text{a.e. on } \Omega\}, 
\quad X_\varphi= L^2\om,\quad U_0=H^2\om.
\end{align*}
Verification of weak continuity and the growth condition for $F$ can be carried out similarly as in Proposition \ref{lem-ex0-existence}; moreover, weak continuity of $(X_a\times H)\ni(a,u)\mapsto\nabla\cdot(a\nabla)\in W$ can be confirmed like the part $(c,u)\mapsto cu$, without the need of evaluating directly the Nemytskii operator. %Verification of the Lipschitz condition to confirm differentiability is similar to Remark \ref{lem-differentiability}.
This is the setting in which we will study in detail the application \eqref{ex0}.
\end{rem}

\section{Case studies in Hilbert space framework}\label{sec:case_study}

\subsection{Setup for case studies}\label{sec:casestudy}
In this section, for the sake of simplicity of implementation, we carry out case studies for some minimization examples in a Hilbert space framework, where we drop the unknown $\alpha$ and use the regularizers \commB{$\Rc_1=\|\cdot\|^2_{X\times U_0\times\Vc}$, $\Rc_2=\|\cdot\|^2_\Theta$}.

% \todo{can be removed: We assume that it has a slightly simpler architecture than in Section \ref{sec:CNN} with single channel in each hidden layer, i.e. $k_i=\ldots=k_L=1$,  yielding the hyperparameters in the $l$-layer: $\weight^l_{s,j}=\weight^l\in \R^{n_l\times n_{l-1}}$ and $\bias^l_{s,j}=\bias^l\in\R^{n_l}$.} %Apart from that,  we assume to have full observation $M=\text{Id}$ or several discrete observations $M=(\cdot)_{t=t_i}, t_i\in(0,T)$ of the state $u$.

\begin{prop}\label{prop-adjoints-cont}
Consider the minimization problem \eqref{eq:main_identification_aao_setting} (or \eqref{eq:standard_pid_aao_setting}) associated with the learning informed PDE
\begin{alignat}{3}
& \dot{u}-\nabla\cdot(a\nabla u) + cu -  \varphi -\Nt(u)=:\dot{u} - F(\lambda,u)-\Nc(u,\theta) =0 \quad&&\mbox{ in }\Omega\times\ti\nonumber\\
& u(0) = u_0 &&\mbox{ in }\Omega\nonumber
\end{alignat}
for $\commB{\sigma\in}\, \Cc^1(\R,\R)$, $M=\text{Id}$ in the Hilbert spaces %\todo{again define spaces $V,H,\tilde{V}$ etc. explicitly} \commA{done}
\begin{align*}
&\Vc=H^1(0,T;H^2\om\cap H^1_0\om)\embed C(0,T;H^2\om),\qquad \Wc=\Yc=L^2(0,T;L^2\om),\\
& V=\widetilde{V}=H=H^2\om\cap H^1_0\om, \quad W=Y=L^2\om,\\
& X_c= L^2\om,\quad X_a= H^2\om %: 0<\underline{a}\leq a\leq \overline{a}\text{ a.e. on } \Omega\}
,\quad X_\varphi= L^2\om,\quad U_0=H^2\om.
\end{align*}
The following statements are true:
\begin{enumerate}[label=(\roman*)]
\item The minimization problem admits minimizers.
\item The corresponding model operator $\Gc$ is \commB{G\^ateaux} differentiable with locally bounded $\Gc'$.
\item The adjoint of the derivative operator is given by
\begin{align*}
&\Gc'(\lambda,u,\theta)^*: \Wc\times H\times\Yc\to X\times\Vc\times\Theta\\
&\Gc'(\lambda,u,\theta)^*=
\begin{pmatrix}
-F'_\lambda(\lambda,u)^* &  0 & 0\\
\left(\frac{d}{dt}-F'_u(\lambda,u)-\Nc_u'(u,\theta)\right)^* & (\cdot)_{t=0}^* & M^*\\
-\Nc_\theta'(u,\theta)^* & 0 &0
\end{pmatrix}
=:(g_{i,j})_{i,j=1}^3
\end{align*}
with
\begin{alignat*}{3}
&F'_\lambda(\lambda,u)^*:\Wc\to X, \qquad && F'_u(\lambda,u)^*: \Wc\to \Vc, \qquad &&(\cdot)_{t=0}^*:H\to \Vc\\
&\Nc_\theta'(u,\theta)^*: \Wc\to \Theta,\ && \Nc_u'(u,\theta)^*:\Wc\to \Vc, && M^*: \Yc\to\Vc.
\end{alignat*}
\end{enumerate}
By defining $\Dinv: L^2(\Omega)\ni k^z\mapsto \ztil\in H^2(\Omega)\cap H^1_0(\Omega)$ %\todo{define corresponding spaces} 
such that $\ztil$ solves
\begin{align}\label{auxiliaryPDEs}
\begin{cases}
-\Delta\ztil&=z_1 \quad \text{in } \Omega\\
\quad\ztil&=0 \quad \text{ on }\partial\Omega
\end{cases},
\qquad
\begin{cases}
-\Delta z_1+z_1&=k^z \quad \text{in } \Omega\\
\qquad\quad z_1&=0 \quad \text{ on }\partial\Omega,
\end{cases}
\end{align}
we can write explicitly
{\allowdisplaybreaks
\begin{align}
&g_{2,2}: \quad (\cdot)^*_{t=0}h=h, \label{adjoint-t0}\\[1ex]
&g_{2,3}: \quad M^*z(t)=\int_0^T(t+1)\Dinv z(t)\wrt t-\int_0^t(t-s)\Dinv z(s)\,ds, \label{adjoint-M}\\[1ex]
&g_{2,1}: \quad \left(\frac{d}{dt}-F'_u(\lambda,u)-\Nc_u'(u,\theta)\right)^*z(t) \nonumber\\
&\qquad\quad=\int_0^T(t+1)\Dinv \Ktil z(t)\wrt t-\int_0^t\Dinv [(t-s)\Ktil z(s)-z(s)]\,ds \nonumber\\
&\qquad\quad\text{with } \Ktil = -\nabla\cdot(a\nabla\cdot)+c-\Nc_u'(u,\theta)\text{ and } \Nc_u' \text{ is computed as in Lemma } \ref{NN-Lipschitz}, \label{adjoint-FuNu}\\[1ex]
&g_{1,1}: \quad-F'_\lambda(\lambda,u)^*z=
\begin{cases}
\int_0^T z(t)u(t)\wrt t \qquad &\text{for } \lambda=c\\[1.5ex]
\int_0^T -z(t)\wrt t \qquad &\text{for } \lambda=\varphi\\[1.5ex]
\int_0^T \Dinv(-\nabla\cdot(z\nabla u))(t)\wrt t \qquad\quad& \text{for } \lambda=a,\\
\end{cases} \label{adjoint-Flambda}
\end{align}
$g_{3,1}$: one has the recursive procedure
\begin{align}\label{adjoint-Ntheta}
&\delta_L := 1, \qquad \delta_{l-1}:= {a'}^T_{l-1} \weight^T_l \delta_l,  \quad\qquad l=L\ldots 2, \nonumber\\
&\nabla_{\weight_{l-1}}\Nc(u,\theta)^*z= \int_0^T\int_\Omega \delta_{l-1}  a_{l-2}^T \,z \wrt x\wrt t,\\
&\nabla_{\bias_{l-1}}\Nc(u,\theta)^*z= \int_0^T\int_\Omega \delta_{l-1}\,z \wrt x\wrt t, \nonumber
\end{align}}
with $a_l,a'_l$ detailed in the proof.
\end{prop}

\begin{proof}
Assertion i) follows from Remark \ref{rem-Hilbertspace-app}.
Using Proposition \ref{prop-Differentiability}, Assertion ii) can be shown similarly as in Lemma \ref{lem-differentiability}.
%Assertion ii) follows from Proposition \ref{prop-Differentiability}, Lemma \ref{lem-differentiability}.
The proof for assertion iii) is presented in Appendix \ref{appendix-proof-adjoint}.
\end{proof}

\begin{cor}[Discrete measurements]\label{prop-adjoints-dis}
In case of discrete measurements $M_i:\Vc\to Y, M_i(u)=u(t_i), t_i\in(0,T)$, where the pointwise time evaluation is well-defined as $\Vc\embed C(0,T;H^2(\Omega))$, %\todo{mention that point-evaluations are well-defined for elements in $\Vc$ and operator is continuous (reference), again define involved quantities first} \note{done}
the adjoint $g_{2,3}$ is modified as follows. For $h\in Y$,
\begin{align*}
&(h,v(t_i))_\ltn=(\tilde{h},v(t_i))_\htn=\int_0^{t_i}(-\uhddot(t),v(t))_\htn\wrt t+(\tilde{h},v(t_i))_\htn\\
&=\int_0^{t_i}(\uhdot(t),\dot{v}(t))_\htn\wrt t+(u^h(0),v(0))_\htn-(\uhdot(t_i)-\tilde{h}(t),v(t_i))_\htn+(\uhdot(0)-u^h(0),v(0))_\htn\\
&=(u^h,v)_{H^1(0,t_i;H^2(\Omega))}=(u^h,v)_\Vc,
\end{align*}
provided that $u^h=$ const in $[t_i,T]$ in order to form the integral of the full time line $(0,T)$ in the last line. %\todo{where we set $u^h$ to be constant in $[t_i,T]$?}. \note{done}
Above, $h,\tilde{h}$ are respectively in place of $k^z$ and $\ztil$ in \eqref{auxiliaryPDEs}; besides,  $u^h$ solves
\begin{equation}
\begin{split}
&\uhddot(t)=0 \qquad t\in(0,t_i)\\
&\uhdot(t_i)=\tilde{h}, \quad \uzdot(0)-u^z(0)=0.
\end{split}
\end{equation}
Thus we arrive at 
\begin{align}\label{discrete}
(M_i)^*h=u^h(t)=
\begin{cases}
\Dinv h(t+1) \qquad& 0<t\leq t_i\\
\Dinv h(t_i+1) & t_i<t\leq T.
\end{cases}
\end{align}
This shows a numerical advantage of processing discrete observations in an Kaczmarz scheme, for instance in deterministic or stochastic optimization. To be specific, for each data point in the forward propagation, thanks to the all-at-one approach, no nonlinear model needs to be solved; in the backward propagation, by the same reason and \eqref{discrete}, one needs to compute the corresponding adjoint only for small time intervals.
\end{cor}

\subsection{Numerical results}
This section is dedicated to a range of numerical experiments carried out in two parallel settings: by way of analytic adjoints in Section \ref{sec:numerical-analytic}, and with Pytorch in Section \ref{sec:numerical-pytorch}.
\commA{While, in our experiments, we evaluate and compare the proposed method for different settings, such as varying the number of time measurements or noise, we highlight that the main purpose of these experiments is to show numerical feasibility of the proposed approach in principle, rather than providing highly optimized results. In particular, a tailored optimization of, e.g., regularization parameters and initialization strategies involved in our method might still be able to improve results significantly.}

For both settings \commA{(analytic adjoints and Pytorch)}, we use the following learning-informed PDE as special case of the one considered in Proposition \ref{prop-adjoints-cont}:%\todo{sign flip of $\Nt$ compared to code, deal with decomposition $\varphi = \psi + \phi$ and of $\Nt$}
\begin{equation}\label{eq:numerics}
\begin{alignedat}{3}
%& \dot{u}-\Delta u -  (\varphi + \phi)-\Nt(u) =0 \quad&&\mbox{ in }\Omega\times\ti\\
& \dot{u}-\Delta u - \varphi -\Nt(u) =0 \quad&&\mbox{ in }\Omega\times\ti\\
& u(0) = u_0=0 &&\mbox{ in }\Omega,
\end{alignedat}
\end{equation}
%\commE{The residual term $\phi$ is added solely for the purpose of conveniently having an analytic solution.}
%\todo{I removed all mentions of $\phi$, as 1) it is not explained or used at all further in the paper and will lead to questions from the reviewers, and 2) it causes confusion with how $\phi$ is used previously as a nonlinearity.}

We deal with  time-discrete measurements as in Corollary \ref{prop-adjoints-dis}, i.e., we use a time-discrete measurement operator $M:\Vc \rightarrow L^2(\Omega)^{n_T}$, with $n_T \in\N$, given as $M(u)_{t_i} = u(t_i)$ for $t_0 = 0$ and $t_i \in (0,T)$ with $i=1,\ldots,n_T-1$. We further let a noisy measurement of the initial state $u_0$ be given at timepoint $t=0$. % \textcolor{red}{and include $u_0$ only implicitly as parameter via $u(0)$.}\todo{Comment: What does this mean? Response: With this I meant that we did not explicitly include $u_0$ as parameter and added the constraint $u(0) = u_0$ in the code, but rather just left $u(0)$ vary unconstrained.\\ TN: for my tests I kept $u(0)=0$ fixed as it is a part of the model rather than the data. I think we can end the explanation without "and include $u_0$ only implicitly as parameter via $u(0)$".} 
Further, we consider two situations:
\begin{enumerate}
\item The source $\varphi$ in \eqref{eq:numerics} is fixed; we estimate the state $u$ and the nonlinearity $\Nt$ only, yielding a model operator $\Gc_\varphi:H^1(0,T;H^2\om\cap H^1_0\om) \times \Theta\to L^2(0,T;L^2\om) \times L^2(0,T;L^2\om)$ given as
\[ \Gc_\varphi(u,\theta) =  
\begin{pmatrix}
%\dot{u}-\Delta u -  (\varphi + \phi) -\Nt(u) \\
\dot{u}-\Delta u -  \varphi - \Nt(u) \\
Mu
\end{pmatrix}.
\]
\item The source $\varphi$ in \eqref{eq:numerics} is unknown, and we estimate the state $u$, the source $\varphi$ and the nonlinearity $\Nt$. This results in a model operator $\Gc:L^2(\Omega) \times H^1(0,T;H^2\om\cap H^1_0\om) \times \Theta\to L^2(0,T;L^2\om) \times L^2(0,T;L^2\om)$ given as
\[ \Gc(\varphi,u,\theta) =  
\begin{pmatrix}
%\dot{u}-\Delta u -  (\varphi + \phi)-\Nt(u) - \phi \\
\dot{u}-\Delta u - \varphi - \Nt(u) \\
Mu
\end{pmatrix}.
\]
%\todo{Comment: Both situations use the source decomposition $\to$ mention this or skip in both cases, also unify notation $\phi,\varphi$ in title of figs. Suggest: Move the fixed $\phi$ to \eqref{eq:numerics} and use $\varphi$ as unknown from thereon to be consistent with the whole paper. Response: Done, please check if you agree.\\ TN: I rephrased  a little bit to minimize the number of changes. Thus, I think for the rest, e.g. in \eqref{eq:numerics_minprob_state_par_net}, we do not need to signal that we change notation from $\varphi$ to $\phi$.}
\end{enumerate}
For these two settings, the special case of the learning problem \eqref{eq:main_identification_aao_setting} we consider here is given as
\begin{equation}\label{eq:numerics_minprob_state_net}
\min_{\substack{
(u^k)_k \in \Vc  \\
\theta \in \Theta
}} \sum_{k=1}^K \left( \| \Gc_\varphi(u^k,\theta) - (0,y^k) \|^2_{\Wc\times\Yc} + \|u^k\|_{\Vc}^2 \right) + \|\theta\|_2 ^2,
\end{equation}
for state- and nonlinearity identification and

\begin{equation}\label{eq:numerics_minprob_state_par_net}
\min_{\substack{
(\varphi^k,u^k)_k \in L^2(\Omega) \times \Vc \\
\theta \in \Theta
}} \sum_{k=1}^K \left( \| \Gc(\varphi^k,u^k,\theta) - (0,y^k) \|^2_{\Wc\times\Yc} + \|u^k\|_{\Vc}^2 + \|\varphi\|_{L^2(\Omega)}^2   \right) + \|\theta\|_2 ^2
\end{equation}
for state-, parameter and nonlinearity identification.

%Considering 
It is clear that identifying both the nonlinearity and the state introduces some ambiguities, since the PDE is for instance invariant under a constant offset in both terms (with flipped signs). To account for that, we always correct such a constant offset in the evaluation of our results. As the following remark shows, at least if the state $u$ is fixed appropriately, a constant shift is the only ambiguity that can occur. %\note{added $\frac{\partial}{\partial t}u(x,t)\neq 0$ as assumption, shortened a bit}
%We note the following regarding the ambiguities in the joint reconstruction of the nonlinearity and the parameter.\todo{add details}
\begin{rem}[Offsets]
With $\Omega_y:=u(\Omega\times(0,T))$ \commA{the range of $u$ for all $x\in\Omega, t\in(0,T)$, and given that} $\frac{\partial}{\partial t}u(x,t)\neq 0$, consider any solutions $f: \Omega_y\to\R$, $\varphi:\Omega\to\R$ of \eqref{ex0}. Then all solutions of \eqref{ex0} are on the form
$$
	\tilde{f}(y) := f(y) + c, \qquad \tilde{\varphi}(x) := \varphi(x) - c, \qquad c\in\R.
$$
Indeed, assume $\tilde{f}$, $\tilde{\varphi}$ are solutions, and define $g(y):=\tilde{f}(y) - f(y)$, $\Phi(x):=\tilde{\varphi}(x) - \varphi(x)$. Since these are solutions, one has $0 = g(u(x,t)) + \Phi(x)$ for all $(x,t)$ such that
%$$
%	0 = \tilde{f}(u(x,t)) + \tilde{\varphi}(x) = f(u(x,t)) + g(u(x,t)) + \varphi(x) + \Phi(x) = g(u(x,t)) + \Phi(x).
%$$
%Accordingly,
$$
	0 = -\frac{\partial}{\partial t}\Phi(x) = \frac{\partial}{\partial t}g(u(x,t)) = g'(u(x,t))\frac{\partial}{\partial t}u(x,t).
$$
As $\frac{\partial}{\partial t}u(x,t)\neq 0$ on $\Omega\times(0,T)$, it follows that $g'(y)\equiv 0$ on $u(\Omega\times(0,T))$, that is, there is some $c\in\R$ such that $c=g(u(x,t))=-\Phi(x)$ for all $(x,t)\in\Omega\times(0,T)$.

Moreover, finding \emph{any} solutions $f$, $\varphi$ and setting
$$
	c := \frac{\int_\Omega\varphi(x)\, \wrt x - \int_{\Omega_y}f(y)\,dy}{|\Omega| + |\Omega_y|}
$$
yields solutions $\tilde{f}(y):=f(y)+c$, $\tilde{\varphi}(x):=\varphi(x)-c$, minimizing $\|\varphi\|_{L^2(\Omega)}^2 + \|f\|_{L^2(\Omega_y)}^2$.
\end{rem}

\begin{rem}[Different measurement operators] \label{rem:measurement_operators}
\commA{In our experiments, we use a time-discrete measurement operator, and at times where data was measured, we assume measurements to be available in all of the domain. As will be seen in the next two subsections, reconstruction of the nonlinearity is possible in this case even with rather few time measurements.
A further extension of the measurement setup could be to use partial measurements also in space. While we expect similar results for approximately uniformly distributed partial measurements in space, highly localized measurements such as boundary measurements and measurements on subdomains are more challenging. In this case, we expect the reconstruction quality of the nonlinearity to strongly depend on the range of values the state $u$ admits in the observed points, but given the analytical focus of our paper, we leave this topic to future research.}
\end{rem}

\paragraph{Discretization.} In all but one experiment \commA{(in which we test different spatial and temporal resolutions)}, we consider a time interval $T=[0,0.1]$, uniformly discretized with $50$ time steps, and a space domain $\Omega = (0,1)$, uniformly discretized with $51$ grid points. The time-derivative as well as the Laplace operator was discretized with central differences. For the neural network $\Nt$, we consider a fully-connected network with $\tanh$ activation functions, and three single-channel hidden layers of width $[2,4,2]$ for all experiments. 
\commA{Note that this network architecture was chosen empirically by evaluating the approximation capacity of different architectures with respect to different nonlinear functions. For the sake of simplicity, we choose a simple, rather small architecture (satisfying the assumptions of our theory) for all experiments considered in this paper. In general, the architecture (together with regularization of the network parameters) must be chosen such that a balance between expressivity and overfitting may be reached (see for instance \cite[Sections 1.2.2 and 3]{Kutyniok21math_deep_learning}), but a detailed evaluation of different architectures is not within the scope of our work.
}
%The width of the hidden layers differs between the experiments with analytic adjoints and with pytorch, and will be specified in the corresponding sections.
%\todo{clarify numerical integration?} \commA{see paragraph \emph{PDE and adjoints.}}

\subsubsection{Implementation with analytic adjoints}\label{sec:numerical-analytic}
\paragraph{Set up.} In what follows, we apply Landweber iteration to solve the minimization problem \eqref{eq:main_identification_aao_setting}. The Landweber algorithm is implemented with the analytic adjoints computed in  Proposition \ref{prop-adjoints-cont} and Corollary \ref{prop-adjoints-dis}, ensuring that the backward propagation maps to the correct spaces. %We present two examples whose unknown physical parameters are the source term $\varphi$, the state $u$ the and nonlinearity $f$.

\emph{PDE and adjoints.} 
We employed finite difference methods to numerically compute the derivatives in the PDE model, as well as in the adjoints outlined in Proposition \ref{prop-adjoints-cont} and Corollary \ref{prop-adjoints-dis}. In particular, central difference quotients were used to approximate time and space derivatives. For numerical integration, we applied the trapezoidal rule. The inverse operator $\Dinv$ constructed in \eqref{auxiliaryPDEs} is called in each Landweber iteration.

\emph{Neural network.} %We chose networks with 3 hidden layers of [2,4,2] neurons to represent the unknown nonlinearity $f$. 
In the examples considered, $f: u(x)\mapsto f(u(x))$ is a real-valued smooth function, hence the suggested simple architecture with 3 hidden layers of $[2,4,2]$ neurons is appropriate. As the reconstruction is carried out in the all-at-once setting, the hyperparameters were estimated simultaneously with the state. The iterative update of the hyperparameters is done in the recursive fashion  \eqref{adjoint-Ntheta}.

\emph{Data measurement.} We work with measured data $y$ as limited snapshots of $u$ (see Corollary \ref{prop-adjoints-dis}) and evaluated examples in the case of no noise and $\delta=3\%$ \commB{relative} noise. \commB{Noise  $\epsilon$ is sampled from a Gaussian distribution $\Nc(0,1)$, and the measured data is $y=u+\delta\epsilon(\|u\|_2/\|\epsilon\|_2)$.}

\emph{Error.} Error between the reconstruction and the ground truth was measured in the corresponding norms, i.e. $X_\varphi$-norm for $\varphi$ and $\Wc$-norm for the PDE residual and the error of $f$. For $u$, $\Vc$-norm is the recommended measure; for simplicity, we displayed $L^2$-error.

\emph{Minimization problem.} The regularization parameters are $R_u=R_\varphi$ and $M_i(u)=10\,u(t_i)$ (c.f Corollary \ref{prop-adjoints-dis}). We implement an adaptive
Landweber step size scheme, i.e. if the PDE residual in the current step decreases, the step size is accepted, otherwise it is bisected. For noisy data, the iterations are terminated after a stopping rule via a discrepancy principle (c.f. \cite{KalNeuSch08}) is reached. 

\paragraph{Numerical results.} %Figure \ref{ex-linear-fulam} shows the result of identifying $\varphi,u$ and $f$ simultaneously. The top left panel (we denote by panel $(1,1)$) depicts the parameter reconstruction, where we observe that although initializing from zero, the reconstructed $\varphi$ perfectly fits the ground truth. Panels $(2,1),(2,2)$ display the evolution of the exact and estimated state $u$ in $(t,x)$. Their difference is projected in panel $(2,3)$, which is noticeably almost a zero plane. 
%In the bottom panels, the unknown nonlinearity $f$ is represented by a network of 3 hidden layers with $[2,3,2]$ neurons. In panel $(3,1)$, the learned nonlinearity $f$  is plotted on the range of $u$ as $f:[0, 0.5]\ni u(t,x)\mapsto f(u(t,x))$. The reconstruction nicely captures the main feature of the true $f$. The decreasing trends depicted in panel $(1,2), (1,3), (3,2)$ for the error in $\varphi,u,f$ confirm the approximation effect. Altogether, the PDE residual in panel $(3,3)$ numerically hints at the convergence of the cost functional to a minimizer. 

Figure \ref{ex-multidata} discusses the example where only a few snapshots of $u$ are measured; explicitly, we here have three measurements $y_j=u(t_j), j=1, 25, 50, n_T=3$. We test the performance using three datasets of differing source terms and states (i.e. $K=3$ in \eqref{eq:numerics_minprob_state_net}), but identical nonlinearity $f$. The top left panel (we denote by panel $(1,1)$) displays three measurements of dataset $u_1$, each line  here represents a plot of $u_1(t_i)$. The same plotting style applies for dataset 2 (panel $(1,2)$) and dataset 3 (panel $(1,3)$). The exact source $\varphi_i,i=1,2,3$ in three equations are given in panel $(2,1)$. In panel $(3,2)$, the nonlinearity $f$ is expressed via a network of 3 hidden layers with $[2,4,2]$ neurons. In this example, we identify $u_i, i=1,2,3$ (panels $(2,3-6)$) and $f$ (see Section \ref{sec:numerical-pytorch} for more experiments, including recovering physical parameters). The output errors in $f$ (panel $(3,3)$), $u$ (panels $(3,4-6)$) and PDE (panel $(2,3)$) hint at the convergence of the cost functional to a minimizer. The noisy case is presented in Figure \ref{ex-multidata-noise1}.

%\commA{to save space: drop case 5\% noise, combine Fig 1 and 2 in 1 page}

%This phenomenon is understandable if considered from the perspective of a training problem. With $(u_i,f(u_i)),i=1,2,3$ taking the role of training data, where the training inputs $u_i$ are not known exactly due to noisy measurements $y^\delta$, and the labels $f(u_i):=\dot{u}_i-\Delta u_i+cu_i$ are thus similarly imprecise. When both training input and labels are inaccurate, the imperfection in estimating $f$ is, therefore, predictable.

% Set 1
%\begin{figure} 
%\centering
%\includegraphics[width=1\columnwidth]{figs/file-Linear-Iden_NN-Parm}
%\caption{$f(u)=-u+1$. Find $f$. }\label{ex-linear-f}
%\end{figure}
%
%\begin{figure}
%\centering
%\includegraphics[width=1\columnwidth]{figs/file-Linear-Iden_NN_U-Parm}
%\caption{$f(u)=-u+1$. Find $f,u$. }\label{ex-linear-fu}
%\end{figure}

\begin{figure}[p] 
\centering
%\includegraphics[width=0.6\columnwidth]{figs/file-Linear-Iden_NN_U_lamda-Parm}
%\caption{$f(u)=-u+1$, noise free. Find $f,u,\lambda$. }\label{ex-linear-fulam}
%\vspace{0.5cm}
\includegraphics[width=1\columnwidth]{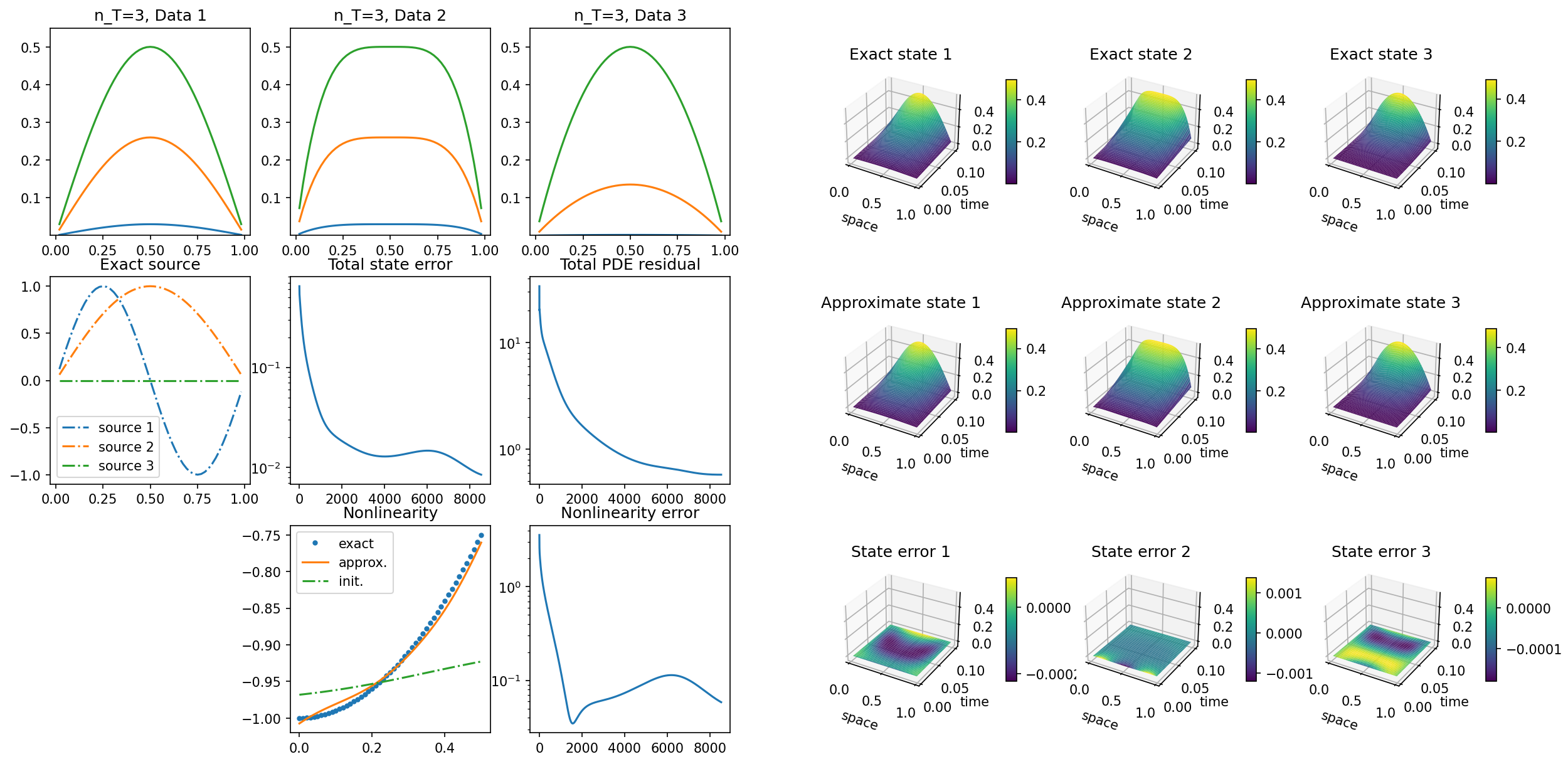}
\vspace{-0.5cm}
\caption{\commB{Numerical identification of state $u$ and ground-truth nonlinearity $f(u)=u^2-1$ in \eqref{eq:numerics} for three different values of the source term $\varphi$. In each case, three noise-free observations are given ($n_T=3$).} 
Plots 1-3 and 4-6 in the top line show the given data and the ground truth state for the three equations, respectively. The content of the remaining plots is described in the titles.
%Top left: given clean data, right: recovered state, bottom left: recovered nonlinearity (orange) compared to ground truth (blue).
}\label{ex-multidata}
\vspace{0.5cm}
\includegraphics[width=1\columnwidth]{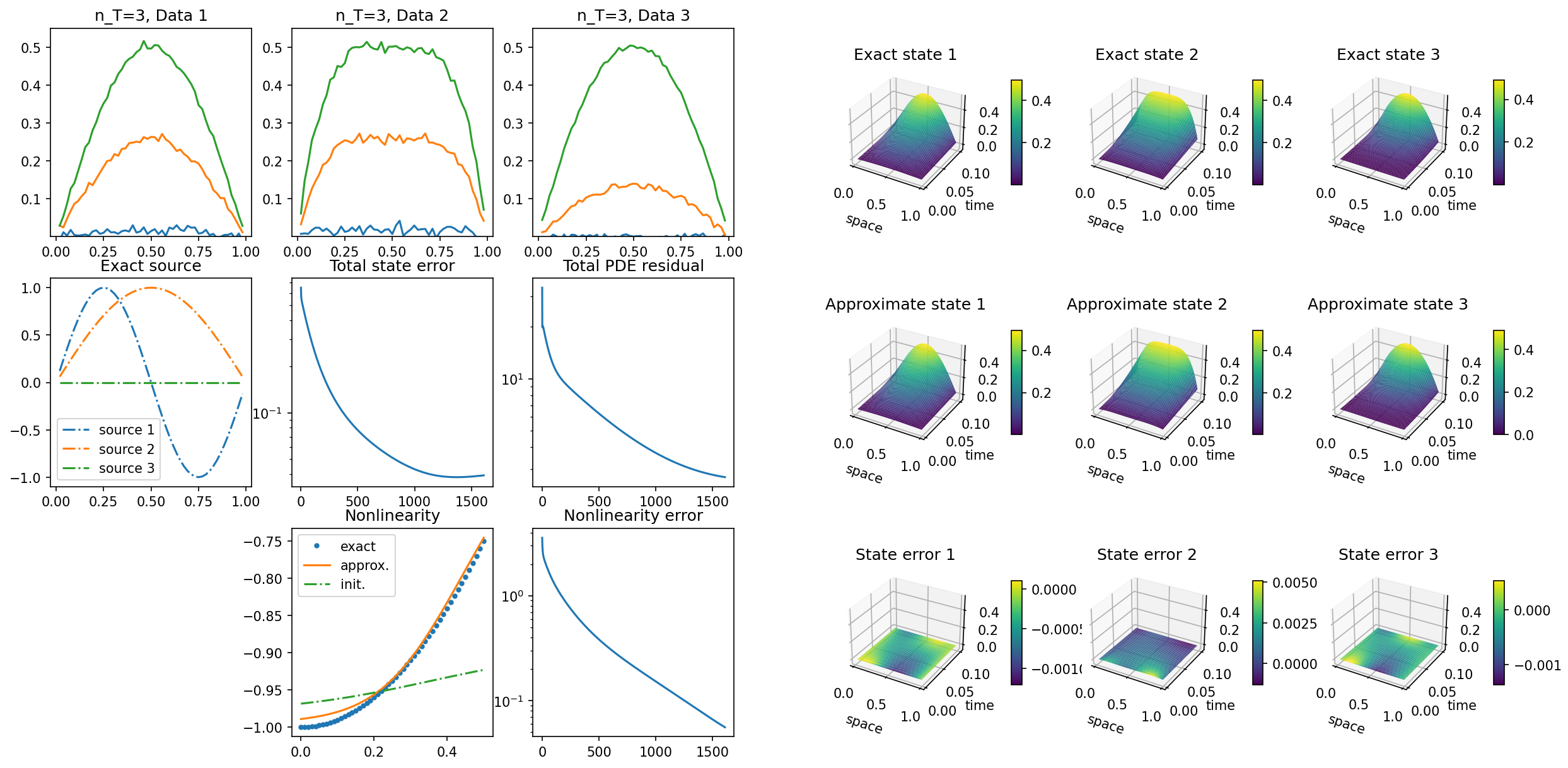}
\vspace{-0.5cm}
\caption{\commB{Numerical identification of state $u$ and ground-truth nonlinearity $f(u)=u^2-1$ in \eqref{eq:numerics} for three different values of the source term $\varphi$. In each case, three observations ($n_T=3$) with 3\% noise are given.} 
Plots 1-3 and 4-6 in the top line show the given data and the ground truth state for the three equations, respectively. The content of the remaining plots is described in the titles.
%Top left: given noisy data, right: recovered state, bottom left: recovered nonlinearity (orange) compared to ground truth (blue).
}\label{ex-multidata-noise1}

%\vspace{0.5cm}
%\includegraphics[width=0.7\columnwidth]{figs/file-3data_every25-noise3-Parm}
%\caption{$f(u)=u^2-1$, 5\% noise. Find $f,u$. }\label{ex-multidata-noise3}
\end{figure}

%\begin{figure}[p] 
%\centering
%\includegraphics[width=0.7\columnwidth]{figs/file-3data_every10-Parm}
%\caption{$f(u)=u^2-1$, noise free. Find $f,u$.}\label{ex-multidata}
%\vspace{0.5cm}
%\includegraphics[width=0.7\columnwidth]{figs/file-3data_every10-noise1-Parm}
%\caption{$f(u)=u^2-1$, 3\% noise.. Find $f,u$. }\label{ex-multidata-noise1}
%\vspace{0.5cm}
%\includegraphics[width=0.7\columnwidth]{figs/file-3data_every10-noise3-Parm}
%\caption{$f(u)=u^2-1$, 5\% noise. Find $f,u$. }\label{ex-multidata-noise3}
%\end{figure}

\subsubsection{Implementation with Pytorch}\label{sec:numerical-pytorch}

The experiments of this section were carried out using the Pytorch \cite{NEURIPS2019_9015}  package to numerically solve \eqref{eq:numerics_minprob_state_net} and \eqref{eq:numerics_minprob_state_par_net}. More specifically, we used the pre-implemented ADAM \cite{kingma2014adam} algorithm with automatic differentiation, a learning rate of $0.01$ and $10^4$ iterations for all experiments. \commA{In case noise is added to the data, we use Gaussian noise with zero mean and different standard deviations denoted by $\sigma$.}

\paragraph{Solving for state and nonlinearity}
In this paragraph we provide experiments for the learning problem with a single datum, where we solve for the state and the nonlinearity and test with increasing noise levels and \commA{reducing the number of observations}. We refer to Figure \ref{fig:quadratic_u_w} for the visualization of selected results, and to Table \ref{tbl:error_summary} (top) for error measures for all tested parameter combinations.

It can be observed that reconstruction of the nonlinearity works reasonable well even up to a rather low number of measurements together with a rather high noise level: The shape of the nonlinearity is reconstructed correctly in all cases except the one with three time measurements and a noise level of \commA{$\sigma = 0.1$}. 

\paragraph{Solving for parameter, state and nonlinearity}

In this section, we provide experiments for the learning problem with a single datum, where we solve for the parameter, the state and the nonlinearity and test with increasing noise levels and decreasing of observations. We refer to Figure \ref{fig:quadratic_u_w_phi} for the visualization of selected results and to Table \ref{tbl:error_summary} (bottom) for error measures for all tested parameter combinations.

It can again be observed that the reconstruction works rather well, in this case for both the nonlinearity and the parameter. Nevertheless, due to the additional degrees of freedom, the reconstruction breaks down earlier than in the case of identifying just the state and the nonlinearity.

\paragraph{Varying the discretization level}\commA{
In this paragraph, we test the result of different spatial and temporal resolution levels of the state. To this aim, we reproduce the experiment as in line 3 of Figure \ref{fig:quadratic_u_w_phi} (6 time measurements, $\delta=0.03$, quadratic nonlinearity, solving for nonlinearity and state) for $501 \times 500$ and $5001 \times 5000$ gridpoints in space $\times $ time (instead of $51$ as in the original example).}

\commA{The result can be found in Figure \ref{fig:quadratic_u_w_highres}. As can be observed there, changing the resolution level has only a minor effect on result, possibly slightly decreasing the reconstruction quality for the nonlinearity. We attribute this to the fact that the number of spatial grid points for the measurement was equally increased, see also Remark \ref{rem:measurement_operators} for a discussion of localized measurements.}

\paragraph{Reconstructing the nonlinearity from multiple samples}
In this paragraph we show numerically the effect of having different numbers of datapoints available, i.e., the effect of different numbers $K\in \N$ in \eqref{eq:numerics_minprob_state_par_net}. We again consider the identification of state, parameter and nonlinearity and use three time measurements and a noise level of $0.08$; a setting where the identification of the nonlinearity breaks down when having only a single datum available.

As can be observed in Figure \ref{fig:quadratic_u_w_phi_multsamples}, having multiple data samples improves reconstruction quality as expected. It is worth noting that here, even though each single parameter is reconstructed rather imperfectly with strong oscillations, the nonlinearity is recovered reasonable well already for three data samples. This is to be expected, as the nonlinearity is shared among the different measurements, while the parameter differs. 

\paragraph{Comparison of different approximation methods}
Here we evaluate the benefit of approximating the nonlinearity with a neural network, as compared to classical approximation methods. As test example, we consider the identification of the state and the nonlinearity only, using a noise level of 0.03 and 10 discrete time measurements. We consider four different ground-truth nonlinearities: $f(u) = 2 - u $ (linear), $f(u) = u^2 - 1$ (square), $f(u) = (u-0.1)(u-0.5)(141.6u-30)$ (polynomial) and $f(u) = \cos(3\pi u)$ (cosine). 

As approximation methods we use polynomials as well as trigonometric polynomials, where in both settings we allow for the same number ($=29$) of degrees of freedom as with the neural network approximation. For all methods, the same algorithm (ADAM) was used, and the regularization parameters for the state and the parameters of the nonlinearity were optimized by gridsearch to achieve the best performance.

The results can be seen in Figure \ref{fig:approx_comparision}. While each methods yields a good approximation in some cases, it can be observed that the polynomial approximation performs poorly both for the cosine-nonlinearity and the polynomial-nonlinearity (even tough the degrees of freedom would be sufficient to represent the later exactly). The trigonometric polynomial approximation on the other hand performs generally better, but produces some oscillations when approximating the square nonlinearity. The neural network approximation performs rather well for all types of nonlinearity, which might be interpreted as such that neural-network approximation is preferable when no structural information on the ground-truth nonlinearity is available. It should be noted, however, that due to non-convexity of the problem, this result depends many factors such as the choice of initialization and numerical algorithm. 
%While we strove to use a comparable setup for all methods (using, in particular, the same algorithm, initialization and regularization type), more tailed implementations for the different approximation types might yield different results.

\begin{figure}[p] 
\centering
\includegraphics[width=0.8\columnwidth]{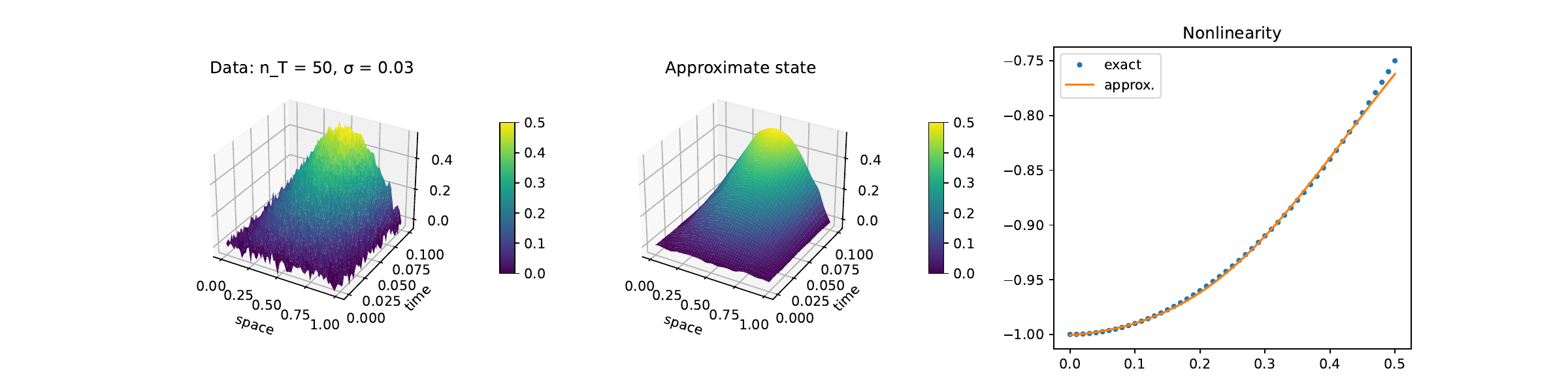}
\includegraphics[width=0.8\columnwidth]{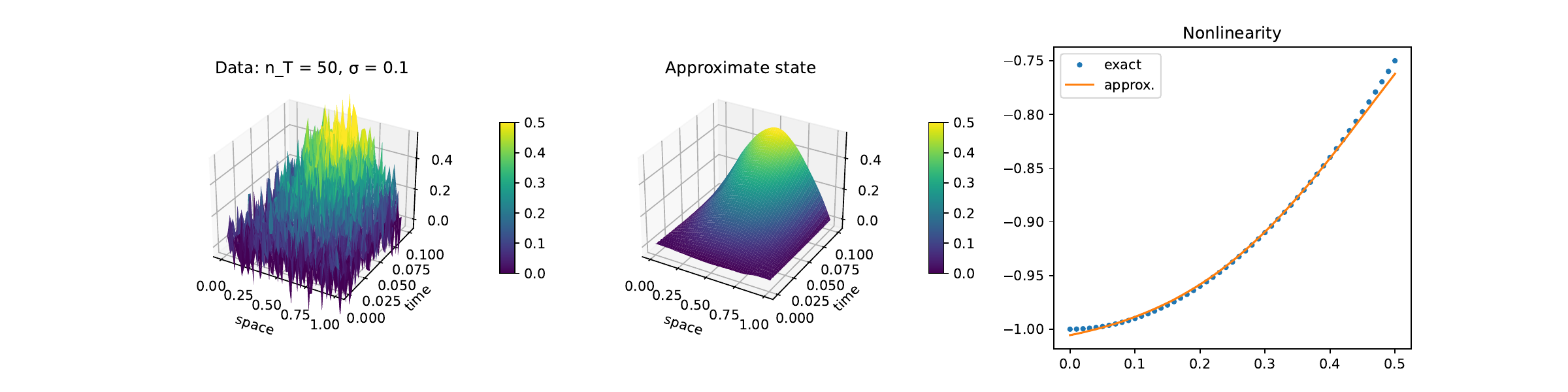}
\includegraphics[width=0.8\columnwidth]{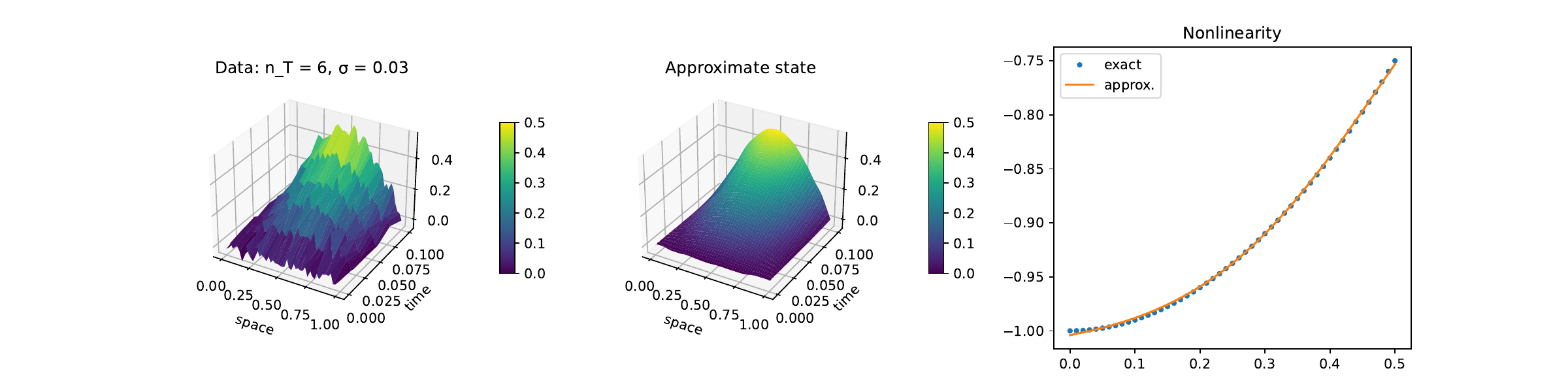}
\includegraphics[width=0.8\columnwidth]{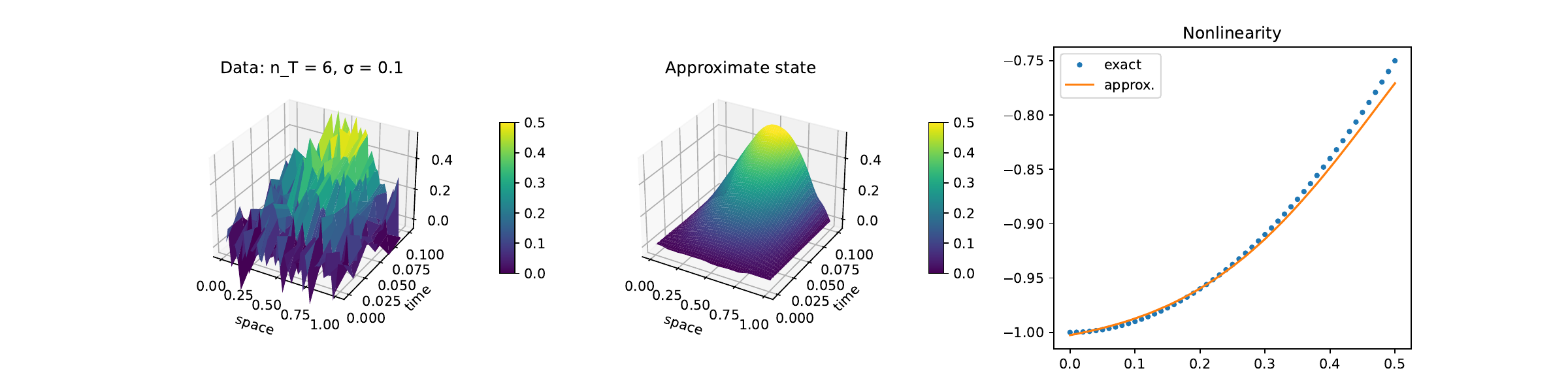}
\includegraphics[width=0.8\columnwidth]{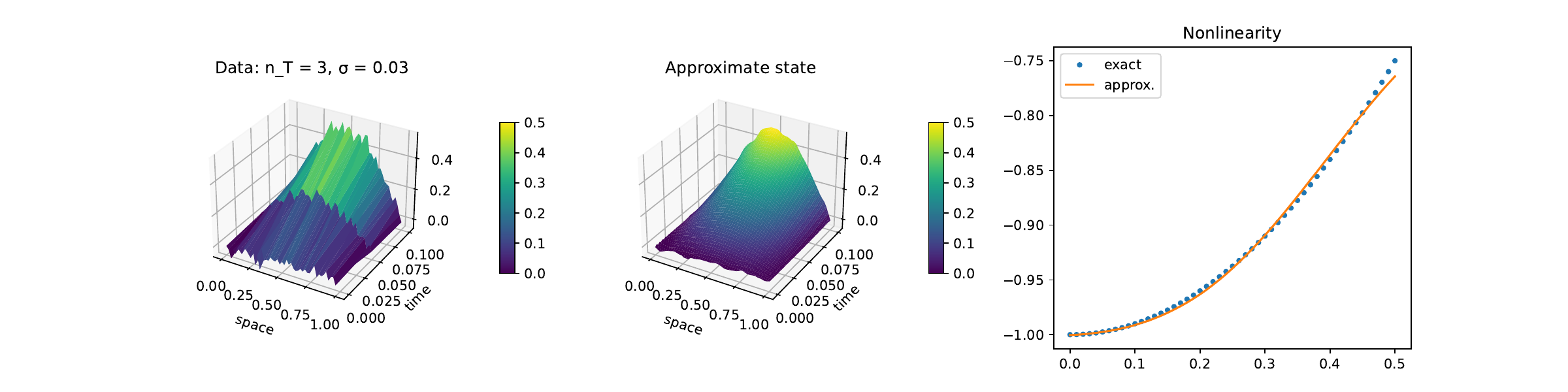}
\includegraphics[width=0.8\columnwidth]{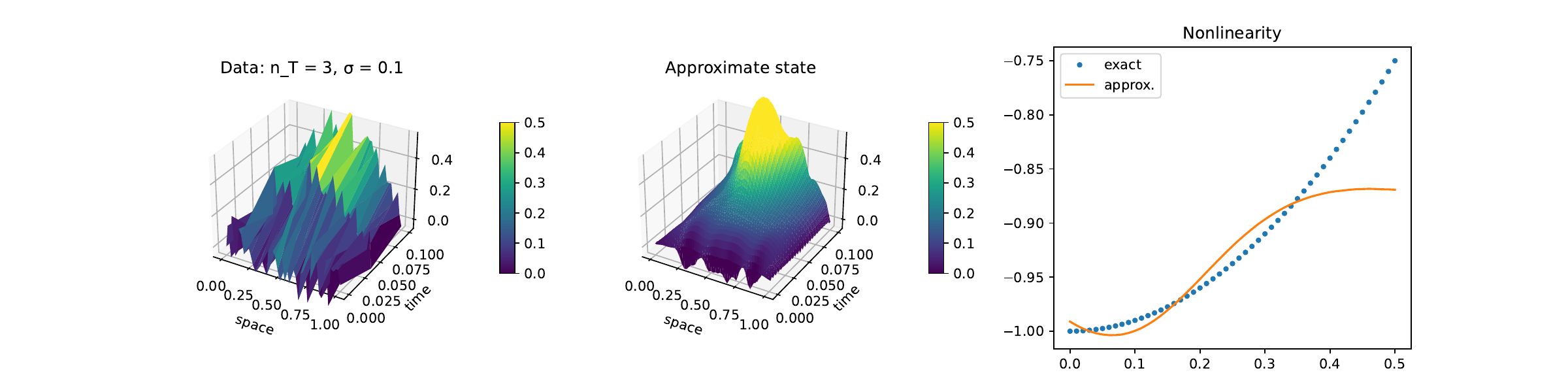}
\caption{\label{fig:quadratic_u_w} \commB{Numerical identification of state $u$ and ground-truth nonlinearity $f(u)=u^2-1$ in \eqref{eq:numerics}} for decreasing numbers of discrete observations (lines 1-2, 3-4 and 5-6) and increasing noise levels (even lines versus odd lines). Left: Given data, center: recovered state, right: recovered nonlinearity (orange) compared to ground truth (blue).}
\end{figure}

\begin{table}[p]
\begin{center}
\textbf{\commA{Recovering nonlinearity and state}}
\end{center}

\begin{tabularx}{\textwidth}{X>{\centering\arraybackslash}p{1.5cm}>{\centering\arraybackslash}p{1.5cm}>{\centering\arraybackslash}p{1.5cm}>{\centering\arraybackslash}p{1.5cm}>{\centering\arraybackslash}p{1.5cm}}

%0.01,0.03,0.05 0.1, 0.2

 & $\sigma = 0.01$ & $\sigma = 0.03$ & $\sigma = 0.05$ & $\sigma = 0.1$ & $\sigma = 0.2$ \\ \toprule 
 \textbf{Nonlinearity error} 
 \\ \midrule 
tmeas = 50 (= full)  & 1.94e-06 & 7.08e-06 & 1.22e-05 & 1.06e-05 & 1.91e-05 \\
tmeas =6  & 2.33e-06 & 1.89e-06 & 5.20e-06 & 4.07e-05 & 7.14e-05
 \\
tmeas = 3   & 3.58e-06 & 1.28e-05 & 6.03e-05 & 1.24e-03 & 1.51e-02
 \\ \midrule
\textbf{State error} \\ \midrule
tmeas = 50 (= full)  & 7.09e-06 & 1.68e-05 & 2.75e-05 & 4.53e-05 & 2.76e-05
 \\
tmeas = 6  & 7.45e-06 & 2.71e-05 & 2.04e-05 & 1.20e-04 & 1.52e-03
 \\
tmeas = 3    & 8.04e-06 & 2.40e-05 & 1.20e-04 & 7.70e-03 & 2.21e-02
 \\  \bottomrule
\end{tabularx}

\begin{center}
\textbf{\commA{Recovering nonlinearity, state and parameter}}
\end{center}

\begin{tabularx}{\textwidth}{X>{\centering\arraybackslash}p{1.5cm}>{\centering\arraybackslash}p{1.5cm}>{\centering\arraybackslash}p{1.5cm}>{\centering\arraybackslash}p{1.5cm}>{\centering\arraybackslash}p{1.5cm}}

%0.01,0.03,0.05 0.1, 0.2

 & $\sigma = 0.01$ & $\sigma = 0.03$ & $\sigma = 0.05$ & $\sigma = 0.08$ & $\sigma = 0.1$ \\ \toprule 
 \textbf{Nonlinearity error} 
 \\ \midrule 
tmeas = 50 (= full)  & 1.38e-06 & 3.97e-06 & 4.05e-06 & 1.85e-05 & 3.36e-05
 \\
tmeas =10  & 1.98e-06 & 8.25e-06 & 1.62e-05 & 1.22e-04 & 7.12e-01
 \\
tmeas = 6  & 4.22e-06 & 1.54e-05 & 3.86e-04 & 5.47e-04 & 5.33e-01
\\ \midrule
\textbf{Parameter error} \\ \midrule
tmeas = 50 (= full)   & 6.11e-05 & 1.15e-04 & 2.04e-04 & 3.59e-04 & 4.79e-04
 \\
tmeas = 10   & 1.44e-04 & 5.15e-04 & 9.13e-04 & 2.08e-03 & 4.89e-01
 \\
tmeas = 6  & 2.38e-04 & 7.23e-04 & 2.29e-03 & 4.36e-03 & 4.26e-01
 \\  \midrule
\textbf{State error} \\ \midrule
tmeas = 50 (= full)   & 1.73e-05 & 6.23e-05 & 1.63e-04 & 2.47e-04 & 3.24e-04
  \\
tmeas = 10  & 6.46e-05 & 1.91e-04 & 3.48e-04 & 8.45e-04 & 1.82e-02
\\
tmeas = 6   & 2.30e-04 & 4.44e-04 & 2.35e-03 & 3.48e-03 & 1.73e-02
 \\  \bottomrule
\end{tabularx}

\caption{\label{tbl:error_summary} Summary of errors in recovering nonlinearity and state (top) and in recovering nonlinearity, state and parameter (bottom) for different noise levels and different numbers of discrete measurements \commB{(denoted by \emph{tmeas})}.}

\end{table}

\begin{figure}[p] 
\centering
\includegraphics[width=0.97\columnwidth]{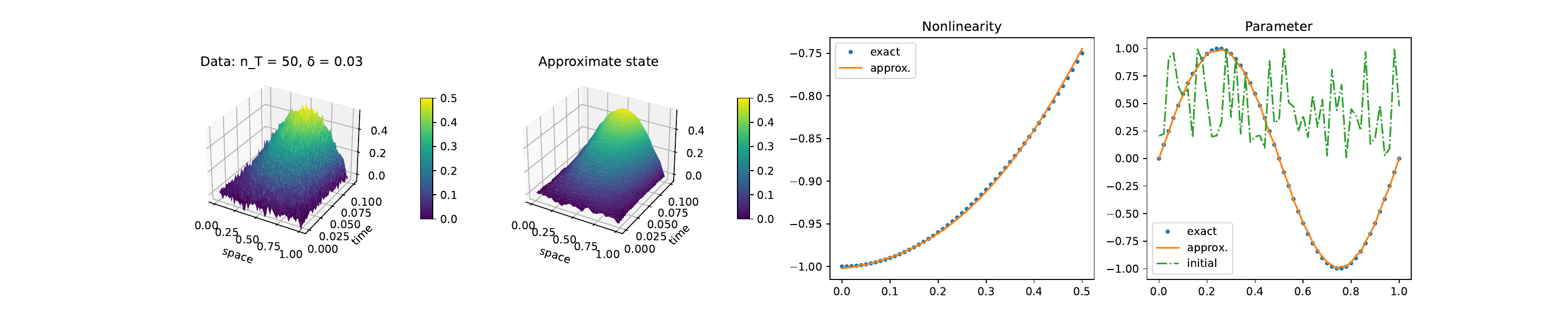}
\includegraphics[width=0.97\columnwidth]{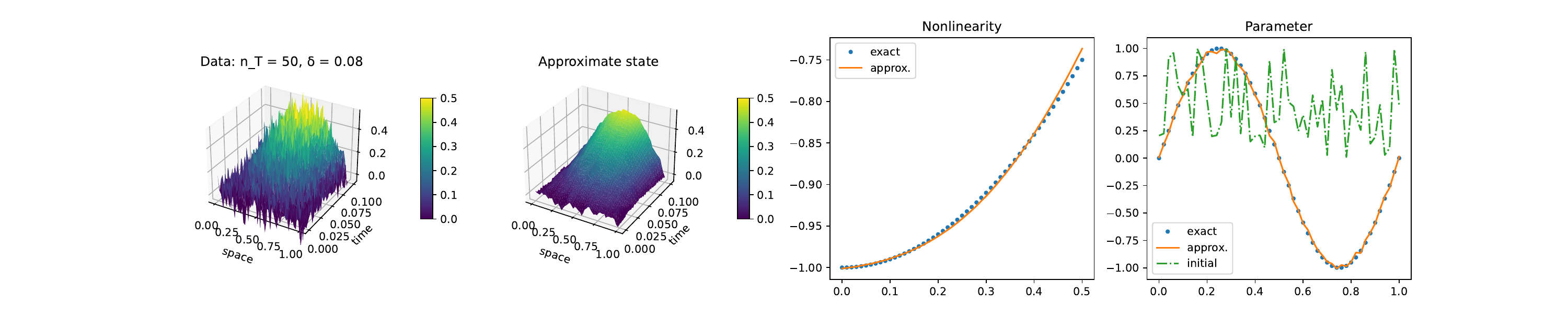}
\includegraphics[width=0.97\columnwidth]{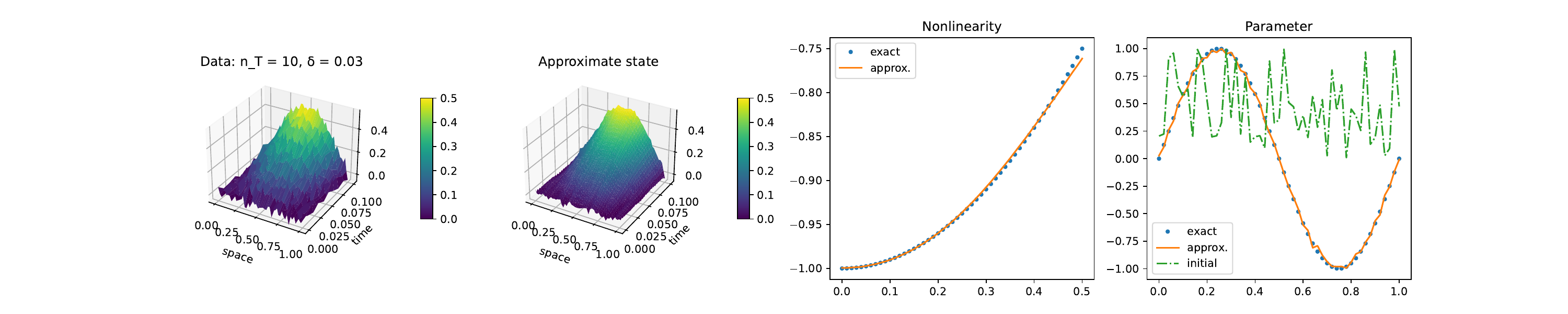}
\includegraphics[width=0.97\columnwidth]{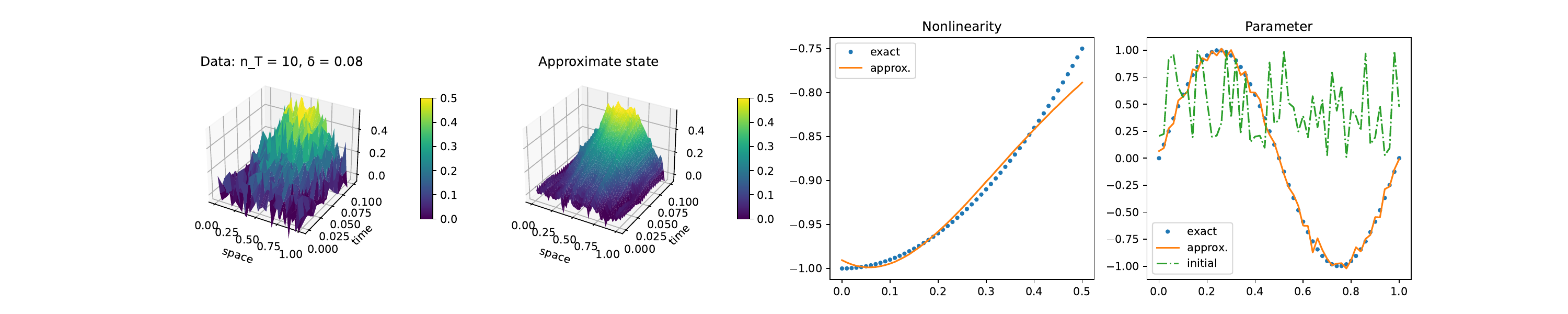}
\includegraphics[width=0.97\columnwidth]{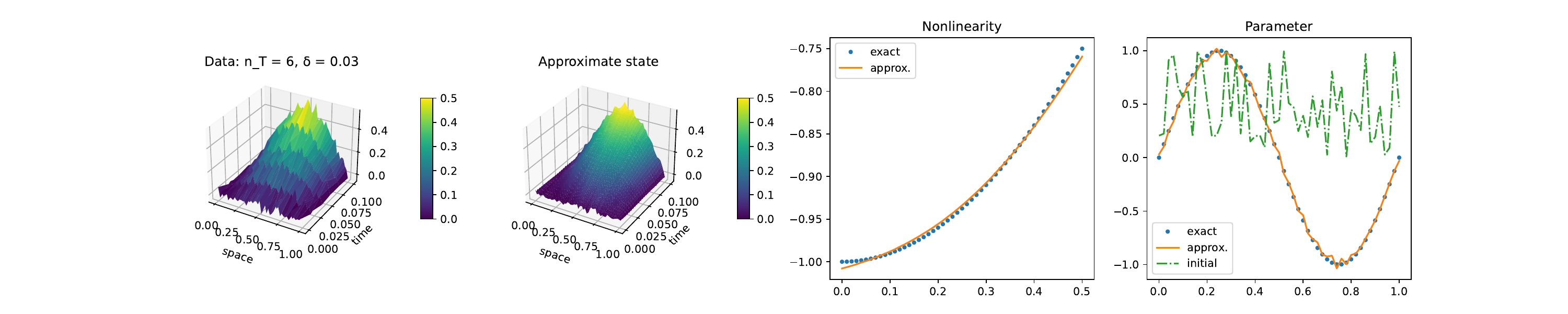}
\includegraphics[width=0.97\columnwidth]{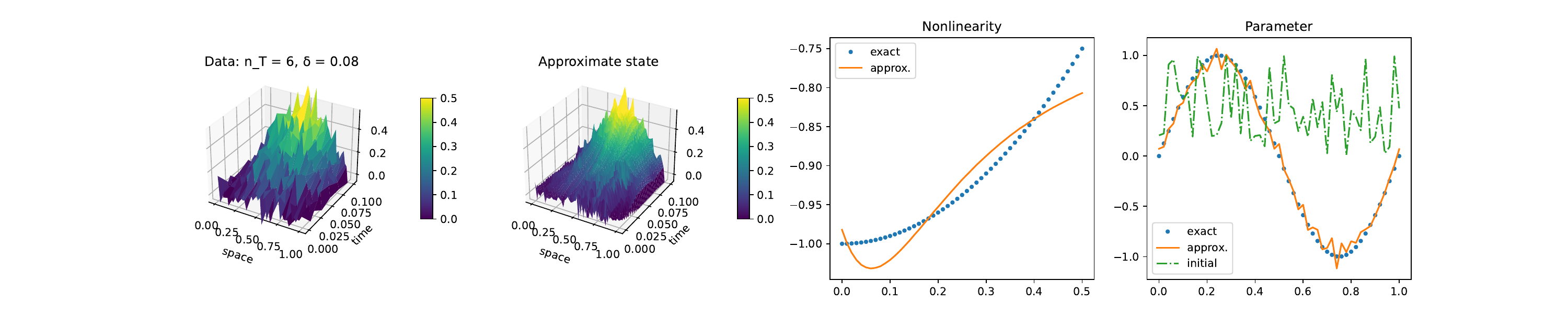}
\caption{\label{fig:quadratic_u_w_phi}
\commB{Numerical identification of state $u$, ground-truth nonlinearity $f(u)=u^2-1$ and the parameter $\varphi $ in \eqref{eq:numerics}} for decreasing numbers of discrete observations (lines 1-2, 3-4 and 5-6) and increasing noise levels (even lines versus odd lines). From left to right: Given data, recovered state, recovered nonlinearity (orange) compared to ground truth (blue), recovered parameter (orange) compared to ground truth (blue) and initialization (green)}
\end{figure}

%
%\begin{figure}[t] 
%\begin{tabularx}{\textwidth}{X>{\centering\arraybackslash}p{1.5cm}>{\centering\arraybackslash}p{1.5cm}>{\centering\arraybackslash}p{1.5cm}>{\centering\arraybackslash}p{1.5cm}>{\centering\arraybackslash}p{1.5cm}}
%
%%0.01,0.03,0.05 0.1, 0.2
%
% & $\sigma = 0.01$ & $\sigma = 0.03$ & $\sigma = 0.05$ & $\sigma = 0.08$ & $\sigma = 0.1$ \\ \toprule 
% \textbf{Nonlinearity error} 
% \\ \midrule 
%tmeas = 50 (= full) & 7.03e-05 & 2.02e-04 & 2.07e-04 & 9.43e-04 & 1.71e-03
% \\
%tmeas =10  & 1.01e-04 & 4.21e-04 & 8.28e-04 & 6.20e-03 & 3.63e+01
% \\
%tmeas = 6   & 2.15e-04 & 7.84e-04 & 1.97e-02 & 2.79e-02 & 2.72e+01
%\\ \midrule
%\textbf{Parameter error} \\ \midrule
%tmeas = 50 (= full)  & 3.12e-03 & 5.85e-03 & 1.04e-02 & 1.83e-02 & 2.44e-02
% \\
%tmeas = 10  & 7.33e-03 & 2.62e-02 & 4.66e-02 & 1.06e-01 & 2.50e+01
% \\
%tmeas = 6   & 1.21e-02 & 3.69e-02 & 1.17e-01 & 2.22e-01 & 2.17e+01
% \\  \midrule
%\textbf{State error} \\ \midrule
%tmeas = 50 (= full)  & 4.41e-02 & 1.59e-01 & 4.16e-01 & 6.29e-01 & 8.26e-01
%  \\
%tmeas = 10  & 1.65e-01 & 4.87e-01 & 8.87e-01 & 2.15e+00 & 4.65e+01
%\\
%tmeas = 6   & 5.86e-01 & 1.13e+00 & 6.00e+00 & 8.86e+00 & 4.42e+01
% \\  \bottomrule
%\end{tabularx}
%\caption{\label{tbl:quadratic_u_w_phi} Error in recovering the nonlinearity, the parameter and the state for different noise levels and different numbers of discrete measurements.}
%\end{figure}

\begin{figure}[p] 
\centering
\includegraphics[width=0.45\columnwidth, trim = 16cm 0.5cm 0 0.5cm, clip]{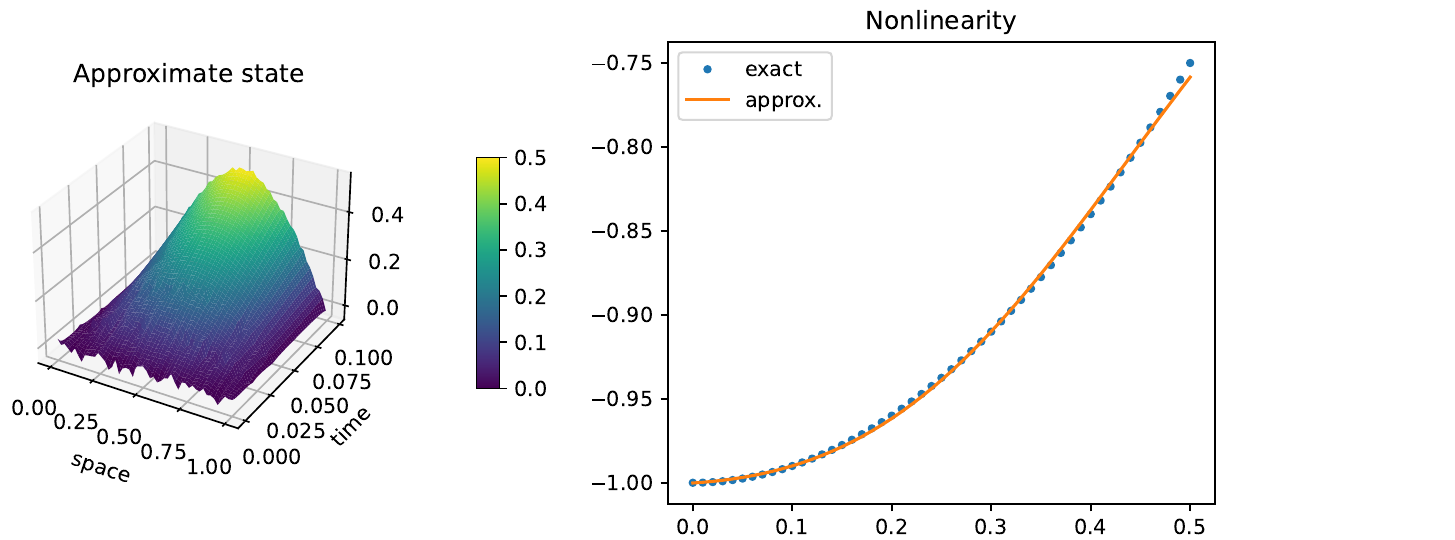}
\includegraphics[width=0.45\columnwidth, trim = 16cm 0.5cm 0 0.5cm, clip]{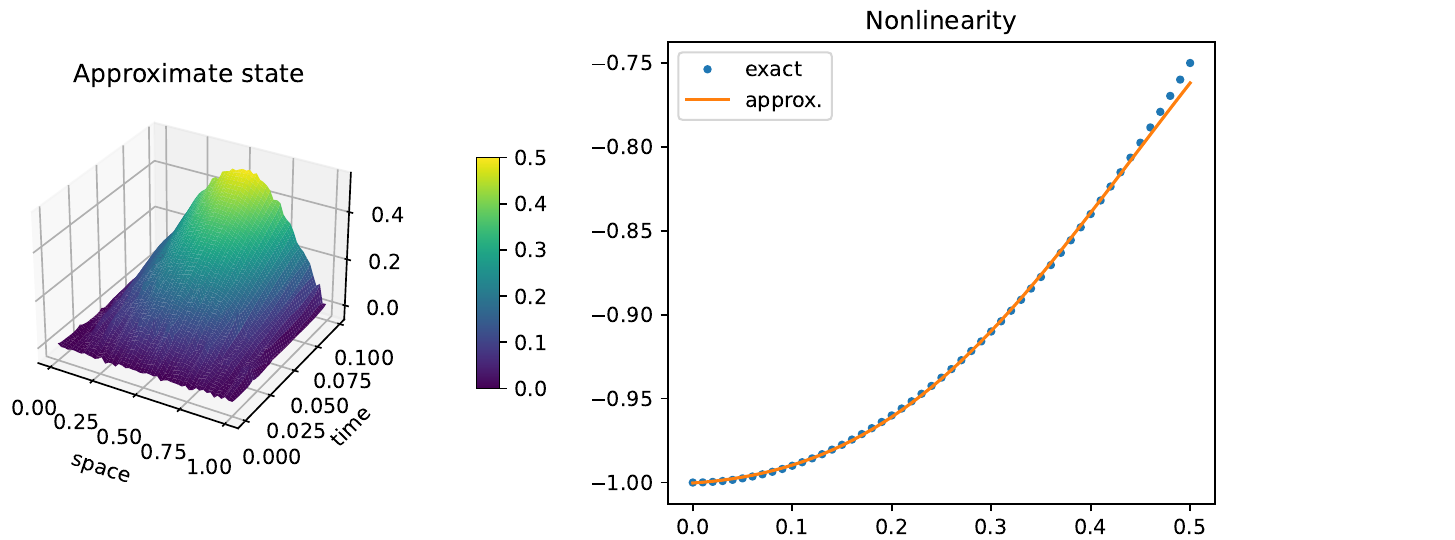}
\caption{\label{fig:quadratic_u_w_highres}
\commA{Identical setting as in line 3 of Figure \ref{fig:quadratic_u_w_phi} (6 time measurements, $\delta=0.03$, quadratic nonlinearity, solving for nonlinearity and state), but with different spatial $\times$ temporal resolution levels. Left to right: Plots 1 and 3: Approximate state obtained with $501 \times 500$ and $5001 \times 5000$ grid points, respectively. Plots 2 and 4: Recovered nonlinearity (orange) compared to ground truth (blue) for $501 \times 500$ and $5001 \times 5000$ grid points, respectively. The error in the nonlinearity is 3.60e-06 for $501 \times 500$ gridpoints and   5.74e-06  for $5001 \times 5000$ gridpoints (compare Table \ref{tbl:error_summary}).}}
\end{figure}

\newcommand\fsc{0.75}
\begin{figure}[p] 
\centering
\includegraphics[width=\fsc\columnwidth, trim = 0 0.5cm 0 0.5cm, clip]{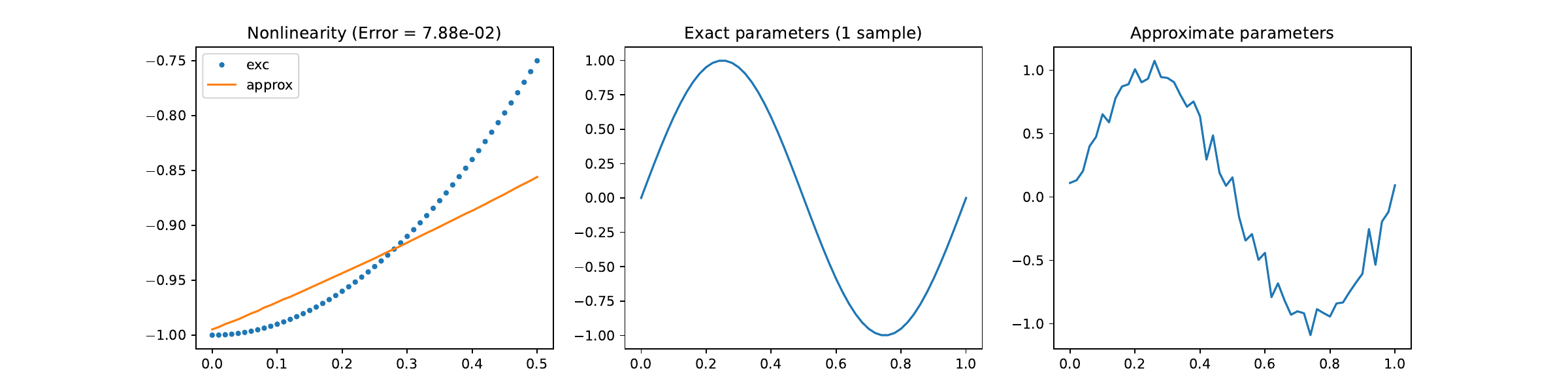}
\includegraphics[width=\fsc\columnwidth, trim = 0 0.5cm 0 0.5cm, clip]{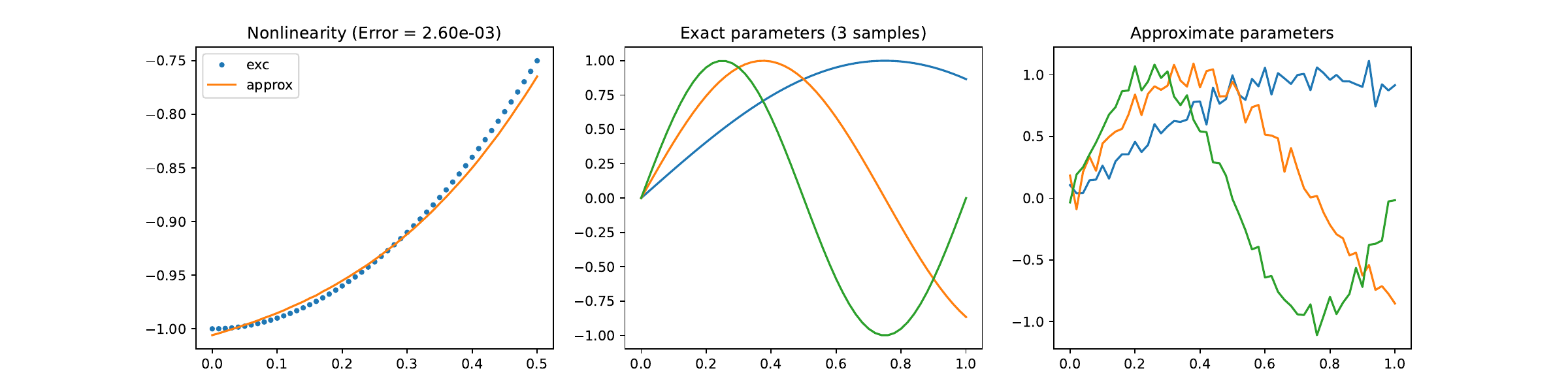}
\includegraphics[width=\fsc\columnwidth, trim = 0 0.5cm 0 0.5cm, clip]{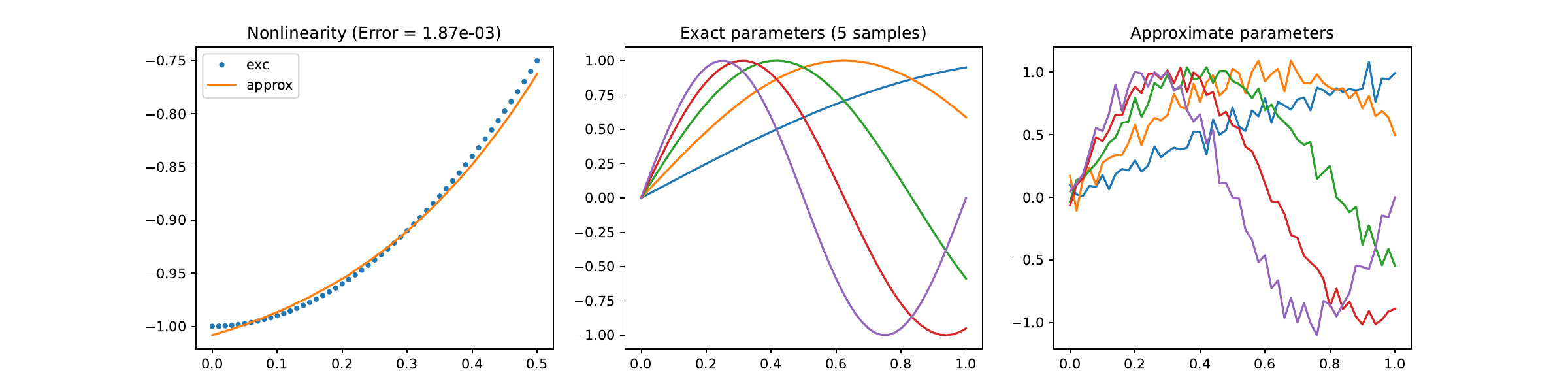}
\caption{\label{fig:quadratic_u_w_phi_multsamples}
\commB{Numerical identification of state $u$, ground-truth nonlinearity $f(u)=u^2-1$ and the parameter $\varphi $ in \eqref{eq:numerics}} for an increasing number of measurement data. Top to bottom: 1,3 and 5 measurements. Left to right: recovered nonlinearity (orange) compared to ground truth (blue), ground truth parameters, recovered parameters.}
\end{figure}

\renewcommand\fsc{0.185}
\begin{figure}[h] 
\centering
\includegraphics[width=\fsc\columnwidth, trim = 0 0.5cm 0 0.5cm, clip]{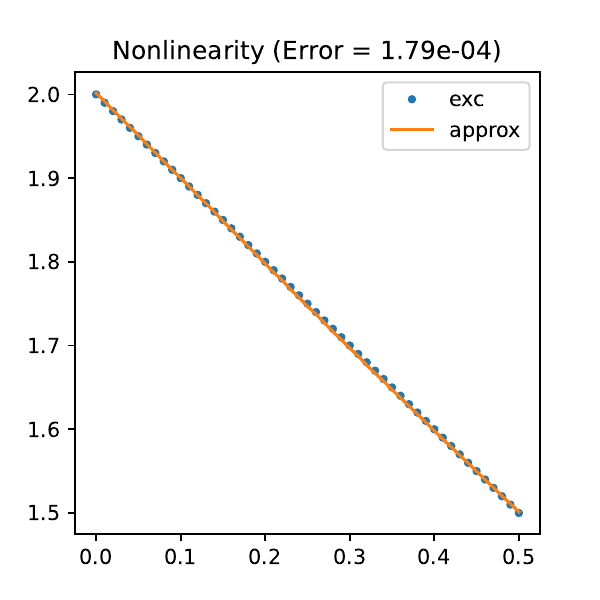}
\includegraphics[width=\fsc\columnwidth, trim = 0 0.5cm 0 0.5cm, clip]{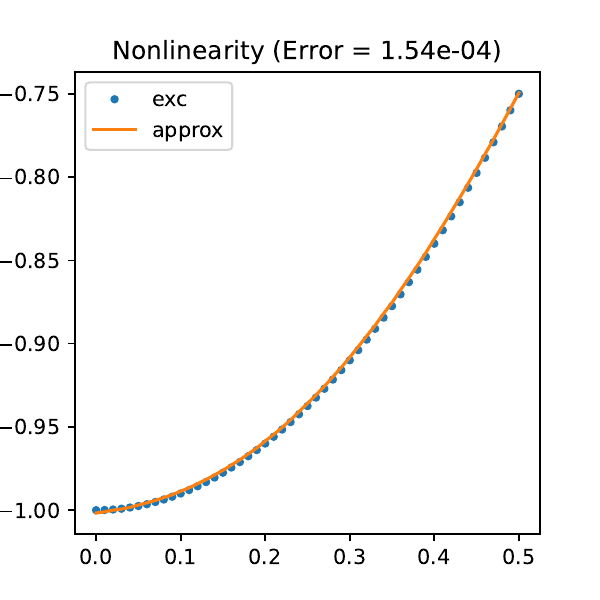}
\includegraphics[width=\fsc\columnwidth, trim = 0 0.5cm 0 0.5cm, clip]{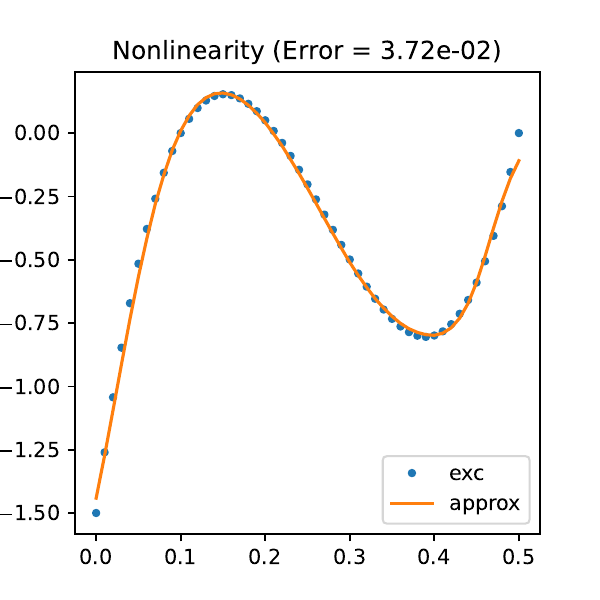}
\includegraphics[width=\fsc\columnwidth, trim = 0 0.5cm 0 0.5cm, clip]{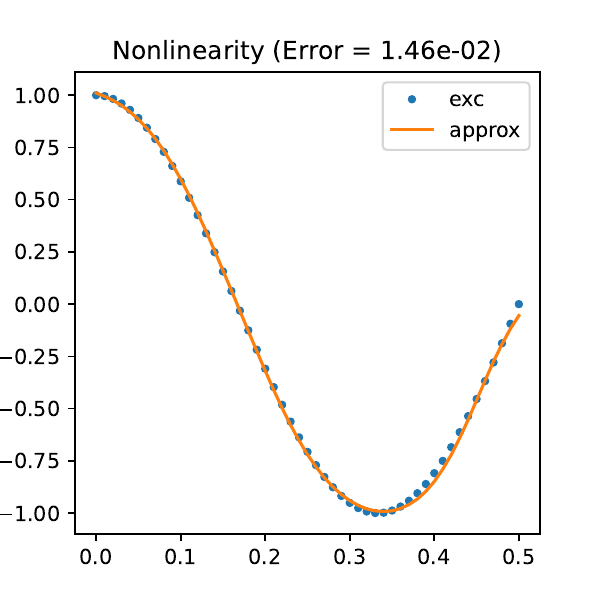}

\includegraphics[width=\fsc\columnwidth, trim = 0 0.5cm 0 0.5cm, clip]{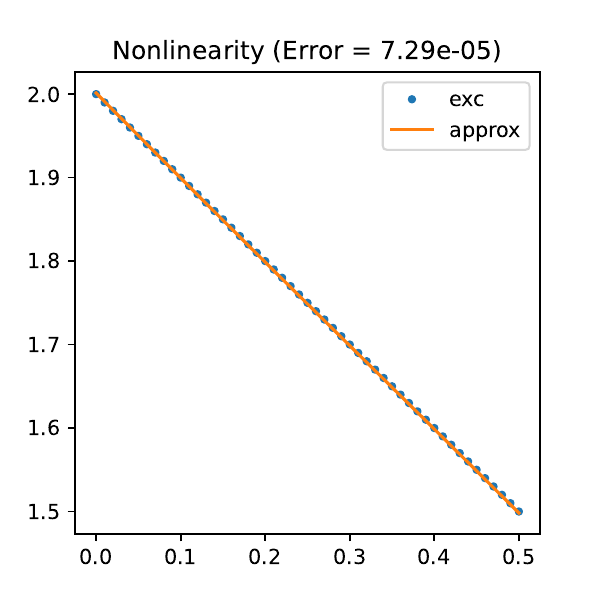}
\includegraphics[width=\fsc\columnwidth, trim = 0 0.5cm 0 0.5cm, clip]{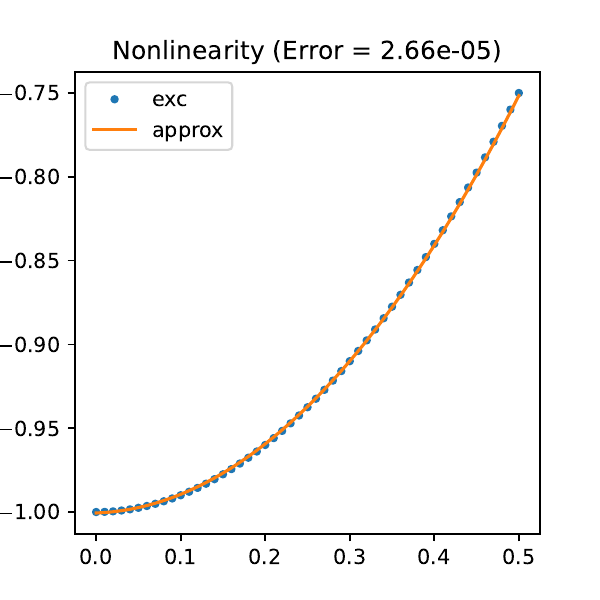}
\includegraphics[width=\fsc\columnwidth, trim = 0 0.5cm 0 0.5cm, clip]{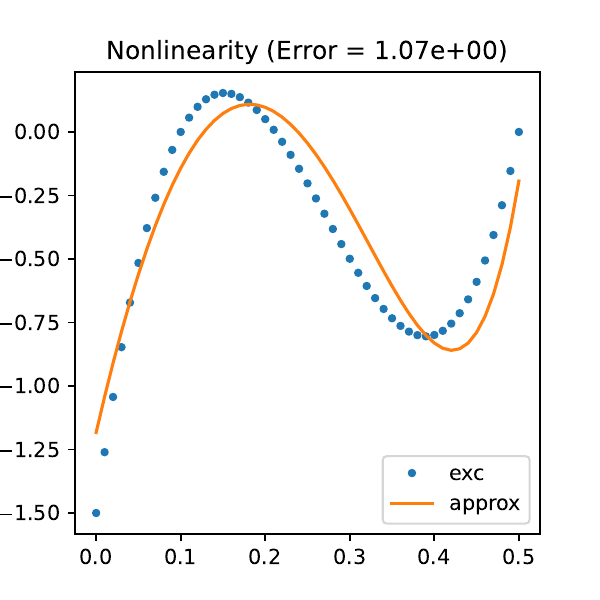}
\includegraphics[width=\fsc\columnwidth, trim = 0 0.5cm 0 0.5cm, clip]{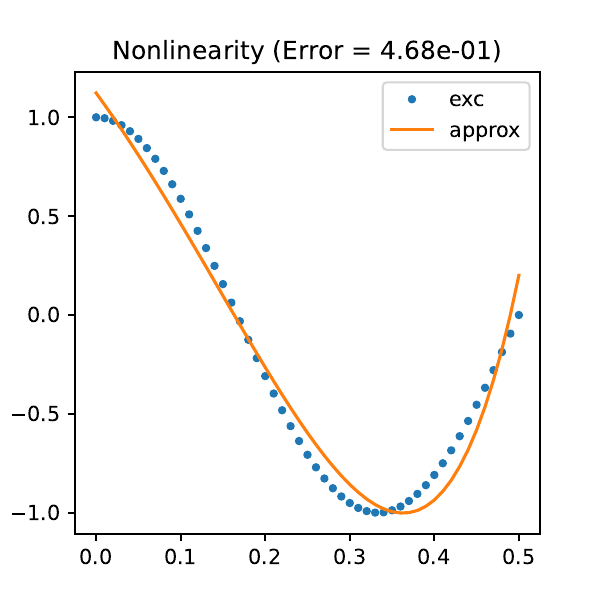}

\includegraphics[width=\fsc\columnwidth, trim = 0 0.5cm 0 0.5cm, clip]{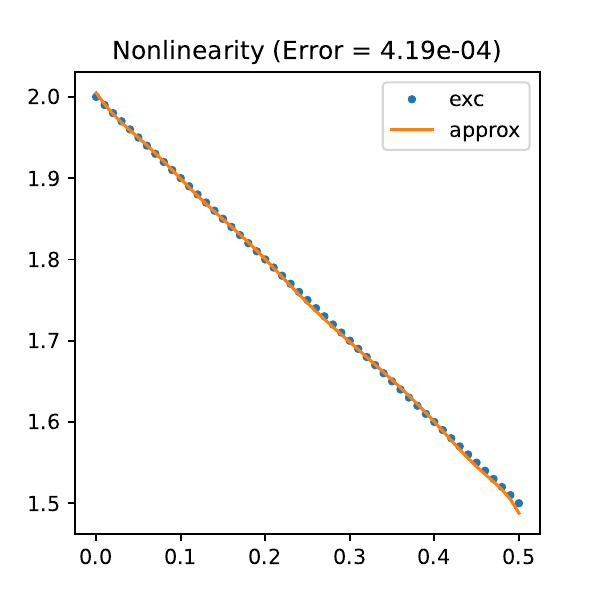}
\includegraphics[width=\fsc\columnwidth, trim = 0 0.5cm 0 0.5cm, clip]{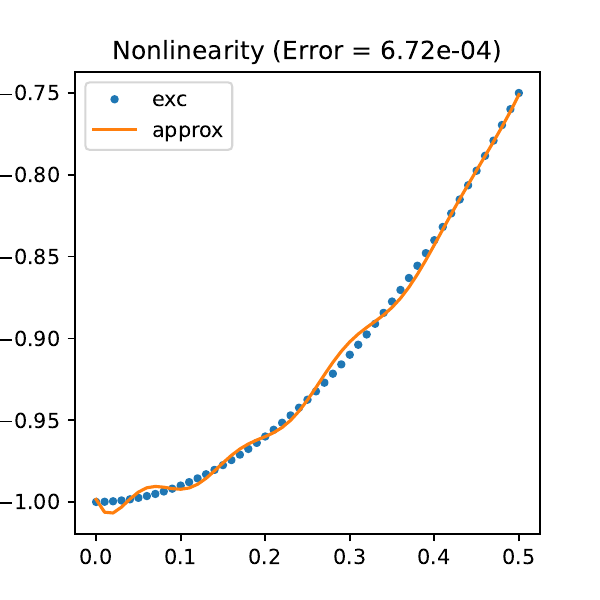}
\includegraphics[width=\fsc\columnwidth, trim = 0 0.5cm 0 0.5cm, clip]{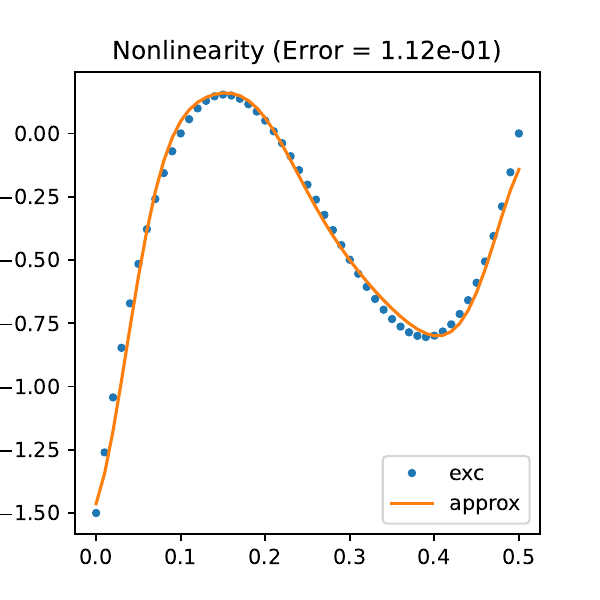}
\includegraphics[width=\fsc\columnwidth, trim = 0 0.5cm 0 0.5cm, clip]{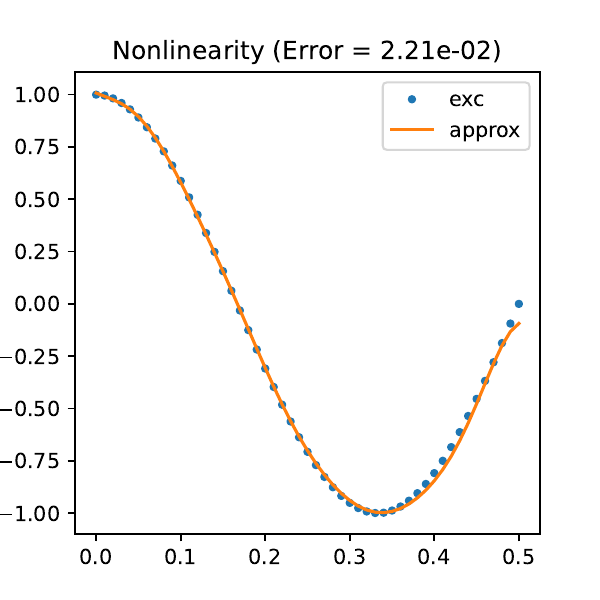}

\caption{\label{fig:approx_comparision}
Comparison of different approximation methods. From top to bottom: Neural network, polynomial, trigonometric polynomial. From left to right: \commB{ground-truth nonlinearity} $f(u) = 2 - u $, $f(u) = u^2 - 1$, $f(u) = (u-0.1)(u-0.5)(141.6u-30)$ and $f(u) = \cos(3\pi u)$.}
\end{figure}

\section{Conclusion}
We have considered the problem of learning a partially unknown PDE model from data, in a situation where access to the state is possible only indirectly via incomplete, noisy observations of a parameter-dependent system with unknown physical parameters. The unknown part of the PDE model was assumed to be a nonlinearity acting pointwise, and was approximated via a neural network. Using an all-at-once formulation, the resulting minimization problem was analyzed and well-posedness was obtained for a general setting as well a concrete application. Furthermore, a \commB{tangential cone condition} was ensured for the \commB{neural network part} of a resulting learning-informed parameter identification problem, thereby providing the basis for local uniqueness and convergence results. Finally, numerical experiments using two different types of implementation strategies have confirmed practical feasibility of the proposed approach.

\paragraph{Acknowledgments.}
The authors wish to thank both reviewers for fruitful comments leading to an improved version of the manuscript.

\section*{Appendices}\label{sec:append}
\renewcommand{\thesubsection}{\Alph{subsection}}
\newtheorem*{lem*}{Lemma}
\newtheorem*{prop*}{Proposition}
\subsection{Auxiliary results}\label{appendix-Roubicek} \commA{In this appendix, for convenience of the reader, we provide some definitions and results of \cite{Roubicek} and \cite{KalNeuSch08} that are relevant for our work.}

\commA{For $V_1$ a Banach space and $V_2$ a locally convex space, $V_1\subseteq V_2$, we define
\[W^{1,p,q}([0,T]; V_1, V_2):=\{u\in L^p([0,T];V_1) \mid \dot{u}\in L^q([0,T];V_2)\} \qquad 1\leq p,q\leq+\infty.\]
\begin{lem*}\cite[Lemma 7.3.]{Roubicek}
Let $V\subseteq H\cong H^*\subseteq V^*$, and $p'= p/(p-1)$ be the conjugate exponent to $p\in[1,+\infty]$. Then $W^{1,p,p'}([0,T]; V, V^*) \embed C([0,T];H)$ (a continuous embedding), and the following integration-by-parts formula holds for any $u, v \in W^{1,p,p'}([0,T]; V, V^*)$ and any $0\leq t_1\leq t_2\leq T$:
\[(u(t_2),v(t_2))-(u(t_1),v(t_1))=\int_{t_1}^{t_2}\langle\dot{u}(t),v(t)\rangle_{V^*,V}+\langle u(t),\dot{v}(t)\rangle_{V,V^*}\wrt t.\]
\end{lem*}
\begin{lem*}\cite[Lemma 7.7 (Aubin and Lions)]{Roubicek}
Let $V_1$, $V_2$ be Banach spaces,
$V_3$ a metrizable Hausdorff locally convex space, such that $V_1$ is separable and reflexive,
$V_1 \compt V_2$ (a compact embedding) and $V_2 \embed V_3$ (a continuous embedding), and fix $1 < p < +\infty,
1 \leq q \leq +\infty$. Then \[W^{1,p,q} ([0,T]; V_1, V_3 )\compt L^p (I; V_2).\]
\end{lem*}
\begin{prop*}\cite[Proposition 2.1, (ii)]{KalNeuSch08}
Let $\rho, \epsilon> 0$ be such that
\[\|G (x) - G (\tilde{x}) - G'(x)(x - \tilde{x})\| \leq c(x, \tilde{x})\|G (x) - G (\tilde{x})\|,\quad
x, \tilde{x} \in B_\rho(x_0) \subseteq \mathcal{D}(G)\]
for some $c(x, \tilde{x})\geq 0$, where $c(x, \tilde{x}) < 1$ if $\|x-\tilde{x}\|\leq\epsilon$.\\
If $G (x) = y$ is solvable in $B_\rho(x_0)$, then a unique $x_0$-minimum-norm solution exists. It is characterized as the solution $x^\dagger$ of $G(x) = y$ in $B_\rho(x_0)$ satisfying the condition
\[x^\dagger - x_0 \in \mathcal{N}(G'(x^\dagger))^\perp.\]
\end{prop*}
Note that in this proposition, the claim does not change if the statement is made for the ball $B_\rho(x^\dagger)$ with $x_0\in B_\rho(x^\dagger)$.
}

\subsection{Proofs}

\begin{proof}[{\bf Proof of Lemma \ref{NN-Lipschitz}}]\label{appendix-proof-NN}
Observe that for any $z$, $\tilde{z}$, $\omega$, $\tilde{\omega}$, $\beta$ and $\tilde{\beta}$, the inequalities
\begin{equation}
\begin{aligned}\label{sigma_lipschitz}
|\sigma(\omega z + \beta) - \sigma(\omega \tilde{z} + \beta)| & \leq C_\sigma |\omega||z - \tilde{z}|, \qquad &
|\sigma'(\omega z + \beta) - \sigma'(\omega \tilde{z} + \beta)| & \leq C'_\sigma |\omega||z - \tilde{z}|, \\
|\sigma(\omega z + \beta) - \sigma(\tilde{\omega} z + \beta)| & \leq C_\sigma |z||\omega - \tilde{\omega}|, \qquad &
|\sigma'(\omega z + \beta) - \sigma'(\tilde{\omega} z + \beta)| & \leq C'_\sigma |z||\omega - \tilde{\omega}|, \\
|\sigma(\omega z + \beta) - \sigma(\omega z + \tilde{\beta})| & \leq C_\sigma |\beta - \tilde{\beta}|, \qquad &
|\sigma'(\omega z + \beta) - \sigma'(\omega z + \tilde{\beta})| & \leq C'_\sigma |\beta - \tilde{\beta}|
\end{aligned}
\end{equation}
lead to straightforward computations showing that for every layer $i$, $1\leq i\leq L$, one has
\begin{align}
\left|\Nc^i_{\theta^{i}}(z) - \Nc^i_{\theta^{i}}(\tilde{z})\right| & \leq (C_\sigma)^i \left(\prod_{k=1}^i\left|\omega^k\right|\right)\left|z - \tilde{z}\right|, \label{lipschitz_nn:z} \\
\left|\Nc^i_{\theta^{i}}(z) - \Nc^i_{\bar{\theta}^{i}}(z)\right| & \leq (C_\sigma)^{i-l+1} \left(\prod_{k=l+1}^i\left|\omega^k\right|\right)\left|\Nc^{l-1}_{\theta^{l-1}}(z)\right|\left|\omega^l - \tilde{\omega}^l\right| \text{ for $i\geq l$, $0$ otherwise}, \label{lipschitz_nn:w} \\
\left|\Nc^i_{\theta^{i}}(z) - \Nc^i_{\hat{\theta}^{i}}(z)\right| & \leq (C_\sigma)^{i-l+1}\left(\prod_{k=l+1}^i\left|\omega^k\right|\right)\left|\beta^l - \tilde{\beta}^l\right| \text{ for $i\geq l$, $0$ otherwise}, \label{lipschitz_nn:b}
\end{align}
which yields \eqref{lipschitz_estimates_nn} when $i=L$. \commA{Here, one recalls that $l$ is the fixed layer with regards to which we aim to compute derivatives and associated Lipschitz estimates.}

More care must be taken regarding the Lipschitz estimates \eqref{lipschitz_estimates_nn_derivative:z} for the derivatives. Recursively writing out the chain rule, define $A_L(z,\theta):=\sigma'(\omega^L\Nc_{\theta^{L-1}}^{L-1}(z)+\beta^L)\in\R$ and $$
	A_i(z,\theta):=A_{i+1}(z,\theta)\omega^{i+1}\sigma'(\omega^i\Nc^{i-1}(z,\theta^{i-1}) + \beta^i) \in \R^{1\times n_i} \text{ for $1\leq i<L$}
$$
(understanding $\sigma'(\omega^i\Nc^{i-1}(z,\theta^{i-1}) + \beta^i)$ as a diagonal matrix in $\R^{n_i\times n_i}$), which satisfies the estimate
$\sup_{(z,\theta)\in \Bc}\left|A_i(z,\theta)\right|\leq \left(\prod_{k=i+1}^Ls_k\left|\omega^k\right|\right)s_i$. Due to the chain rule, it is not difficult to see that
\begin{align}
\begin{split}\label{chain_rule}
\Nc'_z(z,\theta) = A_1(z,\theta)\omega^1 \in \R, \quad \Nc'_{\beta^l}(z,\theta) = A_l(z,\theta) \in \Lc(\R^{n_l},\R) \\
\Nc'_{\omega^l}(z,\theta)= \big[ \commA{\R^{n_l\times n_{l-1}}\ni w} \mapsto A_l(z,\theta)w\Nc^{l-1}(z,\theta^{l-1}) \in \R \big].
\end{split}
\end{align}
The estimate \eqref{lipschitz_estimates_nn_derivative:z} will now be shown via backwards induction, with the various constants defined in \eqref{lipschitz_constant_derivative} acting as the Lipschitz constants of the $A_i$. Begin by noting
\begin{align*}
\left|A_L(z,\theta) - A_L(\tilde{z},\theta)\right| & = \left|\sigma'(\omega^L\Nc^{L-1}(z,\theta^{L-1}) + \beta^L) - \sigma'(\omega^L\Nc^{L-1}(\tilde{z},\theta^{L-1}) + \beta^L)\right| \\
& \leq C'_\sigma\left|\omega^L\right|\left|\Nc^{L-1}(z,\theta^{L-1}) - \Nc^{L-1}(\tilde{z},\theta^{L-1})\right| \\
& \leq C'_\sigma \left|\omega^L\right|(C_\sigma)^{L-1}\left(\prod_{k=1}^{L-1}\left|\omega^k\right|\right)\left|z - \tilde{z}\right| = C^z_L\left|z-\tilde{z}\right|,
\end{align*}
where the first inequality is immediate from \eqref{sigma_lipschitz} and the second follows from \eqref{lipschitz_nn:z} with $i=L-1$.

Let now $1\leq i< L$ be arbitrary. Assume $\left|A_{i+1}(z,\theta) - A_{i+1}(\tilde{z},\theta)\right| \leq C^z_{i+1}\left|z-\tilde{z}\right|$, and observe
\begin{align*}
&|A_i(z,\theta)  - A_i(\tilde{z},\theta)| \\
 = &
\left|A_{i+1}(z,\theta)\omega^{i+1}\sigma'(\omega^i\Nc^{i-1}(z,\theta^{i-1}) + \beta^i) - 
  A_{i+1}(\tilde{z},\theta)\omega^{i+1}\sigma'(\omega^i\Nc^{i-1}(\tilde{z},\theta^{i-1}) + \beta^i)\right| \\
\leq & 
\left|A_{i+1}(z,\theta)\omega^{i+1}\sigma'(\omega^i\Nc^{i-1}(z,\theta^{i-1}) + \beta^i) - 
  A_{i+1}(z,\theta)\omega^{i+1}\sigma'(\omega^i\Nc^{i-1}(\tilde{z},\theta^{i-1}) + \beta^i)\right| \\ 
  & +
\left|A_{i+1}(z,\theta)\omega^{i+1}\sigma'(\omega^i\Nc^{i-1}(\tilde{z},\theta^{i-1}) + \beta^i) - 
  A_{i+1}(\tilde{z},\theta)\omega^{i+1}\sigma'(\omega^i\Nc^{i-1}(\tilde{z},\theta^{i-1}) + \beta^i)\right| \\
 = & 
\left|A_{i+1}(z,\theta)\right|\left|\omega^{i+1}\right|\left|\sigma'(\omega^i\Nc^{i-1}(z,\theta^{i-1}) - \beta^i) - 
  \sigma'(\omega^i\Nc^{i-1}(\tilde{z},\theta^{i-1}) + \beta^i)\right| \\ 
  & +
\left|A_{i+1}(z,\theta)-A_{i+1}(\tilde{z},\theta)\right|\left|\omega^{i+1}\right|\left|\sigma'(\omega^i\Nc^{i-1}(\tilde{z},\theta^{i-1}) + \beta^i)\right|.
\end{align*}
We apply \eqref{sigma_lipschitz}, then \eqref{lipschitz_nn:z} and the bound on $A_{i+1}$ to the first line, while we apply the induction assumption together with the definition of $s_i$ to the second line to obtain
\begin{align*}
& \left|A_i(z,\theta) - A_i(\tilde{z},\theta)\right| \\
\leq & \left|A_{i+1}(z,\theta)\right|\left|\omega^{i+1}\right|\left|\sigma'(\omega^i\Nc^{i-1}(z,\theta^{i-1}) - \beta^i) - 
  \sigma'(\omega^i\Nc^{i-1}(\tilde{z},\theta^{i-1}) + \beta^i)\right| \\ 
  & +
\left|A_{i+1}(z,\theta)-A_{i+1}(\tilde{z},\theta)\right|\left|\omega^{i+1}\right|\left|\sigma'(\omega^i\Nc^{i-1}(\tilde{z},\theta^{i-1}) + \beta^i)\right| \\
\leq &
\left|A_{i+1}(z,\theta)\right|\left|\omega^{i+1}\right|C'_\sigma\left|\omega^i\right|\left|\Nc^{i-1}(z,\theta^{i-1}) - \Nc^{i-1}(\tilde{z},\theta^{i-1})\right| +
C^z_{i+1}\left|z-\tilde{z}\right|\left|\omega^{i+1}\right|s_i \\
\leq &
\left[\left(\prod_{k=i+2}^Ls_k\left|\omega^k\right|\right)s_{i+1} \left|\omega^{i+1}\right|C'_\sigma\left|\omega^i\right|(C_\sigma)^{i-1} \left(\prod_{k=1}^{i-1}\left|\omega^k\right|\right)
+
C^z_{i+1}s_i\left|\omega^{i+1}\right|\right]\left|z-\tilde{z}\right| = C^z_i\left|z-\tilde{z}\right|.
\end{align*}
\eqref{lipschitz_estimates_nn_derivative:z} now follows immediately from  \eqref{chain_rule} and the fact that $|N'_z(z)|=1$, since this is a matrix with a single entry $1$ and otherwise consisting of zeros.

Completely analogous computations, employing \eqref{lipschitz_nn:w} and \eqref{lipschitz_nn:b}, respectively, in place of \eqref{lipschitz_nn:z}, similarly yield \eqref{lipschitz_estimates_nn_derivative:w} and \eqref{lipschitz_estimates_nn_derivative:b}, concluding the proof.

\end{proof}

\bigskip
\begin{proof}[\bf Proof of Proposition \ref{prop-adjoints-cont}, iii)]\label{appendix-proof-adjoint}
On $\Vc=H^1(0,T;V)\embed C(0,T;V)$, we impose the norm $\|\cdot \|_\Vc$ via the inner product
\[(u,v)_\Vc=\int_0^T (\dot{u}(t),\dot{v}(t))_V\,dt + (u(0),v(0))_V, \]
since it induces an equivalent norm to the standard norm $\|u\|_{H^1(0,T;V)}=\sqrt{\int_0^T \|\dot{u}(t)\|^2_V+\|u(t)\|^2_V\,dt}$. Indeed, from the estimates (c.f. \cite[Lemma 7.1]{Roubicek})%\todo{continuous in time implies absolutely continuous in time? (cf the assumptions for the fundamental theorem mentioned in the proof of Prop. \ref{prop-Differentiability})} \commA{added ref}
\begin{align*}
&\|u(t)\|_V\leq \|u(0)\|_V+\int_0^T\|\dot{u}(t)\|_V\,dt \leq \max\{\sqrt{2},\sqrt{2T}\}\|u\|_\Vc \,\,\Rightarrow\,\, \|u\|_\ltv\leq  \sqrt{2}\max\{\sqrt{T},T\}\|u\|_\Vc,
\end{align*}
such that $\|u\|_{H^1(0,T;V)} \leq  c\|u\|_\Vc$ for $c>0$, and
\begin{align*}
&\|u(0)\|_V\leq \|u(t_0)\|_V+\int_0^{t_0}\|\dot{u}(t)\|_V\,dt\leq  \int_0^T \frac{\|u(t)\|_V}{T}+\|\dot{u}(t)\|_V\,dt\leq \sqrt{2}\max\{\frac{1}{\sqrt{T}},\sqrt{T}\}\|u\|_{H^1(0,T;V)}.
\end{align*}
for some $t_0\in(0,T)$ such that $\|u\|_\Vc \leq C \|u\|_{H^1(0,T;V)} $ for $C>0$. Here, we have used $\|u\|_V=\|u\|_\hto:=\sqrt{\|\Delta u\|^2_\ltn+\|\nabla u\|^2_\ltn}$, which is an equivalent norm on $\hto$ as a consequence of the Poincar\'e-Friedrichs inequality.

At first, we carry out some general computations. First note that $L^2(\Omega)\ni k^z\mapsto \ztil\in H^2(\Omega)\cap H^1_0(\Omega)$ is well-defined due to unique existence of the solution to the linear auxiliary problems \eqref{auxiliaryPDEs}. Thus, with $k^z(t) \in L^2(\Omega)$ and $\ztil(t)\in H^2(\Omega)$ as in \eqref{auxiliaryPDEs}, for any $v(t) \in H^2(\Omega)$, we can write the identity 
\begin{align*}
(v,\ztil)_\htn&=\int_\Omega\Delta v\Delta\ztil + \nabla v\cdot\nabla\ztil\wrt x=\int_\Omega  \nabla v\cdot\nabla z_1+ vz_1\wrt x =\int_\Omega v(-\Delta z_1+z_1)\wrt x=(v,k^z)_\ltn.
\end{align*}
Given $\ztil\in L^2(0,T;H^2(\Omega))$, let $u^z \in H^1(0,T;H^2(\Omega))=\Vc$ be the solution of the ordinary equation
\begin{equation}\label{auxiliaryODE}
\begin{split}
&\uzddot(t)=-\ztil(t) \qquad t\in(0,T)\\
&\uzdot(T)=0, \quad \uzdot(0)-u^z(0)=0,
\end{split}
\end{equation}
Now let $K:\Vc \rightarrow L^2(0,T;L^2(\Omega))$ be any bounded, linear operator.
%\note{, and let $\Ktil$ be the $\ltlt$-adjoint of $K:\Vc\to \ltlt$ (meaning that $\Ktil:\ltlt\to \ltlt$) \todo{how is $\Ktil$ obtained from $K$ precisely}.}
For $z\in\ltlt, v\in \Vc$, let $k^z \in \ltlt$ be such that \[(z,Kv)_\ltlt = (k^z,v)_\ltlt \quad\text{then define}\quad \widetilde{K}z:=k^z.\] 
%Then, with $\ztil(t), z_1(t) $ solutions of \eqref{auxiliaryPDEs} given $k^z(t)$ at every $t \in (0,T)$, and $u^z$ a solution of \eqref{auxiliaryODE} given $\ztil$, we compute
Then
\begin{align*}
&(z,Kv)_\ltlt=:\int_0^T(k^z(t),v(t))_\ltn\,dt=\int_0^T(\ztil(t),v(t))_\htn\,dt=:\int_0^T(-\uzddot(t),v(t))_\htn\,dt\\
&\quad=\int_0^T(\uzdot(t),\dot{v}(t))_\htn\,dt+(u^z(0),v(0))_\htn - (\uzdot(T),v(T))_\htn+(\uzdot(0)-u^z(0),v(0))_\htn\\
&\quad=(u^z,v)_\Vc.
\end{align*}
Using the fact that $u^z\in \Vc$ in \eqref{auxiliaryODE} can be computed analytically, we obtain $K^*:\ltlt\to\Vc$ via
\begin{align*}
K^*z&=u^z=\int_0^T(t+1)\ztil(t)\,dt-\int_0^t(t-s)\ztil(s)\,ds\\
&=\int_0^T(t+1)\Dinv\Ktil z(t)\,dt-\int_0^t(t-s)\Dinv\Ktil z(s)\,ds.
\end{align*}
With this derivation, $g_{2,3}=M^*z$ with $M=\text{Id}:\Vc\to\Yc=\ltlt$ can be obtained by setting $K = M$, thus $\Ktil=\text{Id}$, yielding the adjoint as in \eqref{adjoint-M}.
%\note{With this derivation, the claimed form of $g_{2,3}$ as in \eqref{adjoint-M} is obtained immediately by setting $K=M$ and noting that, in this case, $\tilde{K}z = z$.}

%\todo{shorter: defining $K$, argue form of $\tilde{K}$. Then $\left(z,Kv \right)_\Wc = (u^z,v)_\Vc$ is immediate from the above, only need to compute with $\frac{d}{dt}$. For this, again define involved quantities first and then compute. Here: discuss time derivative of $\zttil$}
We then compute $g_{2,1}$. For $\frac{d}{dt}+K:=\frac{d}{dt}-F'_u(\lambda,u)-\Nc_u'(u,\theta):\Vc\to\Wc=\ltlt$, one has, for $z\in\Wc, v\in\Vc,$
%{\allowdisplaybreaks
\begin{align*}
&\left(z,\left(\frac{d}{dt}+K\right)v\right)_\Wc=\int_0^T(z(t),Kv(t))_\ltn\,dt+\int_0^T(z(t),\dot{v}(t))_\ltn\,dt\\
&=\int_0^T(\Ktil z(t),v(t))_\ltn\,dt + \int_0^T\left(\frac{d}{dt}\left(\int_0^t z(s)\,ds\right),\dot{v}(t)\right)_\ltn\,dt\\
&=:\int_0^T(k^z(t),v(t))_\ltn\,dt + \int_0^T(\dot{h^z}(t),\dot{v}(t))_\ltn\,dt\\
&=\int_0^T(\ztil(t),v(t))_\htn\,dt + \int_0^T(\frac{d}{dt}\zttil(t),\dot{v}(t))_\htn\,dt=:\int_0^T(-\uzddot(t),v(t))_\htn+(\frac{d}{dt}\zttil(t),\dot{v}(t))_\htn\,dt\\
&=\int_0^T(\uzdot(t),\dot{v}(t))_\htn\,dt+(u^z(0),v(0))_\htn + \int_0^T(\frac{d}{dt}\zttil(t),\dot{v}(t))_\htn\,dt +  (\zttil(0),v(0))_\htn\\
&=(u^z+\zttil,v)_\Vc,
\end{align*}%}
where $k^z,\ztil$ are the same as before, and $\zttil$ solves \eqref{auxiliaryPDEs} with $h^z:=\int_0^t z(s)\,ds$ in place of $k^z$. Above, we notice that $\zttil(0)=0$ since $h^z(0)=0$ and unique existence result of linear PDEs in \eqref{auxiliaryPDEs}. $u^z$ is still, as defined earlier, the solution to \eqref{auxiliaryODE}.
For $K=-F'_u(\lambda,u)-\Nc_u'(u,\theta)$, we deduce 
\[\Ktil z=-\nabla\cdot(a\nabla z)+c-\Nc_u'(u,\theta)z\] 
yielding $g_{2,1}$ as in \eqref{adjoint-FuNu}.

The next adjoint $g_{1,1}=-F'_\lambda(\lambda,u)^*$ is computed as follows. For $z\in\Wc, \xi\in X$,
\begin{align*}
(z,K\xi)_\Wc&=\int_0^T (z(t),K\xi)_\ltn\,dt
%= \int_0^T\left(\Ktil z(t),\xi \right)_\ltn\,dt
=
\begin{cases}
\left(\int_0^T\Ktil z(t)\,dt,\xi \right)_\ltn \quad\text{if } \lambda=c \text{ or } \lambda=\varphi\\[1.5ex]
\left(\int_0^T\ztil(t)\,dt,\xi \right)_\htn \qquad \text{if } \lambda=a
\end{cases}=(K^*z,\xi)_X,
\end{align*}
where $\Ktil$ is the $\ltn$-adjoint of $K=-F'_\lambda(\lambda,u)$; and $\ztil$ solves \eqref{auxiliaryPDEs} for $k^z:=\Ktil z$. \eqref{adjoint-Flambda} follows by
\begin{alignat*}{2}
&\lambda=c:\quad &&\Ktil z= zu,\qquad\qquad \lambda=\varphi: \quad \Ktil z=-z,\\
&\lambda=a: &&\ztil= \Dinv(-\nabla\cdot(z\nabla u)).
\end{alignat*}

The adjoint for $g_{2,2}=(\cdot)^*_{t=0}$ can be derived in a similar manner.

%For $g_{2,2}=(\cdot)^*_{t=0}$, let  $h\in U_0=\htn, v\in\Vc$ then \eqref{adjoint-t0} is shown via
%\begin{align*}
%(h,v(0))_\htn&=\int_0^T(\dot{h},\dot{v}(t))_\htn\,dt+(h,v(0))_\htn=(h,v)_\Vc.
%\end{align*}

We now compute the last adjoint $g_{3,1}=-\Nc_\theta'(u,\theta)^*$ involving the neural network with weights $\weight$, biases $\bias$ and the fixed activation $\sigma$. With the architecture mentioned at the beginning of this section, we define by $a_l$ the output of the l-th layer 
\begin{align*}
a_l=\sigma(\weight_l  a_{l-1}+\bias_l), \quad a_0=\text{input data } u\qquad l=1\ldots L,
\end{align*}
and introduce
\begin{align*}
a_l'=\sigma'(\weight_l  a_{l-1}+\bias_l), \quad a'_0=\text{input data }u \qquad l=1\ldots L,
\end{align*}
with $\sigma=\text{Id}$ in the L-th (output) layer, and $\sigma'$ is the derivative of $\sigma$.\\
In each layer, one searches for the unknown $\theta_l=(\weight_l,\bias_l)\in \R^{n_l\times n_{l-1}}\times\R^{n_l}$. For any $Q\in\R^{n_l\times n_{l-1}}, z\in\Wc$,
\begin{align*}
&\left(\nabla_{\weight_l}\Nc(u,\theta) Q,z \right)_\Wc=\int_0^T\int_\Omega \weight_L  a'_{L-1} \ldots  \weight_{l+1}  a'_l   Q  a_{l-1}\,z \wrt x\,dt\\
&=Q \cdot\int_0^T\int_\Omega (a_{l-1}  \weight_L  a'_{L-1} \ldots  \weight_{l+1}  a'_l)^T z \wrt x\,dt =Q \cdot\int_0^T\int_\Omega  {a'}_l^T  \weight_{l+1}^T \ldots {a'}_{L-1}^T  \weight_L^T  a_{l-1}^T \,z \wrt x\,dt\\
&=:Q \cdot\int_0^T\int_\Omega \delta_l  a_{l-1}^T \,z \wrt x\,dt\quad=: Q \cdot K^*_l z,
\end{align*}
where $K^*_l$ is indeed the desirable adjoint $\nabla_{\weight_l}\Nc(u,\theta)^*$ in layer l-th. With the use of $\delta_l$, one can perform a recursive routine for computing the adjoints in all layers, starting from the last layer
\begin{align*}
&\delta_L = 1, \qquad \delta_{l-1}= {a'}^T_{l-1} \weight^T_l \delta_l,  \quad\qquad l=L\ldots 2\\
&\nabla_{\weight_{l-1}}\Nc(u,\theta)^*z= \int_0^T\int_\Omega \delta_{l-1}  a_{l-2}^T \,z \wrt x\,dt.
\end{align*}
A similar derivation yields $\nabla_{\bias_{l-1}}\Nc(u,\theta)^*z$, completing \eqref{adjoint-Ntheta}.
%In the same vein, for any $q\in\R^{n_l}, z\in\Wc$,
%\begin{align*}
%\left(\nabla_{b_l}\Nc(u,\theta) q,z \right)_\Wc&=\int_0^T\int_\Omega \weight_L\cdot a'_{L-1}\cdot\ldots\cdot \weight_{l+1}\cdot a'_l \cdot q\,z \wrt x\,dt\\
%&=q:\int_0^T\int_\Omega (\weight_L\cdot a'_{L-1}\cdot\ldots\cdot \weight_{l+1}\cdot a'_l)^T \,z \wrt x\,dt=q:\int_0^T\int_\Omega \delta_l \,z \wrt x\,dt\quad=: q:H^*_l z
%\end{align*}
%implies $H^*_l=\nabla_{b_l}\Nc(u,\theta)^*$ with a recursive computation scheme
%\begin{align*}
%&\delta_L = 1, \qquad \delta_{l-1}= {a'}^T_{l-1}\cdot\weight^T_l\cdot\delta_l,  \quad\qquad l=L\ldots 2\\
%&\nabla_{\bias_{l-1}}\Nc(u,\theta)^*z= \int_0^T\int_\Omega \delta_{l-1}\,z \wrt x\,dt.
%\end{align*}

\end{proof}

%\paragraph*{Acknowledgments.}
\bibliography{mh_lit_dat}
\bibliographystyle{abbrv}

%\commA{Thoughts:
%\begin{itemize}
%\item add short "Algorithm". e.g at the end of Sec 1.1? as many ML papers have this.
%\item stress more on key contributions in section Introduction (before moving to Sec 1.1?)
%\item move some to Appendix. e.g Step 1-2 (thus Theo. 21) for Prop 23, proof of Prop. 30? 
%\item comment out some intermediate estimates,... to save space.
%\item skip, e.g. Remark 27,... to save space.
%\end{itemize}}

\end{document}